\documentclass[12pt]{amsart}
\usepackage{epsfig,comment}
\usepackage{mathtools}
\usepackage[subnum]{cases}
\usepackage{enumerate}
\usepackage{bbm,bm}
\usepackage{amsmath,amssymb,mathrsfs,amsthm}
\usepackage{color}
\usepackage[bookmarks=true,
bookmarksnumbered=true, breaklinks=true,
pdfstartview=FitH, hyperfigures=false,
plainpages=false, naturalnames=true,
colorlinks=true,pagebackref=true,
pdfpagelabels,colorlinks=true,
	linkcolor=blue,
	citecolor=black,
	filecolor=black,
	urlcolor=black]{hyperref}
\usepackage{wasysym}
\usepackage[hang]{footmisc} %flushes footnotes. to flush left use option ``[flushmargin]"
\usepackage{verbatim}

\usepackage{pgfplots}
\pgfplotsset{width=10cm,compat=1.9}
\usepgfplotslibrary{fillbetween}

\usepgfplotslibrary{external}
\newcommand{\N}{\mathbb{N}}
\newcommand{\R}{\mathbb{R}}
\newcommand{\Sp}{\mathbb{S}^{n-1}}

\DeclareMathOperator{\vol}{vol}
\DeclareMathOperator{\hyp}{hyp}
\DeclareMathOperator{\epi}{epi}
\DeclareMathOperator{\Gr}{Gr}
\DeclareMathOperator{\LC}{LC}
\DeclareMathOperator{\conv}{conv}

\DeclareMathOperator{\proj}{proj}
\DeclareMathOperator{\dom}{dom}
\DeclareMathOperator{\lev}{lev}
\DeclareMathOperator{\supp}{supp}
\DeclareMathOperator{\base}{base}
\DeclareMathOperator*{\hypolim}{hypo-lim}
\DeclareMathOperator*{\epilim}{epi-lim}

\usepackage{scalerel}

\newtheorem {theorem}{Theorem}[section]
\newtheorem {proposition}[theorem]{Proposition}
\newtheorem {lemma}[theorem]{Lemma}
\newtheorem {corollary}[theorem]{Corollary}

\theoremstyle{definition}
\newtheorem {definition}{Definition}[section]
\newtheorem {remark}[theorem]{Remark}

\newcommand*\sq{\mathbin{\vcenter{\hbox{\rule{.4ex}{.4ex}}}}}

\DeclareFontFamily{U} {MnSymbolC}{}

\DeclareFontShape{U}{MnSymbolC}{m}{n}{
  <-6> MnSymbolC5
  <6-7> MnSymbolC6
  <7-8> MnSymbolC7
  <8-9> MnSymbolC8
  <9-10> MnSymbolC9
  <10-12> MnSymbolC10
  <12-> MnSymbolC12}{}
\DeclareFontShape{U}{MnSymbolC}{b}{n}{
  <-6> MnSymbolC-Bold5
  <6-7> MnSymbolC-Bold6
  <7-8> MnSymbolC-Bold7
  <8-9> MnSymbolC-Bold8
  <9-10> MnSymbolC-Bold9
  <10-12> MnSymbolC-Bold10
  <12-> MnSymbolC-Bold12}{}

\DeclareSymbolFont{MnSyC} {U} {MnSymbolC}{m}{n}

\DeclareMathSymbol{\bigsquare}{\mathbin}{MnSyC}{107}
\DeclareMathSymbol{\bigstar}{\mathbin}{MnSyC}{205}

\usepackage{fonttable}

\title[On Minkowski symmetrizations of $\alpha$-concave functions]{On Minkowski  symmetrizations of $\alpha$-concave functions and related applications}

\author{Steven Hoehner}
\date{\today}

\begin{document}

%\fonttable{MnSymbolC10}

\setcounter{footnote}{0}

\begin{abstract}\noindent
The Minkowski symmetral of an $\alpha$-concave function is studied, and some of its fundamental properties are derived. It is shown that for a given $\alpha$-concave function, there exists a sequence of Minkowski symmetrizations that hypo-converges to its  ``reflectional hypo-symmetrization''. As an application, it is shown that the reflectional hypo-symmetrization of a log-concave function $f$ is always harder  to approximate than $f$ is by ``inner log-linearizations'' with a fixed number of break points. This is a functional analogue of the classical geometric result which states that among all convex bodies of a given mean width, a Euclidean ball is hardest to approximate by inscribed polytopes with a fixed number of vertices. Finally, a general extremal property of the reflectional hypo-symmetrization is deduced, which includes a Urysohn-type inequality and the aforementioned approximation result as special cases.
\end{abstract}

\maketitle

\renewcommand{\thefootnote}{}

\renewcommand{\thefootnote}{\arabic{footnote}}
\setcounter{footnote}{0}

\tableofcontents

%%%%%%%%%%%%%%%%%%%%%%%%
\section{Introduction}
%{\color{red}Colesanti, Rotem, Roysdon}

Symmetrizations are among the most powerful tools in  geometry and its applications. They are used to ``round off'' a convex body into a ball after multiple applications while preserving some geometric quantity, such as volume, surface area or mean width. For example, the Steiner symmetrization, which is chief among the various symmetrizations in convex geometry, is defined as follows. Given a convex body $K$ in $\R^n$ and a direction $u$ in the unit sphere $\Sp$ in $\R^n$, the \emph{Steiner symmetral} $S_u K$ of $K$ about the hyperplane $u^\perp$ is the convex body obtained by translating, in the direction $u$, each chord of $K$ parallel to $u$  so that it is bisected by $u^\perp$, and then taking the union of all of the translated chords. Note that by Cavalieri's principle, the Steiner symmetral $S_u K$ has the same  volume as $K$. Moreover, for any convex body $K$ in $\R^n$, there always exists a sequence of successive Steiner symmetrizations which transforms $K$ into a Euclidean ball with the same volume as $K$. This ball  provides the solution to some of the most important extremal problems in classical convex geometry, including the isoperimetric inequality, the Brunn--Minkowski inequality and the Blaschke--Santal\'o inequality. 

Another important symmetrization in convex geometry is the Minkowski symmetrization.  Given a direction $u\in\Sp$, the \emph{Minkowski symmetral} $\tau_u K$ of a convex body $K$ in $\R^n$ about the hyperplane $H=u^\perp$ is defined by 
\begin{equation}\label{minkowski-symmetral-convex-bodies}
\tau_u K=\frac{1}{2}K+\frac{1}{2}R_u K. 
\end{equation}
Here $K+L=\{x+y: x\in K, y\in L\}$ is the Minkowski sum of convex bodies $K$ and $L$, and for $\lambda>0$, the $\lambda$-homothety of $K$ is $\lambda K=\{\lambda x: x\in K\}$. The reflection $R_u K$ of $K$ about $H$ is $R_u K=\{R_u(x): x\in K\}$, where $R_u(x)=x-2\langle x,u\rangle u$ is the reflection of $x\in\R^n$ about $H$. 

The Minkowski symmetrization  satisfies a number of useful geometric properties. It preserves mean width, is monotonic with respect to set inclusion, symmetric with respect to the hyperplane $H$, and invariant on sets that are symmetric with respect to $H$. Furthermore, there exists a sequence of Minkowski symmetrizations of $K$ which converges in the Hausdorff metric to a Euclidean ball with the same mean width as $K$. Naturally, this ball delivers the solution to a number of extremal problems involving the mean width, including Urysohn's inequality.   For more background on symmetrizations in convex geometry, we refer the reader to, for example, \cite{symm-in-geom,GruberBook,SchneiderBook}.

Let us now turn our attention to the topic of this paper, functional symmetrizations. Let $f:\R^n\to[0,\infty)$ be a measurable function with $\vol_n(\lev_{\geq t}f)<\infty$ for every $t>0$, where $\lev_{\geq t}f=\{x\in\R^n: \, f(x)\geq t\}$ and $\vol_n(\cdot)$ is the $n$-dimensional volume (Lebesgue measure). Given $u\in\Sp$, the \emph{Steiner symmetral} $S_u f$ of $f$ about $H=u^\perp$ is defined by
\begin{equation}\label{steiner-symm-def-fns}
S_u f(x) = \int_0^\infty \mathbbm{1}_{S_u(\lev_{\geq t}f)}(x)\,dt
=\sup\{t>0:\, x\in S_u(\lev_{\geq t}f)\},
\end{equation}
where $\mathbbm{1}_A$ is the characteristic function of the set $A\subset\R^n$. It turns out that there is a sequence of Steiner symmetrizations in different directions which ``rounds off'' the function  after multiple applications (in the sense of rotation invariance about the dependent variable axis). The limiting function is  the symmetric decreasing volume rearrangement $f_{\vol}^*$, which may be defined by the relation $\lev_{\geq t}f_{\vol}^*=(\lev_{\geq t}f)_{\vol}^*$, where $(\lev_{\geq t}f)_{\vol}^*$ is the Euclidean ball centered at the origin with the same volume as $\lev_{\geq t}f$. This rearrangement delivers the solution to a number of extremal problems in the calculus of variations and PDEs. Some examples of functional Steiner symmetrizations and their applications can be found in the works \cite{AKM2004,Burchard2009,CGN-2018,kawohl,Lin-JFA-2017,Lin-2017,Lin-Leng,Volcic}, for example.

A natural question that arises is to consider Minkowski symmetrizations of functions. In view of the Steiner symmetral \eqref{steiner-symm-def-fns}, one can define a Minkowski symmetral of a function in a similar way. Let $f:\R^n\to[0,\infty)$ be a measurable function with $w(\lev_{\geq t}f)<\infty$ for every $t>0$, where $w(\cdot)$ is the classical mean width functional (see Definition \ref{alpha-mw-def} below). Given $u\in\Sp$, the \emph{Minkowski symmetral} $\widetilde{\tau}_u f$ of $f$ about $H=u^\perp$ may be defined as
\begin{equation}\label{minkowski-symm-def-fns}
\widetilde{\tau}_u f(x) := \int_0^\infty \mathbbm{1}_{\tau_u(\lev_{\geq t}f)}(x)\,dt.
\end{equation}

In this article, we study a reflection-based Minkowski symmetral for $\alpha$-concave functions, of which  \eqref{minkowski-symm-def-fns} is shown to be a special limiting case. This extends the classical definition of the Minkowski symmetral of a convex body to the functional setting, while retaining the analogous fundamental properties of its geometric counterpart. This construction is related to earlier mean width rearrangements in the PDE literature, notably the mean width rearrangement introduced by Salani \cite{Salani-2015} via $p$-convolutions of rotated copies of a solution and a limiting rotation-mean procedure (see also \cite{Tso-1981}). We focus instead on a reflection-based version called the Minkowski symmetral, and  investigate  its structural properties from the point of view of support functions, mean width, and hypo-convergence. This work represents another contribution to the (now vast) ``geometrization of probability'' program, a term coined by V. Milman, which has been  studied extensively during the past couple of decades in convex geometry and functional analysis. For some examples of the works most closely related to this paper, we refer the reader to  \cite{AKM2004,BCF-2014,CLM-Hadwiger1,CLM-Hadwiger2,CGN-2018,Hoehner-Chasioti,Hoehner-Mussnig,Hofstatter-Schuster,Milman-Rotem-alpha,Milman-Rotem,Rotem2012,Rotem2013,Roysdon-Xing-Lp-functions}.%, for example. 

Let us briefly explain the main definition of this paper, which is the Minkowski symmetral of an $\alpha$-concave function. First, recall that for fixed $\alpha\in[-\infty,+\infty]$, a function $f:\R^n\to[0,\infty)$ is \emph{$\alpha$-concave} if $f$ has convex support $\supp(f)=\overline{\{x\in\R^n: f(x)>0\}}$, and for every $x,y\in\supp(f)$ and every $\lambda\in[0,1]$ it holds that
\begin{equation}
f(\lambda x+(1-\lambda)y) \geq M_\alpha^{(\lambda,1-\lambda)}(f(x),f(y)).
\end{equation}
Here and throughout the paper, for $\alpha\in[-\infty,+\infty]$ and $s,t,u,v>0$, the \emph{$\alpha$-means} $M_\alpha^{(s,t)}(u,v)$ are defined by
\begin{equation}
M_\alpha^{(s,t)}(u,v)=\begin{cases}
(su^\alpha+tv^\alpha)^{1/\alpha}, &\text{if }\alpha\neq 0\\
u^s v^t, &
\text{if }\alpha=0\\
\min\{u,v\}, &\text{if }\alpha=-\infty\\
\max\{u,v\}, &\text{if }\alpha=+\infty.
\end{cases}
\end{equation}
\noindent The cases $\alpha=-\infty,0,+\infty$ are each understood in a limiting sense.

On the family $\mathcal{C}_\alpha(\R^n)$  of $\alpha$-concave functions, one may define an addition operation $\star_\alpha$ and a scalar multiplication operation $\cdot_\alpha$, which are explained in  Subsection \ref{alpha-concave-defns}. In Section \ref{Minkowski-symm-section}, for $\alpha\leq 0$ we define the \emph{$\alpha$-Minkowski symmetral} $\tau_u^\alpha f$ of $f\in\mathcal{C}_\alpha(\R^n)$ by
\begin{equation}\label{alpha-symmetral}
\tau_u^\alpha f := \frac{1}{2}\cdot_\alpha f\star_\alpha\frac{1}{2}\cdot_\alpha R_u f,
\end{equation}
where $R_u f$ is the reflection of $f$ about the hyperplane $H$. We also show in Section \ref{Minkowski-symm-section} that the $\alpha$-Minkowski symmetral has  properties analogous to those of the classical Minkowski symmetral \eqref{minkowski-symmetral-convex-bodies}. In particular, it preserves the $\alpha$-mean width of $f$, an extension of the classical mean width to the realm of $\alpha$-concave functions introduced by Rotem \cite{Rotem2013}. It is also monotone, symmetric about the hyperplane $H$, and invariant on functions which are symmetric with respect to $H$. In the case $\alpha=-\infty$, it will be shown that the definitions \eqref{minkowski-symm-def-fns} and \eqref{alpha-symmetral} coincide.

Naturally, the question arises to study the convergence properties of successive Minkowski symmetrizations of functions. In Section \ref{section-successive-symmetrizations}, we show that there exists a sequence of Minkowski symmetrizations that hypo-converges to a radial function which we call its ``reflectional hypo-symmetrization''. This function is related to  the mean width rearrangement $u_p^*$ introduced by Salani \cite{Salani-2015} in the PDE setting, which is obtained from $p$-convolutions of rotated copies of a solution and a limiting rotation-mean procedure (see also \cite{Tso-1981}).

Next, in Section \ref{polytopes-section}, we present an application. First, we connect the approximation of convex bodies by inscribed polytopes to the approximation of log-concave functions by ``inner log-linearizations''. These log-affine minorants are defined in terms of  inner linearizations of the corresponding convex base function and extend the classical notion of a polytope inscribed in a convex body to the functional setting. In the main result of the section, we prove that the reflectional   hypo-symmetrization of $f$ is always harder to approximate than $f$ is by inner log-linearizations with a given number of ``break points''. This result is a functional extension of the classical geometric fact stating that among all convex bodies in $\R^n$ of a given mean width, a Euclidean unit ball is hardest to approximate by inscribed polytopes with a restricted number of vertices. In Subsection \ref{extremal-section}, we prove a general extremal property of the reflectional hypo-symmetrization of an $\alpha$-concave function. Finally, as an application, we derive a  functional version of Urysohn's inequality in Subsection \ref{sec:urysohn}.

The approximation result in Section \ref{polytopes-section} is dual in spirit to the general outer linearization theory recently developed in \cite{Hoehner-Mussnig}. There, one approximates coercive convex functions from below by supporting affine data, or equivalently, by slope data in the domain of the Legendre--Fenchel transform. In the present paper, by contrast, the relevant approximants are inner log-linearizations, or more generally inner $\alpha$-linearizations, determined by finitely many points in the epigraph of the base function or, equivalently, in the hypograph of the original function.

Before formally presenting our definitions and main results, we begin in the next section with some background and notation that will be used throughout the paper.
%%%%%%%%%%%%%%%%%%%%%%%%%%%%%%%%%%%
\section{Preliminaries}\label{preliminaries-section}

The $n$-dimensional Euclidean space $\R^n$ is equipped with inner product $\langle x,y\rangle = \sum_{i=1}^n x_i y_i$, where $x,y\in\R^n$, and Euclidean norm $|x|=\sqrt{\langle x,x\rangle}$. The interior, boundary and closure of a set $A\subset\R^n$ are denoted by $\operatorname{int}(A)$, $\partial A$ and $\overline{A}$, respectively.  The closed ball centered at $x\in\R^n$ with radius $r>0$ is $B(x,r)=\{y\in\R^n:\,|x-y|\leq r\}$. In particular, we denote the Euclidean unit ball in $\R^n$ centered at the origin $o\in\R^n$ by $B_n=B(o,1)=\{x\in\R^n: |x|\leq 1\}$. The unit sphere in $\R^n$ is denoted $\Sp=\partial B_n=\{x\in\R^n: |x|=1\}$. 

Given $m,j\in\mathbb{N}$ with $m\geq j$, the Grassmannian manifold of $j$-dimensional subspaces of $\R^m$ is denoted $\Gr(m,j)$. For a hyperplane $H\in\Gr(n,n-1)$, we set  $\widetilde{H}=H\times\R$ and note that $\widetilde{H}\in\Gr(n+1,n)$. The orthogonal complement of $u\in\Sp$ is $u^\perp=\{x\in\R^n:\,\langle x,u\rangle=0\}\in\Gr(n,n-1)$. 

For two functions $f,g:\R^n\to[-\infty,+\infty]$,  the notation $f\leq g$ means that  $f(x)\leq g(x)$ pointwise for all $x$. The function $f$ is a \emph{minorant} of $g$ if $f\leq g$. We say that a function $g:\R^n\to\R$ is \emph{radial} if it is invariant under transformations in $\mathrm{O}(n)$, i.e., $g\circ\rho=g$ for all $\rho\in\mathrm{O}(n)$. 

%%%%%%%%%%%%%%%%%%%%%%%%%%%%%%%%%%%%%

\subsection{Background on convex sets in $\R^n$}

Let $A$ be a subset of $\R^n$. The \emph{convex hull} $\conv(A)$ of $A$ is the smallest convex set containing $A$, that is, $\conv(A)=\cap \{C: C\supset A,\, C\text{ is convex}\}$. The convex hull of a finite set of points in $\R^n$ is called a \emph{polytope}. An extreme point of a polytope $P$ is called a  \emph{vertex} of $P$.

 A \emph{convex body} $K$ is a convex, compact subset of $\R^n$ with nonempty interior. The set of all convex bodies in $\R^n$ is denoted by $\mathcal{K}^n$.  For $K,L\in\mathcal{K}^n$, the \emph{Hausdorff metric} $d_H:\mathcal{K}^n\times\mathcal{K}^n\to[0,\infty)$ is defined by 
		\begin{equation*}
		    d_H(K,L) = \inf \{ \lambda\geq 0 : K\subset L+\lambda B_n,\, L\subset K+\lambda B_n\}.
		\end{equation*}

The \emph{support function} $h_K:\Sp\to\R$ of $K\in\mathcal{K}^n$ is given by $h_K(u)=\sup_{x\in K}\langle x,u\rangle$. It is convex, and it is Minkowski linear in the sense that for any $K,L\in\mathcal{K}^n$ and any $a,b\geq 0$, we have $h_{aK+bL}=ah_K+bh_L$. The %\emph{width} of $K\in\mathcal{K}^n$ in the direction $u\in\Sp$ is $w_K(u)=h_K(u)+h_K(-u)$, and the 
\emph{mean width} of $K$ is $w(K)=2\int_{\Sp}h_K(u)\,d\sigma(u)$, where $\sigma$ denotes the normalized surface area measure on  $\Sp$. The mean width inherits Minkowski linearity from the support function, that is,    $w(aK+bL)=aw(K)+bw(L)$ for all $K,L\in\mathcal{K}^n$ and all $a,b\geq 0$, and it is invariant under orthogonal transformations. Therefore, the Minkowski symmetrization \eqref{minkowski-symmetral-convex-bodies} preserves mean width:
\[
w(\tau_u K)=\frac{1}{2}w(K)+\frac{1}{2}w(R_u K)=w(K).
\]

\noindent For more background on convex bodies, we refer the reader to the monographs by Gruber    \cite{GruberBook} and Schneider \cite{SchneiderBook}.

%%%%%%%%%%%%%%%%%%%%%%%%%%%%%%
\subsection{Background on convex functions and log-concave functions}

Next, we present the background from convex analysis that will be used throughout the paper. For general references on convex and log-concave   functions, we refer the reader to, for example, \cite{AKM2004, Colesanti-inbook, RockafellarBook, Rockafellar-Wets}.

Let $\mathrm{Conv}(\R^n)=\{\psi:\R^n\to(-\infty,+\infty]:\, \psi\text{ is convex}\}$. A natural way of embedding $\mathcal{K}^n$ into $\mathrm{Conv}(\R^n)$  comes via the \emph{indicator  function}, which for $K\in\mathcal{K}^n$  is defined by 
\begin{equation*}
I_K^\infty(x)=
    \begin{cases}
        0, & \text{if } x \in K\\
        +\infty, & \text{if } x\not\in K.
    \end{cases}
\end{equation*}
This function is convex, and we may embed  $\mathcal{K}^n$ into $\mathrm{Conv}(\R^n)$ via the mapping $K\mapsto I_K^\infty$. 

However, this approach has an immediate limitation: whereas every convex body has finite volume, a convex function need not be integrable. To surmount this obstacle, $\mathcal{K}^n$ is often embedded into  $\LC(\R^n)=\{f=e^{-\psi}:\,\psi\in\mathrm{Conv}(\R^n)\}$, the class of log-concave functions on $\R^n$. A function $f:\R^n\to[0,\infty)$ is \emph{logarithmically concave} (or \emph{log-concave}) if $\log f$ is concave, i.e., if for all $x,y\in\R^n$ and all $\lambda\in[0,1]$,
\[
\log f(\lambda x+(1-\lambda)y)\geq\lambda\log f(x)+(1-\lambda)\log f(y),
\]
which is equivalent to
\[
f(\lambda x+(1-\lambda)y)\geq f(x)^\lambda f(y)^{1-\lambda}.
\]
Every log-concave function $f\in\LC(\R^n)$ can be expressed in the form $f=e^{-\psi}$ for some $\psi\in\mathrm{Conv}(\R^n)$.  The usual embedding $\mathcal{K}^n\hookrightarrow\LC(\R^n)$ is achieved via the mapping $K\mapsto \mathbbm{1}_K$, where $\mathbbm{1}_K$ is the \emph{characteristic function} of $K$   defined by
\begin{equation*}
\mathbbm{1}_K(x)=
    \begin{cases}
        1, & \text{if } x \in K\\
        0, & \text{if } x\not\in K.
    \end{cases}
\end{equation*}
This function is log-concave since $\mathbbm{1}_K=e^{-I_K^\infty}$. 

Further conditions are typically imposed on the convex functions for technical reasons. We say that $\psi\in\mathrm{Conv}(\R^n)$ is \emph{proper} if $\dom(\psi)\neq\varnothing$, where  $\dom(\psi)=\{x\in\R^n: \psi(x)<+\infty\}$ is the \emph{effective domain} of $\psi$.  The function $\psi\in\mathrm{Conv}(\R^n)$ is \emph{coercive} if $\lim_{|x|\to\infty}\psi(x)=+\infty$. We let 
\[
\mathrm{Conv}_{\rm c}(\R^n)=\left\{\psi\in\mathrm{Conv}(\R^n): \, \substack{\psi\text{ is convex, coercive, proper}\\ \text{and lower semicontinuous}}\right\}
\]
 and
\[
\LC_{\rm c}(\R^n)=\{f=e^{-\psi}: \,\psi\in\mathrm{Conv}_{\rm c}(\R^n)\}.
\]
Note that if $f\in\LC_{\rm c}(\R^n)$, then $f\not\equiv 0$, $f$ is upper semicontinuous and $\lim_{|x|\to\infty}f(x)=0$.  We also let $\supp(f)=\{x\in\R^n: f(x)>0\}$ denote the \emph{support} of a function $f:\R^n\to\R$. 

For  $f\in\LC_{\rm c}(\R^n)$, the \emph{total mass} of $f$ is $J(f)=\int_{\R^n}f(x)\,dx$. It can be expressed via the  layer cake formula
\begin{equation}\label{total-mass-def}
J(f)=\int_0^\infty \vol_n(\lev_{\geq t}f)\,dt.
\end{equation}
Moreover,  every $f\in\LC_{\rm c}(\R^n)$ is integrable (see, e.g.,  \cite{cordero-erasquin-klartag}). Thus, in the functional setting, the total mass is a  natural analogue of volume since for the indicator function of a convex body $K\in\mathcal{K}^n$, we have  $J(\mathbbm{1}_K)=\vol_n(K)$. 

Every convex function $\psi:\R^n\to\R$ is uniquely determined by its \emph{epigraph} $\epi(\psi)=\{(x,t)\in\R^n\times\R: t\geq\psi(x)\}$. Equivalently, every log-concave function $f=e^{-\psi}\in\LC(\R^n)$ is uniquely determined by its \emph{hypograph} $\hyp(f)=\{(x,t)\in\R^{n}\times\R: t\leq f(x)\}$. 
For a function $f:\R^n\to[0,\infty)$ and  $t\geq 0$, the \emph{sublevel set} $\lev_{\leq t}f$ and the \emph{superlevel set} $\lev_{\geq t}f$ are  defined by 
\begin{align*}
\lev_{\leq t}f&=\{x\in\R^n: f(x)\leq t\}
\quad\text{and}\quad\lev_{\geq t}f=\{x\in\R^n: f(x)\geq t\}.
\end{align*}
Note that for every $K\in\mathcal{K}^n$ and every $t>0$, we have 
\begin{equation}\label{indicator-level-set}
\lev_{\geq t}\mathbbm{1}_K =\begin{cases}
        K, & \text{if } 0<t\leq 1\\
        \varnothing, & \text{if } t>1.
    \end{cases}
\end{equation}

The semicontinuity condition on $\mathrm{Conv}_{\rm c}(\R^n)$ is the functional analogue of the assumption that the convex sets in $\mathcal{K}^n$ are closed. More specifically, if $\psi$ is convex and lower semicontinuous, then its sublevel sets are convex and closed, respectively. Moreover, if $\psi\in\mathrm{Conv}_{\rm c}(\R^n)$ (respectively, $f=e^{-\psi}\in\LC_{\rm c}(\R^n)$), then the sublevel sets of $\psi$ (respectively, superlevel sets of $f$) are bounded since $\psi$ is coercive. 

Given an upper semicontinuous function $f:\R^n\to[0,\infty)$ and a hyperplane
$H=u^\perp\in \Gr(n,n-1)$, define its \emph{projection onto $H$} by
\[
\proj_H f(x):=\sup_{t\in\R} f(x+tu),\quad x\in H.
\]
If, in addition, $f(x)\to 0$ as $|x|\to\infty$, then for every $x\in H$ the supremum is attained,
and therefore for every $t\ge 0$,
\[
\lev_{\ge t}(\proj_H f)=\proj_H(\lev_{\ge t}f).
\]
In particular, this identity holds for every $f\in\LC_{\rm c}(\R^n)$. Moreover, for every upper semicontinuous $f$,
\[
\hyp(\proj_H f)=\proj_{H\times\R}(\hyp f).
\]
Note also that for any set $A\subset\R^n$, we have $\proj_H \mathbbm{1}_A = \mathbbm{1}_{\proj_H A}$. 
%For a proper lower semicontinuous convex function $\psi:\R^n\to(-\infty,+\infty]$, the natural projection onto $H$ is the infimal projection
%\[
%(\pi_H\psi)(x):=\inf_{t\in\R}\psi(x+tu),
%\]
%which satisfies
%\[
%\epi(\pi_H\psi)=\proj_{H\times\R}\epi(\psi).
%\]
\noindent  For more background on projection functions, see, for example, \cite{CGN-2018} and \cite[Section 3.1]{CLM-Hadwiger2}.
%\subsection{Epi-convergence and hypo-convergence}

A sequence of functions $\{\psi_j\}_{j\in\mathbb{N}}\subset\mathrm{Conv}(\R^n)$ \emph{epi-converges} to $\psi\in\mathrm{Conv}(\R^n)$, denoted by $\psi_j\stackrel{\text{epi}}{\longrightarrow}\psi$ or $\psi=\epilim_{j\in\mathbb{N}}\psi_j$, if for every $x\in\R^n$:
\begin{itemize}
\item[(i)] $\psi(x)\leq\liminf_{j\to\infty}\psi_j(x_j)$ for every sequence $\{x_j\}_{j\in\mathbb{N}}\subset\R^n$ that converges to $x$;

\item[(ii)] $\psi(x)=\lim_{j\to\infty}\psi_j(x_j)$ for at least one sequence $\{x_j\}_{j\in\mathbb{N}}\subset\R^n$ that converges to $x$.
\end{itemize}
A symmetric notion of convergence called hypo-convergence is defined in an analogous way. In particular, a sequence  $f_j=e^{-\psi_j}\in\LC(\R^n)$ \emph{hypo-converges} to $f=e^{-\psi}\in\LC(\R^n)$
if and only if the sequence $\psi_j\in\mathrm{Conv}(\R^n)$ epi-converges to $\psi\in\mathrm{Conv}(\R^n)$. In this case, we write $f_j\stackrel{\text{hyp}}{\longrightarrow}f$ or $f=\hypolim_{j\in\mathbb{N}} f_j$.

\vspace{2mm}

The following result will be used  in the proofs of some of our technical lemmas.

\begin{lemma}
\label{le:epi_conv_pointwise}\cite[Theorem 7.17]{Rockafellar-Wets}
Let $\psi_j\colon \R^n\to(-\infty,\infty]$, $j\in\N$, be a sequence of convex functions. If $\psi\colon \R^n\to(-\infty,\infty]$ is a lower semicontinuous, convex function such that $\dom(\psi)$ has nonempty interior, then $\psi=\epilim_{j\in\N} \psi_j$ if and only if $\psi_j$ converges pointwise to $\psi$ on a dense subset of $\R^n$. Equivalently, $\psi_j$ converges uniformly to $\psi$ on every compact set that does not contain a boundary point of $\dom(\psi)$.
\end{lemma}

%%%%%%%%%%%%%%%%%%%%%%%%%%%%%%%%%

\subsubsection{Operations on $\mathrm{Conv}(\R^n)$ and $\LC(\R^n)$}

The Minkowski sum and $\lambda$-homothety operations on $\mathcal{K}^n$ extend to operations on convex functions as follows. For $\varphi,\psi\in\mathrm{Conv}(\R^n)$ and $x\in\R^n$, the \emph{infimal convolution} $\varphi\square\psi$ is defined by
\[
(\varphi\square\psi)(x)=\inf_{y+z=x}\{\varphi(y)+\psi(z)\}.
\]
If $(\varphi\square\psi)(x)>-\infty$ for every $x\in\R^n$, then $\varphi\square\psi$ is convex, proper and satisfies the identity $\epi(\varphi\square\psi)=\epi(\varphi)+\epi(\psi)$. The \emph{epi-multiplication} $\lambda\sq\psi$ of $\psi\in\mathrm{Conv}(\R^n)$ and $\lambda>0$ is the function defined by %(see, for example,  \cite{CLM-2020})
\[
(\lambda\sq\psi)(x) = \lambda\psi\left(\frac{x}{\lambda}\right),\quad x\in\R^n.
\]
The \emph{Asplund sum} (or \emph{supremal convolution}) $f\star g$ of log-concave functions $f,g\in\LC(\R^n)$  is defined as
\[
(f\star g)(x)=\sup_{y+z=x}f(y)g(z),\quad x\in\R^n
\]
whenever the supremum exists. For $\lambda>0$, the \emph{$\lambda$-homothety} $\lambda\cdot f$ of $f$ is defined by
\[
(\lambda\cdot f)(x)=f\left(\frac{x}{\lambda}\right)^\lambda,\quad x\in\R^n.
\]

These operations are related as follows. For all  $\varphi,\psi\in\mathrm{Conv}_{\rm c}(\R^n)$ with  $f=e^{-\varphi},g=e^{-\psi}\in\LC_{\rm c}(\R^n)$ and all $a,b>0$ and $x\in\R^n$, we have
\begin{align*}
(a\cdot f\star b\cdot g)(x)&=\sup_{y+z=x}f\left(\frac{y}{a}\right)^a g\left(\frac{z}{b}\right)^b\\
&=\exp\left(-\inf_{y+z=x}\left\{a\varphi\left(\frac{y}{a}\right)+b\psi\left(\frac{z}{b}\right)\right\}\right)=\exp(-(a\sq\varphi\square b\sq\psi)(x)).
\end{align*}

It turns out that $\LC_{\rm c}(\R^n)$ is closed under the Asplund sum and $\lambda$-homothety operations, i.e., if $f,g\in\LC_{\rm c}(\R^n)$, then $f\star g\in\LC_{\rm c}(\R^n)$  and $\lambda\cdot f\in\LC_{\rm c}(\R^n)$. Moreover, if $a,b>0$ then $a\cdot f\star b\cdot f=(a+b)\cdot f$ (see, for example, \cite[Lemma 2.3(ii)]{Hofstatter-Schuster} or Lemma \ref{lem:alpha-distributivity} below). These operations extend the classical operations on convex bodies, in the sense that if $K,L\in\mathcal{K}^n$ and $a,b>0$, then $(a\sq I_K^\infty)\square(b\sq I_L^\infty)=I_{aK+bL}^\infty$ and  $a\cdot\mathbbm{1}_K\star b\cdot\mathbbm{1}_L=\mathbbm{1}_{aK+bL}$.

The  Pr\'ekopa--Leindler inequality implies that for any integrable functions $f,g\in\LC(\R^n)$ and any $\lambda\in(0,1)$,
\begin{equation}\label{PL-ineq}
    J(\lambda \cdot f \star (1-\lambda)\cdot g)\geq J(f)^\lambda J(g)^{1-\lambda}.
\end{equation}
This inequality is originally due to Pr\'ekopa \cite{Prekopa1} and Leindler \cite{Leindler} in dimension $n=1$. It was later extended to all dimensions by Pr\'ekopa \cite{Prekopa2,Prekopa3} and Borell \cite{Borell}. Under the mild assumption that $J(f)$ and $J(g)$ are both positive, the equality cases were characterized by Dubuc \cite{Dubuc}. (In degenerate cases, additional possibilities for equality may occur, but we will not need a full equality characterization here.) The Pr\'ekopa--Leindler inequality is a functional extension of the Brunn--Minkowski inequality in multiplicative form for convex bodies: choosing $f=\mathbbm{1}_K$ and $g=\mathbbm{1}_L$ where $K,L\in\mathcal{K}^n$,  we get
\[
\vol_n(\lambda K+(1-\lambda)L)\geq \vol_n(K)^{\lambda}\vol_n(L)^{1-\lambda}.
\]
When $\vol_n(K)\vol_n(L)>0$, as is the case since $K,L\in\mathcal{K}^n$, equality in this multiplicative form holds if and only if $K$ and $L$ are translates. %Indeed, equality requires equality in the usual Brunn--Minkowski inequality, so $K$ and $L$ are homothetic, and equality in the arithmetic-geometric mean step, which forces $\vol_n(K)=\vol_n(L)$; hence the homothety factor is $1$. 
For more background on the Brunn--Minkowski and Pr\'ekopa--Leindler inequalities, we refer the reader to, e.g.,  \cite{gardnerBM,SchneiderBook}.

%%%%%%%%%%%%%%%%%%%%%%%
\subsubsection{Properties of the Legendre--Fenchel transform}

Let $\psi\in\mathrm{Conv}(\R^n)$ be proper and lower semicontinuous. The \emph{Legendre--Fenchel (conjugate) transform} $\mathcal{L}\psi$ of $\psi$   is defined by 
\[
(\mathcal{L}\psi)(x)=\sup_{y\in\R^n}(\langle x,y\rangle-\psi(y)). 
\]
In particular, $\mathcal{L} \psi$ is a  proper, lower semicontinuous, convex function on $\R^n$ (see, e.g., \cite[Theorem 1.6.13]{RockafellarBook}). Furthermore, it satisfies the following properties (see \cite{RockafellarBook,Rockafellar-Wets}):

\begin{lemma}\label{conjugate-properties}
Let $\psi_j,\psi,\varphi\in\mathrm{Conv}(\R^n)$ be proper and lower semicontinuous. Then the following properties hold true: 
\begin{itemize}
\item[(i)] (Involution) $\mathcal{L}\mathcal{L}\psi=\psi$.

\item[(ii)] (Order-reversing) If $\psi\leq\varphi$, then $\mathcal{L}\psi\geq\mathcal{L}\varphi$.

\item[(iii)] (Dilation property) For any $a\geq 0$ we have $\mathcal{L}(a\sq\psi)=a\mathcal{L}\psi$.

\item[(iv)] (Additivity) $\mathcal{L}(\psi\square\varphi)=\mathcal{L}\psi+\mathcal{L}\varphi$.

\item[(v)] (Indicators transform into support functions) If $K\in\mathcal{K}^n$, then $\mathcal{L}(I_K^\infty)=h_K$.

\item[(vi)] (Equivalent modes of epi-convergence) $\psi_j\stackrel{\epi}{\longrightarrow}\psi$ if and only if  $\mathcal{L}\psi_j\stackrel{\epi}{\longrightarrow}\mathcal{L}\psi$.
\end{itemize}
\end{lemma}
\noindent In particular, part (vi) can be found in \cite[Theorem 11.34]{Rockafellar-Wets}.

%%%%%%%%%%%%%%%%%%
\subsection{Quasiconcave functions  and the layer cake definition of mean width}

A nonnegative function $f:\R^n\to[0,+\infty]$ is \emph{quasiconcave} if  for all $x,y\in\R^n$ and all $\lambda\in[0,1]$,
\[
f(\lambda x+(1-\lambda)y) \geq  \min\{f(x),f(y)\}.%M_{-\infty}^{(\lambda,1-\lambda)}(f(x),f(y))=
\]
In particular, every log-concave function is quasiconcave. The class of convex bodies in $\R^n$ can be embedded into the class of quasiconcave functions via the mapping $K\mapsto\mathbbm{1}_K$.  
It turns out that a function $f:\R^n\to[0,\infty)$ is quasiconcave if and only if its superlevel sets $\lev_{\geq t}f$ are convex. For the quasiconcave functions we will be interested in, we will impose the same restrictions as before and consider the class
\[
\mathcal{C}_{-\infty}(\R^n):=\left\{f:\R^n\to[0,\infty)\bigg| \,\substack{f\text{ is }\text{quasiconcave, upper semicontinuous,}\\{{\text{proper, and }\lim_{|x|\to\infty}f(x)=0}}}\right\}.
\]

In view of the layer cake representation  \eqref{total-mass-def} of the total mass, Bobkov, Colesanti and Fragal\`a \cite{BCF-2014}  defined the quermassintegrals of a quasiconcave function in an analogous way. 
In particular, the \emph{mean width} $w(f)$ of $f\in\mathcal{C}_{-\infty}(\R^n)$ may be defined as
\begin{equation}\label{level-set-mw}
w(f) = \int_0^\infty w(\lev_{\geq t}f)\,dt.
\end{equation}

%%%%%%%%%%%%%%%%
\subsection{Algebraic operations on $\alpha$-concave functions}\label{alpha-concave-defns}

%The family of $\alpha$-concave functions, where  $\alpha\in[-\infty,+\infty]$, is an important  subclass of quasiconcave functions. More specifically:

Given $\alpha\in[-\infty,+\infty]$, we shall consider the following family of $\alpha$-concave functions:
\[
\mathcal{C}_{\alpha}(\R^n):=\left\{f:\R^n\to[0,\infty)\bigg| \,\substack{f\text{ is }\alpha\text{-concave, upper semicontinuous,}\\{{\text{proper, and }\lim_{|x|\to\infty}f(x)=0}}}\right\}.
\]
Note that  if $\alpha_1<\alpha_2$, then $\mathcal{C}_{\alpha_1}(\R^n)\supset\mathcal{C}_{\alpha_2}(\R^n)$ (see \cite{BrascampLieb-1976}). Choosing special values of $\alpha$ yields the following notable subfamilies:
\begin{itemize}
\item $\mathcal{C}_{-\infty}(\R^n)$ is the class of nonnegative quasiconcave functions. It is the largest class of $\alpha$-concave functions since   $\mathcal{C}_{-\infty}(\R^n)\supset \mathcal{C}_\alpha(\R^n)$ for every $\alpha\in[-\infty,+\infty]$. %Note that ${\rm QC}(\R^n)\subset\mathcal{C}_{-\infty}(\R^n)$.

\item $\mathcal{C}_0(\R^n)=\LC_{\rm c}(\R^n)$ is the class of log-concave functions on $\R^n$.

\item $\mathcal{C}_1(\R^n)$ is the class of concave functions defined on convex sets $\Omega$ and extended by 0 outside of $\Omega$.

\item $\mathcal{C}_{+\infty}(\R^n)$ is the class of multiples of characteristic functions of convex sets in $\R^n$.
\end{itemize}

There are several standard ways of inducing an algebraic structure on $\mathcal{C}_\alpha(\R^n)$. The first definition we present can be found in, e.g., \cite{BCF-2014, Roysdon-Xing-Lp-functions}.

\begin{definition}\label{alpha-operations}
Let $\alpha\in[-\infty,+\infty]$ be given, and let $f,g\in\mathcal{C}_\alpha(\R^n)$.
\begin{itemize}
    \item[(i)]   The $\oplus_\alpha$-addition $f\oplus_\alpha g$ of $f$ and $g$ is defined by
\begin{equation}
(f\oplus_\alpha g)(x)=\sup_{y+z=x}M_\alpha^{(1,1)}(f(y),g(z))=\sup_{y+z=x}\begin{cases}
(f(y)^\alpha+g(z)^\alpha)^{1/\alpha}, &\text{if }\alpha\neq 0,\pm\infty;\\
f(y)g(z), &\text{if } \alpha=0;\\
\max\{f(y),g(z)\}, &\text{if }\alpha=+\infty;\\
\min\{f(y),g(z)\}, &\text{if }\alpha=-\infty.
\end{cases}
\end{equation}

\item[(ii)] For $\lambda>0$, the $\times_\alpha$ scalar multiplication $\lambda\times_\alpha f$ is defined  by
\begin{equation}
(\lambda\times_\alpha f)(x)=\begin{cases}
\lambda^{1/\alpha}f\left(\tfrac{x}{\lambda}\right), &\text{if }\alpha\neq 0,\pm\infty;\\
f(x)^\lambda, &\text{if }\alpha=0;\\
f(x), &\text{if }\alpha=\pm\infty.
\end{cases}
\end{equation}
\end{itemize}
\end{definition}

Another point of view may be taken by extending the  Asplund sum from the family of log-concave functions to the family of $\alpha$-concave functions. This approach is especially advantageous because in many applications it allows one to appeal directly to the geometry of the  epigraphs of convex functions. Let $\alpha\in(-\infty,0)$. Every function $f\in\mathcal{C}_\alpha(\R^n)$ admits a representation of the form
\begin{equation}\label{alpha-concave-rep}
f(x)=(1-\alpha\psi_{f,\alpha}(x))^{1/\alpha},
\end{equation}
where $\psi_{f,\alpha}\in\mathrm{Conv}(\R^n)$. This convex function is given by the following construction. For $\alpha\in(-\infty,0)$, the \emph{base function} of $f$, denoted by $\base_\alpha f$, is defined by
\[
\base_\alpha f:=\frac{1-f^\alpha}{\alpha}
\]
(see \cite{Rotem2013}). Equivalently, $\base_\alpha f$ is the unique function satisfying
\[
f(x)=(1-\alpha\,\base_\alpha f(x))^{1/\alpha}
\qquad\text{for all }x\in\R^n.
\]
Thus, in \eqref{alpha-concave-rep}, one may simply take
\[
\psi_{f,\alpha}=\base_\alpha f.
\] 
In the limiting case $\alpha= 0$, we have the identity $\base_0 f=-\log f$. Note also that if $K\in\mathcal{K}^n$, then $\base_\alpha\mathbbm{1}_K=I_K^\infty$ for every $\alpha\in(-\infty,0]$. 

The next lemma shows that the base function is well-defined when $\alpha\in(-\infty,0]$.

\begin{lemma}\label{lem:base-alpha-convc}
Let $\alpha\in(-\infty,0]$ and $f\in \mathcal{C}_\alpha(\R^n)$.  For $x\in\R^n$, define the base function of $f$ by
\[
\base_\alpha f(x):=
\begin{cases}
\dfrac{1-f(x)^\alpha}{\alpha}, & \text{if }\alpha<0 \text{ and } f(x)>0\\[1ex]
+\infty, & \text{if }\alpha<0 \text{ and } f(x)=0\\[1ex]
-\log f(x), &\text{if } \alpha=0 \text{ and } f(x)>0\\[1ex]
+\infty, &\text{if } \alpha=0 \text{ and } f(x)=0.
\end{cases}
\]
Then $\base_\alpha f\in \mathrm{Conv}_{\rm c}(\R^n)$.
\end{lemma}

\begin{proof}
Set $\psi:=\base_\alpha f$. We treat the cases $\alpha<0$ and $\alpha=0$ separately.

\medskip
\noindent
\underline{\emph{Case 1: $\alpha<0$.}} 
Since $f\in \mathcal{C}_\alpha(\R^n)$ is proper, $f$ is not identically zero, so there exists $x_0\in\R^n$ such that
$f(x_0)>0$; hence
\[
\psi(x_0)=\frac{1-f(x_0)^\alpha}{\alpha}<\infty.
\]
Thus $\psi\not\equiv+\infty$, i.e., $\psi$ is proper.

Next, we prove that $\psi$ is convex. Let $x,y\in\R^n$ and $\lambda\in[0,1]$, and set
$z:=\lambda x+(1-\lambda)y$.
If either $f(x)=0$ or $f(y)=0$, then either $\psi(x)=+\infty$ or $\psi(y)=+\infty$, respectively, so
\[
\psi(z)\leq \lambda\psi(x)+(1-\lambda)\psi(y)
\]
holds automatically. Assume now that $f(x),f(y)>0$. Since $f$ is $\alpha$-concave,
\[
f(z)\geq \bigl(\lambda f(x)^\alpha+(1-\lambda)f(y)^\alpha\bigr)^{1/\alpha}.
\]
Moreover, since $\alpha<0$ the map $t\mapsto t^\alpha$ is decreasing on $(0,\infty)$, and therefore
\[
f(z)^\alpha\leq \lambda f(x)^\alpha+(1-\lambda)f(y)^\alpha.
\]
Hence
\[
1-f(z)^\alpha
\ge
\lambda\bigl(1-f(x)^\alpha\bigr)+(1-\lambda)\bigl(1-f(y)^\alpha\bigr).
\]
Dividing by the negative number $\alpha$ reverses the inequality and gives
\[
\psi(z)\le \lambda\psi(x)+(1-\lambda)\psi(y).
\]
Thus, $\psi$ is convex.

Next we show that $\psi$ is lower semicontinuous. Let $t\in\R$. If $t\leq 1/\alpha$, then
$\lev_{\leq t}\psi=\varnothing$ (which is closed), because for every $x\in\R^n$ with $f(x)>0$, since $\alpha<0$ we have
\[
\psi(x)=\frac{1-f(x)^\alpha}{\alpha}>\frac{1}{\alpha},
\]
while if $f(x)=0$ then $\psi(x)=+\infty$. Now let $t>1/\alpha$. Since $\alpha<0$,
\[
\psi(x)\leq t
\iff
\frac{1-f(x)^\alpha}{\alpha}\leq t
\iff
f(x)^\alpha\leq 1-\alpha t
\iff
f(x)\geq (1-\alpha t)^{1/\alpha}.
\]
Therefore,
\[
\lev_{\leq t}\psi=\{x\in\R^n:\ f(x)\geq (1-\alpha t)^{1/\alpha}\}=\lev_{\geq(1-\alpha t)^{1/\alpha}}f,
\]
which is closed since $f$ is upper semicontinuous. Thus, $\psi$ is lower semicontinuous.

Finally, we prove that $\psi$ is coercive. Let $x_k\in\R^n$, $k\in\mathbb{N}$, be a sequence with $|x_k|\to\infty$. Since $f\in \mathcal{C}_\alpha(\R^n)$, by definition
$f(x_k)\to 0$. If $f(x_k)=0$ infinitely often, then $\psi(x_k)=+\infty$ along a subsequence. Otherwise, $f(x_k)>0$ for all sufficiently large $k$, and since $\alpha<0$ we have
$f(x_k)^\alpha\to +\infty$. Hence
\[
\psi(x_k)=\frac{1-f(x_k)^\alpha}{\alpha}\longrightarrow +\infty.
\]
Thus $\psi$ is coercive, so $\psi\in\mathrm{Conv}_{\rm c}(\R^n)$.

\medskip
\noindent
\underline{\emph{Case 2: $\alpha=0$.}}
In this case, we have
\[
\psi(x)=
\begin{cases}
-\log f(x), &\text{if } f(x)>0\\
+\infty, &\text{if } f(x)=0.
\end{cases}
\]
Again, $\psi$ is finite at some point because $\supp(f)$ is nonempty, so $\psi$ is proper.

To prove convexity, let $x,y\in\R^n$, $\lambda\in[0,1]$, and set $z=\lambda x+(1-\lambda)y$.
If $f(x)=0$ or $f(y)=0$, then the desired inequality is immediate. If $f(x),f(y)>0$, then log-concavity yields
\[
f(z)\geq f(x)^\lambda f(y)^{1-\lambda},
\]
and taking $-\log$ of both sides we get
\[
\psi(z)\leq \lambda\psi(x)+(1-\lambda)\psi(y).
\]
Thus $\psi$ is convex.

For lower semicontinuity, let $t\in\R$. Then
\[
\lev_{\leq t}\psi=\{x\in\R^n:\, f(x)\geq e^{-t}\}=\lev_{\geq e^{-t}}f,
\]
which is closed because $f$ is upper semicontinuous. Thus $\psi$ is lower semicontinuous.

For coercivity, let $x_k\in\R^n$, $k\in\mathbb{N}$, be a sequence with $|x_k|\to\infty$. Since $f(x_k)\to 0$, we get
\[
\psi(x_k)=-\log f(x_k)\to+\infty
\]
as $k\to\infty$ (with the convention $-\log 0=+\infty$). Hence, $\psi$ is coercive, which proves that $\psi\in\mathrm{Conv}_{\rm c}(\R^n)$.
\end{proof}

The following definition is due to Rotem \cite{Rotem2013}. 
\begin{definition}\label{alpha-Asplund-sum-def}
Given $\alpha\in(-\infty,0]$, $\lambda>0$  and $f,g\in\mathcal{C}_\alpha(\R^n)$, the \emph{$\alpha$-Asplund sum} $f\star_\alpha g$ and the \emph{$(\alpha,\lambda)$-homothety} $\lambda\cdot_\alpha f$ are defined by
\begin{equation}\label{base-fn}
\base_\alpha(f\star_\alpha g)=(\base_\alpha f)\square(\base_\alpha g)\qquad \text{and}\qquad \base_\alpha(\lambda\cdot_\alpha f)=\lambda\sq\base_\alpha f
\end{equation}
where the $\sq$  means epi-multiplication; that is,   $[\base_\alpha(\lambda\cdot_\alpha f)](x)=\lambda(\base_\alpha f)\left(\frac{x}{\lambda}\right)$ for all $x\in\R^n$. 
\end{definition}
In particular, $\star_0$ is just the Asplund sum $\star$ on $\LC(\R^n)$, and $\cdot_0$ is just the usual $\lambda$-homothety operation $\cdot$. Note that if $\alpha\in(-\infty,0]$, then  $\mathcal{C}_\alpha(\R^n)$ is closed under these operations, i.e., if $f,g\in\mathcal{C}_\alpha(\R^n)$ and $a,b>0$, then $a\cdot_\alpha f\star_\alpha b\cdot_\alpha  g\in\mathcal{C}_\alpha(\R^n)$ (see \cite{BCF-2014, Milman-Rotem-alpha}).

\begin{remark}\label{alpha-sum-properties}
Let $\alpha\in(-\infty, 0]$, $f,g\in\mathcal{C}_\alpha(\R^n)$ and $\lambda\in(0,1)$. By \cite[Proposition 10]{Rotem2013},  these operations coincide under convex combinations:
\[
\lambda\cdot_\alpha f\star_\alpha(1-\lambda)\cdot_\alpha g=\lambda\times_\alpha f\oplus_\alpha (1-\lambda)\times_\alpha g.
\]
Furthermore, if $f,g\in\mathcal{C}_\alpha(\R^n)$ are Borel measurable and $a,b>0$, then $a\times_\alpha f\oplus_\alpha b\times_\alpha g$ and $a\cdot_\alpha f\star_\alpha b\cdot_\alpha g$ are Lebesgue measurable \cite{BCF-2014, Rotem2013}.  %It is also useful to note that  $f\oplus_\alpha f=2\times_\alpha f$ and $f\star_\alpha f=2\cdot_\alpha f$.
For more details, we refer the reader to \cite{BCF-2014, Milman-Rotem-alpha, Rotem2013}. 
\end{remark}

We will also use the following distributivity property of the $\alpha$-Asplund operations. The case $\alpha=0$ was shown by Hofst\"atter and Schuster in \cite[Lemma 2.3(ii)]{Hofstatter-Schuster}.

\begin{lemma}\label{lem:alpha-distributivity}
    Let $\alpha\in(-\infty,0]$ and $f\in\mathcal{C}_\alpha(\R^n)$. For all $a,b>0$, we have
    \[
a\cdot_\alpha f\star_\alpha b\cdot_\alpha f=(a+b)\cdot_\alpha f.
    \]
\end{lemma}

\begin{proof}
    Let $\psi=\base_\alpha f$. By definition,
    \[
\base_\alpha\left(a\cdot_\alpha f\star_\alpha b\cdot_\alpha f\right)(x)=\inf_{y+z=x}\left\{a\psi\left(\frac{y}{a}\right)+b\psi\left(\frac{z}{b}\right)\right\},\quad x\in\R^n.
    \]
    By the convexity of $\psi$, for any $y,z\in\R^n$ with $y+z=x$ we have
    \[
\frac{a}{a+b}\psi\left(\frac{y}{a}\right)+\frac{b}{a+b}\psi\left(\frac{z}{b}\right)\geq \psi\left(\frac{a(y/a)+b(z/b)}{a+b}\right)=\psi\left(\frac{x}{a+b}\right).
    \]
    Multiplying both sides by $a+b>0$, we obtain
    \[
    a\psi\left(\frac{y}{a}\right)+b\psi\left(\frac{z}{b}\right)\geq (a+b)\psi\left(\frac{x}{a+b}\right).
    \]
    Hence
    \[
\base_\alpha\left(a\cdot_\alpha f\star_\alpha b\cdot_\alpha f\right)(x)\geq (a+b)\psi\left(\frac{x}{a+b}\right).
    \]
    But equality is attained by taking $y=\frac{a}{a+b}x$ and $z=\frac{b}{a+b}x$, since then $y/a=z/b=x/(a+b)$. Therefore, for every $x\in\R^n$
    \[
\base_\alpha\left(a\cdot_\alpha f\star_\alpha b\cdot_\alpha f\right)(x)= (a+b)\psi\left(\frac{x}{a+b}\right)=\base_\alpha\left((a+b)\cdot_\alpha f\right)(x).
    \]
    This implies $a\cdot_\alpha f\star_\alpha b\cdot_\alpha f=(a+b)\cdot_\alpha f$.
\end{proof}

\begin{remark}
Let $K,L\in\mathcal{K}^n$. For every $\alpha\in(-\infty,0]$, we have $\mathbbm{1}_K\star_\alpha\mathbbm{1}_L=\mathbbm{1}_{K+L}$ and $\lambda\cdot_\alpha\mathbbm{1}_K=\mathbbm{1}_{\lambda K}$ (see \cite{Rotem2013}); hence by Remark \ref{alpha-sum-properties}, for all $\lambda\in(0,1)$ we have
\[
\lambda\times_\alpha\mathbbm{1}_K\oplus_\alpha(1-\lambda)\times_\alpha\mathbbm{1}_L = \lambda\cdot_\alpha\mathbbm{1}_K\star_\alpha(1-\lambda)\cdot_\alpha\mathbbm{1}_L=\mathbbm{1}_{\lambda K+(1-\lambda)L}.
\]
\end{remark}

\begin{remark}\label{linear-level-sets}
By \cite[p. 38]{RockafellarBook}, for all  $f,g\in\mathcal{C}_{-\infty}(\R^n)$ and all $a,b> 0$ we have 
\begin{equation}
\lev_{\geq t} (a\times_{-\infty} f\oplus_{-\infty} b\times_{-\infty} g) = a\lev_{\geq t} f+b\lev_{\geq t} g.
\end{equation}
%{\color{red}This  operation, called the \emph{quasisum},  has applications to convex optimization; see, for example, \cite{}.}
\end{remark}

%%%%%%%%%%%%%%%%%%%%%%%%%%%%%%%%%%
\subsection{The $\alpha$-support function and $\alpha$-mean width of an $\alpha$-concave function}

Fix $\alpha\in(-\infty,0]$ and let $f\in\mathcal{C}_\alpha(\R^n)$. Rotem \cite{Rotem2013} defined the  \emph{$\alpha$-support function} $h_f^{(\alpha)}$ of $f$ by  $h_f^{(\alpha)}=\mathcal{L}(\base_\alpha f)$. Note that  $h_f^{(\alpha)}$ is a convex function on $\R^n$. The $\alpha$-support function is linear in the sense that for any $f,g\in\mathcal{C}_\alpha(\R^n)$ and any $a,b>0$, we have 
\begin{equation}\label{linearity-alpha-supprt}
h^{(\alpha)}_{a\cdot_\alpha f \star_\alpha b\cdot_\alpha g}=ah_f^{(\alpha)}+bh_g^{(\alpha)}.
\end{equation}
In particular, for $f\in\LC_{\rm c}(\R^n)$ we simply denote $h_f=h_f^{(0)}=\mathcal{L}(-\log f)$.
Note also that for every $\alpha\in(-\infty,0]$ and every $K\in\mathcal{K}^n$, we have  $h_{\mathbbm{1}_K}^{(\alpha)}=h_K$. For more background on support functions of $\alpha$-concave functions, see \cite{Artstein-Milman2010, Rotem2012}. 

\begin{definition}\label{alpha-mw-def}
Let $\alpha\in[-\infty,0]$ and $f\in\mathcal{C}_\alpha(\R^n)$. The \emph{$\alpha$-mean width} $w_\alpha(f)$ of $f$ is defined by 
\begin{equation}\label{mw-def}
    w_\alpha(f) = \begin{cases}
    \int_{\R^n}h_f^{(\alpha)}(x)\left(1-\frac{\alpha|x|^2}{2}\right)^{\frac{1-\alpha}{\alpha}}\,dx, &\text{if } \alpha\in(-\infty,0)\\
    \frac{2}{n}\int_{\R^n}h_f(x)\,d\gamma_n(x), &\text{if }  \alpha=0\\
    \int_0^\infty w(\lev_{\geq t}f)\,dt, &\text{if }  \alpha=-\infty.
    \end{cases}
\end{equation}
Here $\gamma_n$ is the standard Gaussian probability measure on $\R^n$ given by $d\gamma_n(x)=(2\pi)^{-n/2}e^{-|x|^2/2}\,dx$. 
\end{definition}

This definition is due to Rotem \cite{Rotem2012,Rotem2013} for $\alpha\in(-\infty,0]$, and due to  Bobkov, Colesanti and Fragal\`a \cite{BCF-2014} for $\alpha=-\infty$. By \eqref{linearity-alpha-supprt}, the $\alpha$-mean width is linear with respect to the operations $\star_\alpha$ and $\cdot_\alpha$. More specifically, for all  $\alpha\in(-\infty,0]$, all   $f,g\in\mathcal{C}_\alpha(\R^n)$ and all $a,b>0$,
\begin{equation}\label{linearity}
w_\alpha(a\cdot_\alpha f\star_\alpha b\cdot_\alpha g) = aw_\alpha(f)+bw_\alpha(g).
\end{equation}
In the case $\alpha=-\infty$, by Remark \ref{linear-level-sets} we have
\begin{equation}\label{QC-linear}
w_{-\infty}(a\times_{-\infty} f\oplus_{-\infty} b\times_{-\infty} g) = aw_{-\infty}(f)+bw_{-\infty}(g).
\end{equation}
The $\alpha$-mean width is also invariant under orthogonal transformations. Note that $w(f)=w_{-\infty}(f)$.  

The $\alpha$-mean width is also monotone, as the next lemma shows.

\begin{lemma}\label{lem:alpha-mw-monotone}
    Let $\alpha\in[-\infty,0]$ and $f,g\in\mathcal{C}_\alpha(\R^n)$. If $f\leq g$, then $w_\alpha(f)\leq w_\alpha(g)$.
\end{lemma}

\begin{proof}
We break the proof into three cases: (i) $\alpha\in(-\infty,0)$; (ii) $\alpha=0$; and (iii) $\alpha=-\infty$.

\medskip

    \noindent\underline{\emph{Case 1: $\alpha\in(-\infty,0)$.}} For $\alpha\in(-\infty,0)$, we have $\base_\alpha f=(1-f^\alpha)/\alpha$ and $h_f^{(\alpha)}=\mathcal{L}(\base_\alpha f)$. If $f\leq g$ and $\alpha<0$, then $f^\alpha\geq g^\alpha$, so $1-f^\alpha\leq 1-g^\alpha$. Dividing both sides by the negative number $\alpha$ reverses the inequality again, and we get $\base_\alpha f\geq\base_\alpha g$. 
    By Lemma \ref{conjugate-properties}(ii), the Legendre--Fenchel transform is order-reversing, so
    \[
h_f^{(\alpha)}=\mathcal{L}(\base_\alpha f)\leq\mathcal{L}(\base_\alpha g)=h_g^{(\alpha)}.
    \]
    When $\alpha<0$, the weight $(1-\alpha|x|^2/2)^{(1-\alpha)/\alpha}$ is strictly positive for all $x\in\R^n$. Thus the previous inequality implies that 
    \[
w_\alpha(f)=\int_{\R^n}h_f^{(\alpha)}(x)\left(1-\frac{\alpha|x|^2}{2}\right)^{\frac{1-\alpha}{\alpha}}dx
\leq \int_{\R^n}h_g^{(\alpha)}(x)\left(1-\frac{\alpha|x|^2}{2}\right)^{\frac{1-\alpha}{\alpha}}dx=w_\alpha(g).
    \]

    \medskip

    \noindent\underline{\emph{Case 2: $\alpha=0$.}} In this case, the mapping $f\mapsto h_f$ on $\LC_{\rm c}(\R^n)$ is again order-preserving (see, e.g., \cite{Rotem2012}): If $f,g\in\LC_{\rm c}(\R^n)$ and $f\leq g$, then $-\log f\geq -\log g$, so
    \[
h_f=\mathcal{L}(-\log f)\leq\mathcal{L}(-\log g)=h_g
    \]
    by Lemma \ref{conjugate-properties}(ii). Thus,
    \[
w_0(f)=\frac{2}{n}\int_{\R^n}h_f(x)\,d\gamma_n(x)\leq\frac{2}{n}\int_{\R^n}h_g(x)\,d\gamma_n(x)=w_0(g).
    \]

     \medskip

    \noindent\underline{\emph{Case 3: $\alpha=-\infty$.}} Let $f,g\in\mathcal{C}_{-\infty}(\R^n)$. If $f\leq g$, then $\lev_{\geq t}f\subset\lev_{\geq t}g$ for every $t$. Since the classical mean width is monotone, this implies $w(\lev_{\geq t}f)\leq w(\lev_{\geq t}g)$ for every $t$. Therefore
    \[
w_{-\infty}(f)=\int_0^\infty w(\lev_{\geq t}f)\,dt\leq\int_0^\infty w(\lev_{\geq t}g)\,dt=w_{-\infty}(g).
    \]
\end{proof}

\begin{remark}
For $\alpha\in(-\infty,0]$, define the function $G_\alpha\in\mathcal{C}_\alpha(\R^n)$ by \[G_\alpha(x)=\left(1-\frac{\alpha|x|^2}{2}\right)^{1/\alpha},
\]
and set $G_0(x)=\lim_{\alpha\to 0}G_\alpha(x)=e^{-|x|^2/2}$. In other words, $\base_\alpha G_\alpha(x)=|x|^2/2$. The $\alpha$-mean width was defined in \cite[Theorem 12]{Rotem2013} as the directional derivative of the integral at the ``point'' $G_\alpha$ in the ``direction'' $f$:
\[
w_\alpha(f)=\lim_{\varepsilon\to 0^+}\frac{\int G_\alpha\star_\alpha[\varepsilon\cdot_\alpha f]-\int G_\alpha}{\varepsilon}.
\]
It was also proved in \cite[Theorem 13]{Rotem2013} that $w_\alpha(f)$ has the formula in Definition \ref{alpha-mw-def}.

In the definition of mean width given in \cite{BCF-2014}, one again takes a directional derivative  of the integral in the direction $f$, but this time at the point $\mathbbm{1}_{B_n}$. By the layer cake principle, this is the same as integrating the mean width of the superlevel sets. This definition can also be extended to $\alpha$-concave functions. The case of $\alpha=-\infty$ in Definition \ref{alpha-mw-def} is justified since both of these notions of mean width coincide in the limit as $\alpha\to-\infty$. 
We refer the reader to \cite{BCF-2014, Rotem2012, Rotem2013, Roysdon-Xing-Lp-functions} for more details.
\end{remark}

It is natural to ask for continuity properties of the $\alpha$-mean width. We give an answer in the following

\begin{lemma}\label{lem:walpha-lsc}
Let $\alpha\in(-\tfrac{2}{n-1},0]$. Then the $\alpha$-mean width $w_\alpha$ is lower semicontinuous with respect to the topology of hypo-convergence on $\mathcal{C}_\alpha(\R^n)$.
\end{lemma}

\begin{proof}
We first treat the case $\alpha<0$, and then the case $\alpha=0$.

\medskip
\noindent
\underline{\emph{Case 1: $\alpha<0$.}} 
Let $f_k,f\in \mathcal{C}_\alpha(\R^n), k\in\mathbb{N}$, and assume that $f_k$ hypo-converges to $f$ as $k\to\infty$.
Set $\psi_k:=\base_\alpha f_k$ and $\psi:=\base_\alpha f$. By Lemma~\ref{lem:base-alpha-convc}, we have $\psi_k,\psi\in\mathrm{Conv}_{\rm c}(\R^n)$. 
Since the Legendre--Fenchel transform preserves epi-convergence of proper lower semicontinuous convex
functions (see Lemma \ref{conjugate-properties}(vi)), it follows that
\[
h_k:=h_{f_k}^{(\alpha)}=\mathcal{L}\psi_k \stackrel{\epi}{\longrightarrow} \mathcal{L}\psi=h_f^{(\alpha)}=:h.
\]

Next, we prove that
\[
h_k(x)\longrightarrow h(x)\qquad\text{for Lebesgue-a.e. }x\in\R^n.
\]
Since $\psi\in\mathrm{Conv}_{\rm c}(\R^n)$, its conjugate $h=\mathcal{L}\psi$ is proper, lower semicontinuous,
convex, and $\operatorname{int}(\dom (h))\neq\varnothing$. By Lemma \ref{le:epi_conv_pointwise}, we have $h_k\to h$ uniformly on
every compact subset of $\R^n\setminus \partial\dom (h)$. In particular, if
$x\in \operatorname{int}(\dom (h))$, then choosing $r>0$ with
$B(x,r)\subset \operatorname{int}(\dom (h))$ we get $h_k(x)\to h(x)$. Likewise, if
$x\in \R^n\setminus \overline{\dom (h)}$, then choosing $r>0$ with
$B(x,r)\subset \R^n\setminus \overline{\dom (h)}$ we get
$h_k(x)\to+\infty=h(x)$. Hence
\[
h_k(x)\longrightarrow h(x)\quad\text{for all }x\in \R^n\setminus \partial\dom(h).
\]
Since $\dom(h)$ is convex and has nonempty interior, its boundary $\partial\dom(h)$ has Lebesgue measure zero. Therefore,
\[
h_k(x)\longrightarrow h(x)\quad\text{for Lebesgue-a.e. }x\in \R^n.
\]

We now construct a common affine lower bound. Since $\psi\in\mathrm{Conv}_{\rm c}(\R^n)$, we can choose
$y\in\R^n$ such that $\psi(y)<\infty$. Since $\psi=\epilim_{k\in\mathbb{N}}\psi_k$, there exists a sequence $y_k\to y$ such that $\psi_k(y_k)\to \psi(y)$. Hence there exist constants $R,M>0$ and $k_0\in\N$ such that for all $k\geq k_0$,
\[
|y_k|\leq R
\quad\text{and}\quad
\psi_k(y_k)\leq M.
\]
By the definition of the Legendre--Fenchel  transform, this implies that for every $x\in\R^n$ and every $k\geq k_0$,
\[
h_k(x)=\sup_{z\in\R^n}\bigl(\langle x,z\rangle-\psi_k(z)\bigr)
\geq
\langle x,y_k\rangle-\psi_k(y_k)
\geq
-R|x|-M.
\]
Thus
\[
h_k(x)+R|x|+M\geq 0
\qquad\text{for all }x\in\R^n,\; k\geq k_0.
\]

Let
\[
\rho_\alpha(x):=\left(1-\frac{\alpha|x|^2}{2}\right)^{\frac{1-\alpha}{\alpha}},
\quad x\in\R^n.
\]
By definition,
\[
w_\alpha(f_k)=\int_{\R^n} h_k(x)\rho_\alpha(x)\,dx
\quad\text{and}\quad
w_\alpha(f)=\int_{\R^n} h(x)\rho_\alpha(x)\,dx.
\]
Next we prove that if $\alpha\in(-2/(n-1),0)$, then the function $(1+|x|)\rho_\alpha(x)$ is integrable on $\R^n$. Set $\beta:=-1/\alpha>0$, so that $\rho_\alpha(x)=(1+|x|^2/(2\beta))^{-\beta-1}$. We want to show that the integral $\int_{\R^n}(1+|x|)\rho_\alpha(x)\,dx$ is finite. For $x\in B_n$, the integrand is continuous and bounded, so it is integrable on $B_n$. For $|x|\geq 1$, we have $1+|x|^2/(2\beta)\geq |x|^2/(2\beta)$. Hence
\[
\rho_\alpha(x)=\left(1+\frac{|x|^2}{2\beta}\right)^{-\beta-1}\leq\left(\frac{|x|^2}{2\beta}\right)^{-\beta-1}=(2\beta)^{\beta+1}|x|^{-2\beta-2}.
\]
Therefore, for $|x|\geq 1$,
\[
0\leq (1+|x|)\rho_\alpha(x)\leq 2|x|\rho_\alpha(x)\leq 2(2\beta)^{\beta+1}|x|^{-2\beta-1}.
\]
Thus it is enough to check that $|x|^{-2\beta-1}\in L^1(\R^n\setminus B_n)$. Using polar coordinates, we get
\[
\int_{\R^n\setminus B_n}|x|^{-2\beta-1}\,dx=\vol_{n-1}(\partial B_n)\int_1^\infty r^{-2\beta-1}r^{n-1}\,dr=\vol_{n-1}(\partial B_n)\int_1^\infty r^{n-2\beta-2}\,dr.
\]
This converges if and only if $n-2\beta-2<-1$, i.e., if and only if 
\[
\beta>\frac{n-1}{2}.
\]
Now since $\beta=-1/\alpha$, this means that the integral converges (and $(1+|x|)\rho_\alpha(x)\in L^1(\R^n)$) if and only if  $\alpha\in(-2/(n-1),0)$. This implies that for any  $\alpha\in(-2/(n-1),0)$, 
\[
g(x):=R|x|+M \in L^1(\rho_\alpha(x)\,dx).
\]

Now since $h_k+g\geq 0$ for all $k\ge k_0$ and $h_k\to h$ almost everywhere, by Fatou's lemma
\[
\int_{\R^n} (h(x)+g(x))\rho_\alpha(x)\,dx
\leq
\liminf_{k\to\infty}\int_{\R^n} (h_k(x)+g(x))\rho_\alpha(x)\,dx.
\]
Subtracting the finite quantity $\int_{\R^n}g\rho_\alpha\,dx$ from both sides of the previous inequality and using the definition of the $\alpha$-mean width, we obtain
\[
w_\alpha(f)=\int_{\R^n} h(x)\rho_\alpha(x)\,dx
\leq
\liminf_{k\to\infty}\int_{\R^n} h_k(x)\rho_\alpha(x)\,dx=\liminf_{k\to\infty} w_\alpha(f_k).
\]
This shows that $w_\alpha$ is lower semicontinuous for $\alpha\in(-2/(n-1),0)$.

\medskip
\noindent
\underline{\emph{Case 2: $\alpha=0$.}} 
Let $f_k,f\in \LC_{\rm c}(\R^n)$, $k\in\mathbb{N}$, and assume that the sequence $f_k$ hypo-converges to $f$ as $k\to\infty$. Set $\psi_k:=-\log f_k$ and $\psi:=-\log f$. Then $\psi_k,\psi\in\mathrm{Conv}_{\rm c}(\R^n)$ and the sequence $\psi_k$ epi-converges to $\psi$. 

Define $h_k:=h_{f_k}=\mathcal{L}\psi_k$ and $h:=h_f=\mathcal{L}\psi$. Applying Lemma \ref{conjugate-properties}(vi) again, we have $h=\epilim h_k$. Following along the same lines as before, this implies
\[
h_k(x)\longrightarrow h(x)\qquad\text{for Lebesgue-a.e. }x\in\R^n.
\]

Now choose $y\in\R^n$ such that $\psi(y)<\infty$, and choose points $y_k\to y$ with
$\psi_k(y_k)\to\psi(y)$. Then there exist $R,M>0$ and $k_0\in\N$ such that
\[
h_k(x)\geq -R|x|-M
\qquad\text{for all }x\in\R^n,\; k\geq k_0.
\]
Since the Gaussian measure has finite first moment, the function $R|x|+M$ is integrable with
respect to $\gamma_n$. Using the definition of $w_0(f)$ and Fatou's lemma, we get
\[
w_0(f)=\frac2n\int_{\R^n} h(x)\,d\gamma_n(x)
\leq
\liminf_{k\to\infty}\frac2n\int_{\R^n} h_k(x)\,d\gamma_n(x)=\liminf_{k\to\infty} w_0(f_k).
\]
Therefore, $w_0$ is lower semicontinuous.
\end{proof}

%\begin{remark}
%    As the proof shows, the range of $\alpha$ in Lemma \ref{lem:walpha-lsc} can be extended from $\alpha\in[-1/n,0]$ to $\alpha\in(-2/(n-1),0]$. We do not know if the value $-2/(n-1)$ is sharp.
%\end{remark}
%%%%%%%%%%%%%%%%%%
\section{Minkowski symmetrizations of functions}\label{Minkowski-symm-section}

In this section, we present the main definition of the paper, which is a parameterized family of Minkowski symmetrals. We also study some of the properties of these symmetrizations.

%%%%%%%%%%%%%%%%%%%%%%%%%%%%%%%%%%%%

\subsection{The $\alpha$-Minkowski symmetral of an $\alpha$-concave function}

The Minkowski symmetral of a set depends on its reflection about the hyperplane of symmetrization. Similarly, any definition of a Minkowski symmetral of a function should  also depend on a reflection about a hyperplane. Recall that for a function $f:\R^n\to\R$ and $u\in\Sp$, the \emph{reflection} $R_u f$ of $f$ with respect to the hyperplane $H=u^\perp\in\Gr(n,n-1)$ is the function on $\R^n$ defined by
\[
R_u f(x):=f(R_u(x)),%=f(x-2\langle x,u\rangle u).
\]
where $R_u(x)=x-2\langle x,u\rangle u$ is the reflection of $x$ about $H$. Note that $R_u^2=\mathrm{Id}$ and $R_u^{-1}=R_u$.

Motivated by the discussions in the introduction and Section \ref{preliminaries-section}, we formulate the following

\begin{definition}\label{mainDef-new}
Let $\alpha\in(-\infty,0)$, $f\in\mathcal{C}_\alpha(\R^n)$ and $u\in\Sp$ be given. The $\alpha$-\emph{Minkowski symmetral} $\tau_u^\alpha f$ of $f$ with respect to the hyperplane $H=u^\perp\in\Gr(n,n-1)$ is defined by
\begin{equation}\label{mainEqn}
    \tau_u^\alpha f:=\frac{1}{2}\cdot_\alpha f\star_\alpha\frac{1}{2}\cdot_\alpha R_u f.
\end{equation}
\end{definition}
\noindent In the special case $\alpha= 0$,  we define 
\[
\tau_u f:=\tau^0_u f=\frac{1}{2}\cdot f\star\frac{1}{2}\cdot R_u f,
\]
and for $\alpha=-\infty$, we set
\begin{align*}
\tau_u^{-\infty} f&:=\frac{1}{2}\times_{-\infty}f\oplus_{-\infty}\frac{1}{2}\times_{-\infty}R_u f.%\\
%\tau_u^{+\infty} f&:=\frac{1}{2}\times_{+\infty}f\oplus_{+\infty}\frac{1}{2}\times_{+\infty}R_u f.
\end{align*}

We denote an ordered sequence of iterated $\alpha$-Minkowski symmetrizations by  $\bigcirc_{i=1}^k\tau_{u_i}^\alpha f:=\tau_{u_k}^\alpha\circ\cdots\circ\tau_{u_1}^\alpha f$. Note that the order is important here as Minkowski symmetrizations do not commute in general. We also use the notation $\tau_H^\alpha f$, $R_H$, etc., to denote symmetrizations,  reflections, etc., about the hyperplane $H=u^\perp$. For $\alpha\in(-\infty,0)$, the binary operation underlying $\tau_u^\alpha f$ may be viewed as a special case of the  
$(p,\mu)$-convolution of Salani \cite{Salani-2015} with $p=\alpha$, $\mu=1/2$, and the second input given by the reflection $R_u f$.

For Definition \ref{mainDef-new} to make sense, we need to show that the $\alpha$-Minkowski symmetral of an $\alpha$-concave function is $\alpha$-concave. This will be shown in the following

\begin{lemma}\label{closure-lemma}
Let $f$ be a nonnegative function on $\R^n$, $\alpha\in[-\infty,0]$ and $u\in\Sp$. 
\begin{itemize}
\item[(i)] If $f\in\mathcal{C}_\alpha(\R^n)$, then $R_u f\in\mathcal{C}_\alpha(\R^n)$.

\item[(ii)] If $f\in\mathcal{C}_\alpha(\R^n)$, then $\tau_u^\alpha f\in\mathcal{C}_\alpha(\R^n)$.
\end{itemize}
\end{lemma}

To prove Lemma \ref{closure-lemma}, we will need the following basic properties of reflections.

\begin{lemma}\label{reflection-lemma}
Let $f,g:\R^n\to\R$ be any functions and let $u\in\Sp$.
\begin{itemize}
\item[(i)] (Involution property) $R_u (R_u f)=f$.

\item[(ii)] (Monotonicity) $f\leq g$ if and only if $R_u f\leq R_u g$.

\item[(iii)] (Linearity) For all $x,y\in\R^n$ and every $\lambda\in[0,1]$, we have
\begin{equation}\label{reflection-linear}
R_u(\lambda x+(1-\lambda) y)=\lambda R_u(x)+(1-\lambda)R_u(y).
\end{equation}

\item[(iv)] (Superlevel sets property) For every $t\in\R$, we have $\lev_{\geq t}R_u f=R_u(\lev_{\geq t}f)$.

\item[(v)] (Graphical properties) Let $\widetilde{H}=H\times \R$. Then $\epi(R_H f)=R_{\widetilde{H}}\epi(f)$ and $\hyp(R_H f)=R_{\widetilde{H}}\hyp(f)$.
\end{itemize}
\end{lemma}

\begin{proof}
(i) For every $x\in\R^n$ and every $u\in\mathbb{S}^{n-1}$, 
\[
(R_u(R_u f))(x)=(R_u f)(R_u x)=f(R_u(R_u x))=f(x).
\]

(ii) If $f\leq g$, then for every $x\in\R^n$ we have
\[
R_u f(x)=f(R_u(x))\leq g(R_u(x))=R_u g(x).
\]
The other direction now follows by replacing $f$ and $g$ by $R_u f$ and $R_u g$, respectively, and then using (i).

(iii) Let $x,y\in\R^n$ and $\lambda\in[0,1]$. By definition, 
\begin{align}
R_u(\lambda x+(1-\lambda) y) &= \lambda x+(1-\lambda)y-2\langle \lambda x+(1-\lambda)y,u\rangle u \nonumber\\
&=\lambda(x-2\langle x,u\rangle u)+(1-\lambda)(y-2\langle y,u\rangle u) \nonumber\\
&=\lambda R_u(x)+(1-\lambda)R_u(y).
\end{align}

(iv) For every $t\in\R$, from (i) we get 
\begin{align*}
\lev_{\geq t}R_u f&=\{x\in\R^n:\, R_u f(x)\geq t\}=\{x\in\R^n: f(R_u(x))\geq t\}\\
&=\{R_u(x)\in\R^n:\, f(x)\geq t\}=R_u\lev_{\geq t}f.
\end{align*}

(v) Applying (i) we get
\begin{equation}\label{reflect-epigraph}
\begin{aligned}
    \epi(R_H f)
&=\{(x,t)\in\R^n\times\R:\ t\ge R_H f(x)\} \\
&=\{(x,t)\in\R^n\times\R:\ t\ge f(R_Hx)\}  \\
&=\{(R_Hx,t)\in\R^n\times\R:\ t\ge f(x)\}
=R_{\widetilde H}\epi(f).
%\epi(R_H f)&=\{(x,t)\in\R^n\times\R:\, R_H f(x)\geq t\}
%=\{(x,t)\in\R^n\times\R:\, f(R_H(x))\geq t\} \nonumber\\
%&=\{(R_H(x),t)\in\R^n\times\R:\,f(x)\geq t\}
%=R_{\widetilde{H}}\epi(f).
\end{aligned}
\end{equation}
The  identity for hypographs is handled in the same way.
\end{proof}

\begin{proof}[Proof of Lemma \ref{closure-lemma}]

 (i) Let $\alpha\in(-\infty,0]$ and $f\in\mathcal{C}_\alpha(\R^n)$. Then by Lemma \ref{lem:base-alpha-convc}, we have $\psi:=\base_\alpha f\in\mathrm{Conv}_{\rm c}(\R^n)$. By the linearity \eqref{reflection-linear} of $R_u(x)$ and the convexity of $\psi$, for all $x,y\in\R^n$ and all  $\lambda\in[0,1]$, 
\begin{align*}
    R_u\psi(\lambda x+(1-\lambda)y)
    &=\psi(R_u(\lambda x+(1-\lambda)y))\\
    &=\psi(\lambda R_u(x)+(1-\lambda)R_u(y))\\
    &\leq \lambda\psi(R_u(x))+(1-\lambda)\psi(R_u(y))\\
    &=\lambda R_u\psi(x)+(1-\lambda)R_u\psi(y).
\end{align*}
Hence, $R_u\psi:\R^n\to(-\infty,+\infty]$ is a convex function. Reflections also preserve the lower semicontinuity, properness and coercivity of $\psi$, so $R_u\psi\in\mathrm{Conv}_{\rm c}(\R^n)$. 

Suppose first that $\alpha\in(-\infty,0)$. Using the identity
\[
R_u f(x)=f(R_u(x))=(1-\alpha\psi(R_u(x)))^{1/\alpha}=(1-\alpha R_u\psi(x))^{1/\alpha},
\]
which holds for every $x\in\R^n$, along with \eqref{alpha-concave-rep} and the fact that $R_u\psi\in\mathrm{Conv}_{\rm c}(\R^n)$, we conclude that $R_u f$ is $\alpha$-concave. Similarly, for $\alpha=0$, we have
\[
R_u f(x)=f(R_u(x))=e^{-\psi(R_u(x))}=e^{-R_u\psi(x)},
\]
for every $x\in\R^n$, which implies $R_u f\in\LC_{\rm c}(\R^n)$ since $R_u\psi\in\mathrm{Conv}_{\rm c}(\R^n)$.

For $\alpha=-\infty$, note that if $f:\R^n\to\R$ is quasiconcave, then  by \eqref{reflection-linear} and the definition of quasiconcavity, for every $x,y\in\R^n$ and every $\lambda\in[0,1]$ we have 
\begin{align*}
R_u f(\lambda x+(1-\lambda)y)&=f(R_u(\lambda x+(1-\lambda)y))=f(\lambda R_u(x)+(1-\lambda)R_u(y))\\
&\geq \min\{f(R_u(x)), f(R_u(y))\}
=\min\{R_u f(x), R_u f(y)\}.
\end{align*}
Therefore, $R_u f$ is also quasiconcave.

Since $f\in\mathcal{C}_\alpha(\R^n)$ is upper semicontinuous, its superlevel sets are closed (see \cite[Theorem 7.1]{RockafellarBook}). Thus, by Lemma \ref{reflection-lemma}(iv), the superlevel sets $\lev_{\geq t}R_u f=R_u(\lev_{\geq t}f)$ are also  closed since reflections preserve closedness. Hence $R_u f$ is upper semicontinuous. Similarly, $R_u f$ is proper and $\lim_{|x|\to\infty}R_u f(x)=0$. Altogether, this proves  that  $R_u f\in\mathcal{C}_\alpha(\R^n)$. 

(ii) Recall that if $f,g\in\mathcal{C}_\alpha(\R^n)$ and $a,b>0$, then $a\cdot_\alpha f\star_\alpha b\cdot_\alpha  g\in\mathcal{C}_\alpha(\R^n)$. Applying this with $a=b=1/2$ and $g=R_u f$ and using (i), it  follows that  if $f\in\mathcal{C}_\alpha(\R^n)$, then $\tau_u^\alpha f\in\mathcal{C}_\alpha(\R^n)$. 
\end{proof}
%%%%%%%%%%%%%%%%%%%%%%

\begin{remark}\label{rmk:indicator-symm}
Since $R_u \mathbbm{1}_K=\mathbbm{1}_{R_u K}$,  for every  $\alpha\in(-\infty,0]$ we have 
\[
\tau_u^\alpha\mathbbm{1}_K = \frac{1}{2}\cdot_\alpha\mathbbm{1}_K\star_\alpha\frac{1}{2}\cdot_\alpha\mathbbm{1}_{R_u K}=\mathbbm{1}_{\frac{1}{2}K+\frac{1}{2}R_u K}=\mathbbm{1}_{\tau_u K}.
\]
Moreover, for $\alpha=-\infty$ we also have $\tau_u^{-\infty}\mathbbm{1}_K=\mathbbm{1}_{\tau_u K}$. In this way, Definition \ref{mainDef-new} extends the classical definition of a Minkowski symmetrization  from the class of convex bodies $\mathcal{K}^n$ to the class of $\alpha$-concave functions $\mathcal{C}_\alpha(\R^n)$ when $\alpha\in[-\infty,0]$.
\end{remark}

\begin{remark}
For $f\in\LC_{\rm c}(\R^n)$, Colesanti \cite{Colesanti-RS-2006} defined the closely related (Asplund) \emph{difference function} $\Delta_*f=\frac{1}{2}\cdot f\star\frac{1}{2}\cdot\overline{f}$, where $\overline{f}(x)=f(-x)$, and proved a functional analogue of the Rogers--Shephard inequality.
\end{remark}

%%%%%%%%%%%%%%%%%%%%%%%%%%%%%%%%%%%
\subsection{Properties of the $\alpha$-Minkowski symmetral}\label{properties-Minkowski-symm}

In analogy with the classical Minkowski symmetral of a convex body, the $\alpha$-Minkowski symmetral of an $\alpha$-concave function satisfies the following basic properties.

\begin{proposition}\label{mainProp}
Fix $\alpha\in[-\infty,0]$ and let  $f,g\in\mathcal{C}_\alpha(\R^n)$. Let   $u\in\mathbb{S}^{n-1}$ and $H=u^\perp$. 
\begin{itemize}
    \item[(i)] ($\alpha$-Mean width preserving) We have $w_\alpha(\tau_u^\alpha f)=w_\alpha(f)$.

     \item[(ii)] (Monotonicity) If  $f\leq g$, then $\tau^\alpha_u f\leq \tau^\alpha_u g$. 

    \item[(iii)] (Symmetry with respect to $H=u^\perp$) The $\alpha$-Minkowski symmetral $\tau_u^\alpha f$ is symmetric with respect to $u^\perp$, that is, $R_u(\tau_u^\alpha f)=\tau_u^\alpha f$. 

    \item[(iv)] (Idempotence) We have $\tau_u^\alpha(\tau_u^\alpha f)=\tau_u^\alpha f$.
    
    \item[(v)] (Invariance on $H$-symmetric functions) 
    If  $f_H\in\mathcal{C}_\alpha(\R^n)$ is $H$-symmetric, that is $R_H f_H=f_H$, then $\tau_H^\alpha f_H=f_H$.

    \item[(vi)] (Projection invariance) We have $\proj_H\tau_u^\alpha f=\proj_H f$. 

  \item[(vii)] (Total mass nondecreasing) If $\alpha\in[-1/n,0]$ and $f\in\mathcal{C}_\alpha(\R^n)$, then  $J(\tau_u^\alpha f) \geq J(f)$. In the case $\alpha=0$, if $J(f)>0$, then equality holds if and only if $R_u f$ is a translate of $f$, i.e., if and only if there exists $z\in\R^n$ such that $R_u f(x)=f(x-z)$ for all $x\in\R^n$.

    \item[(viii)] (Superlevel sets property) For every $t>0$, we have $\lev_{\geq t}\tau_u^{-\infty}f=\tau_u\lev_{\geq t}f$.
\end{itemize}
\end{proposition}

\begin{proof}
(i) We modify the arguments of Rotem in  \cite[Proposition 3.1]{Rotem2012}. Given $\alpha\in(-\infty,0)$, let $f\in\mathcal{C}_\alpha(\R^n)$.  Then $f(x)=(1-\alpha\base_\alpha f(x))^{1/\alpha}$.  For every  $T\in{\rm GL}(n)$, we have $\base_\alpha(f\circ T)=\frac{1-(f\circ T)^\alpha}{\alpha}$ and hence for all $x\in\R^n$,
\begin{align*}
    h^{(\alpha)}_{f\circ T}(x) &=\sup_{y\in\R^n}\big\{\langle x,y\rangle-\base_\alpha(f\circ T)(y)\big\}
    =\sup_{z\in\R^n}\left\{\left\langle x,T^{-1}(z)\right\rangle-(\base_\alpha f)(z)\right\}\\
    &=\sup_{z\in\R^n}\left\{\left\langle (T^{-1})^*(x),z\right\rangle-(\base_\alpha f)(z)\right\}
    =h_f^{(\alpha)}((T^{-1})^*(x)).
\end{align*}
The reflection operator is $R_u=I_n-2uu^T\in\mathrm{O}(n)$, where $I_n$ is the $n\times n$ identity matrix, and it  satisfies $|\det R_u|=1$. Hence $(R_u^{-1})^*=R_u$, and thus $h^{(\alpha)}_{f\circ R_u}(x)=h_f^{(\alpha)}(R_u(x))$ for every $x\in\R^n$. Since  $w_\alpha(R_u f)=w_\alpha(f)$, we have $w_\alpha(\tau_u^\alpha f)=w_\alpha(f)$ by the linearity \eqref{linearity} of the $\alpha$-mean width. In the case $\alpha=0$, the same proof carries over mutatis mutandis. 

We now consider the case $\alpha=-\infty$. For every  $f\in\mathcal{C}_{-\infty}(\R^n)$, by Remark \ref{linear-level-sets} and the invariance of $w_{-\infty}$ under reflections, we get that
\begin{align*}
w_{-\infty}(\tau_u^{-\infty}f)&=\int_0^\infty w\left(\lev_{\geq t}\left(\tfrac{1}{2}\times_{-\infty} f\oplus_{-\infty}\tfrac{1}{2}\times_{-\infty} R_u f\right)\right)dt\\
&=\frac{1}{2}\int_0^\infty w(\lev_{\geq t}f)\,dt+\frac{1}{2}\int_0^\infty w(\lev_{\geq t}R_u f)\,dt\\
&=\frac{1}{2}w_{-\infty}(f)+\frac{1}{2}w_{-\infty}(R_u f)=w_{-\infty}(f).
\end{align*}

(ii) By   Definition \ref{alpha-operations} and Lemma \ref{reflection-lemma}(ii),  for all $\alpha\in(-\infty,0)$ and $x\in\R^n$,
\begin{align*}
\tau_u^\alpha f(x)&=\sup_{y+z=x}\left(\tfrac{1}{2}f(2y)^\alpha+\tfrac{1}{2}R_u f(2z)^\alpha\right)^{1/\alpha}
\leq \sup_{y+z=x}\left(\tfrac{1}{2}g(2y)^\alpha+\tfrac{1}{2}R_u g(2z)^\alpha\right)^{1/\alpha}=\tau_u^\alpha g(x).
%\\
%&\begin{cases}
%\leq \sup_{y+z=x}\left(\tfrac{1}{2}g(2y)^\alpha+\tfrac{1}{2}R_u g(2z)^\alpha\right)^{1/\alpha}=\tau_u^\alpha g(x), &\alpha\in(0,+\infty)\\
%\leq \sup_{y+z=x}\left(\tfrac{1}{2}g(2y)^\alpha+\tfrac{1}{2}R_u g(2z)^\alpha\right)^{1/\alpha}=\tau_u^\alpha g(x), &\alpha\in(-\infty,0).
%\end{cases}
\end{align*}
In the case $\alpha=0$, for every $x\in\R^n$ we have
\begin{align*}
\tau_u  f(x)&=\sup_{y+z=x}\sqrt{f(2y)R_u f(2z)}
\leq \sup_{y+z=x}\sqrt{g(2y)R_u g(2z)}=\tau_u  g(x).
\end{align*}
Similarly, in the case $\alpha=-\infty$, for every $x\in\R^n$ we have
\begin{align*}
\tau_u^{-\infty}f(x)%&=\left(\tfrac{1}{2}\cdot_{-\infty}f\star_{-\infty}\tfrac{1}{2}\cdot_{-\infty}R_u f\right)(x)\\
&=\sup_{y+z=x}\min\{f(y),R_u f(z)\}
\leq \sup_{y+z=x}\min\{g(y),R_u g(z)\}
=\tau_u^{-\infty}g(x).
\end{align*}

(iii) First let $\alpha\in(-\infty,0)$. By Lemma \ref{reflection-lemma},  
\begin{align*}
R_u\tau_u^\alpha f(x)
&=\tau_u^\alpha f(R_u x) 
=\sup_{y+z=R_u x}
\left(\frac12 f(2y)^\alpha+\frac12 (R_u f)(2z)^\alpha\right)^{1/\alpha} \\
&=\sup_{y+z=R_u x}
\left(\frac12 f(2y)^\alpha+\frac12 f(2R_u z)^\alpha\right)^{1/\alpha} =\sup_{R_u y+R_u z=x}
\left(\frac12 f(2y)^\alpha+\frac12 f(2R_u z)^\alpha\right)^{1/\alpha} \\
&=\sup_{a+b=x}
\left(\frac12 f(2R_u a)^\alpha+\frac12 f(2b)^\alpha\right)^{1/\alpha} 
=\sup_{a+b=x}
\left(\frac12 f(2a)^\alpha+\frac12 f(2R_u b)^\alpha\right)^{1/\alpha} 
=\tau_u^\alpha f(x).
\end{align*}
for all $x\in\R^n$. The case $\alpha=0$ follows along the same lines. For $\alpha=-\infty$, we have
\begin{align*}
R_u\tau_u^{-\infty}f(x) &= \sup_{y+z=R_u(x)}\min\{f(y),R_u f(z)\}
=\sup_{R_u(y)+R_u(z)=x}\min\{f(R_u(y)),R_u f(R_u(z))\}\\
&=\sup_{y+z=x}\min\{R_u f(y),f(z)\}=\tau_u^{-\infty}f(x)
\end{align*}
for all $x\in\R^n$. 

(iv) By (iii) and Lemma \ref{lem:alpha-distributivity}, for every $x\in\R^n$ we have
\[
\tau_u^\alpha(\tau_u^\alpha f)(x)=\left(\tfrac{1}{2}\cdot_\alpha\tau_u^\alpha f\star_\alpha\tfrac{1}{2}\cdot_\alpha R_u(\tau_u^\alpha f)\right)(x)=\left(\tfrac{1}{2}\cdot_\alpha\tau_u^\alpha f\star_\alpha\tfrac{1}{2}\cdot_\alpha\tau_u^\alpha f\right)(x)=\tau_u^\alpha f(x).
\]

(v) Since $f_H\in\mathcal{C}_\alpha(\R^n)$ is $H$-symmetric, we have $f_H=R_H f_H$. Hence, by Lemma \ref{lem:alpha-distributivity} for every $x\in\R^n$, 
\[
(\tau_H^\alpha f_H)(x)=\left(\tfrac{1}{2}\cdot_\alpha f_H\star_\alpha\tfrac{1}{2}\cdot_\alpha R_H f_H\right)(x)=\left(\tfrac{1}{2}\cdot_\alpha f_H\star_\alpha \tfrac{1}{2}\cdot_\alpha f_H\right)(x)=f_H(x).
\]

(vi) First note that for every $x\in\R^n$,
\begin{align*}
\proj_H R_u f(x) &=\sup_{t\in\R}R_u f(x+tu)=\sup_{t\in\R}f(R_u(x+tu))
=\sup_{t\in\R}f(x+tu-2\langle x+tu,u\rangle u)\\
&=\sup_{t\in\R}f(x-(t+2\langle x,u\rangle) u)
=\sup_{s\in\R}f(x+su)=\proj_H f(x).
\end{align*}
Using this, \cite[Proposition 5.2]{Roysdon-Xing-Lp-functions} and the identity $\frac{1}{2}\cdot_\alpha f\star_\alpha\frac{1}{2}\cdot_\alpha f=f$, we obtain
\begin{align*}
\proj_H\tau_u^\alpha f=\proj_H\left(\tfrac{1}{2}\cdot_\alpha f\star_\alpha\tfrac{1}{2}\cdot_\alpha R_u f\right)
%=\proj_H\left(\tfrac{1}{2}\cdot_\alpha f\right)\star_\alpha\proj_H\left(\tfrac{1}{2}\cdot_\alpha R_u f\right)
=\left(\tfrac{1}{2}\cdot_\alpha\proj_H f\right)\star_\alpha\left(\tfrac{1}{2}\cdot_\alpha \proj_H R_u f\right)
=\proj_H f.
\end{align*}

(vii) By \cite[Corollary 11]{Rotem2013}, for every $\alpha\in[-1/n,0]$ we have
\[
J(\tau_u^\alpha f) = J(\tfrac{1}{2}\cdot_\alpha f\star_\alpha \tfrac{1}{2}\cdot_\alpha R_u f) \geq \left[\tfrac{1}{2} J(f)^{\frac{\alpha}{1+n\alpha}}+\tfrac{1}{2}J(R_u f)^{\frac{\alpha}{1+n\alpha}}\right]^{\frac{1+n\alpha}{\alpha}}=J(f)
\]
where we used the fact that $J(f)=J(R_u f)$. The case $\alpha=0$ follows from the Pr\'ekopa--Leindler inequality. If $J(f)>0$, then Dubuc's equality characterization for Pr\'ekopa--Leindler implies that equality holds if and only if $R_u f$ is a translate of $f$, equivalently, if and only if there exists $z\in\R^n$ such that $R_u f(x)=f(x-z)$ for all $x\in\R^n$. 

(viii) By Remark \ref{linear-level-sets} and Lemma \ref{reflection-lemma}(iii), for every $t>0$ we have
\begin{align*}
\lev_{\geq t}\tau_u^{-\infty}f
=&\lev_{\geq t}\left(\tfrac{1}{2}\times_{-\infty}f\oplus_{-\infty}\tfrac{1}{2}\times_{-\infty}R_u f\right)
=\frac{1}{2}\lev_{\geq t}f+\frac{1}{2}\lev_{\geq t}R_u f\\
&=\frac{1}{2}\lev_{\geq t}f+\frac{1}{2}R_u(\lev_{\geq t}f)
=\tau_u(\lev_{\geq t}f).
\end{align*}
\end{proof}

%%%%%%%%%%%%%%%%%%%%%%
\subsection{The layer cake definition of a Minkowski symmetral}

The \emph{layer cake decomposition} of a function $f:\R^n\to[0,\infty)$ is  given by
\begin{equation}\label{layer-cake-fn-rep}
f(x)=\int_0^\infty \mathbbm{1}_{\lev_{\geq t}f}(x)\,dt.
\end{equation}

We now recall the definition of a Minkowski symmetral from the introduction.

\begin{definition}\label{layer-cake-symm}
Let $f\in\mathcal{C}_{-\infty}(\R^n)$ and let $u\in\Sp$. The \emph{(layer cake) Minkowski symmetral} $\widetilde{\tau}_u f$ of $f$ about the hyperplane $H=u^\perp$ is defined by
\begin{equation}\label{layer-cake-def1}
\widetilde{\tau}_u f(x) :=\int_0^\infty \mathbbm{1}_{\tau_u(\lev_{\geq t}f)}(x)\,dt.
\end{equation}
\end{definition}

%Note that by the layer cake decomposition \eqref{layer-cake-fn-rep}, we have 
%\[
%\widetilde{\tau}_u f(x)=\int_0^\infty\mathbbm{1}_{\lev_{\geq t}\widetilde{\tau}_u f}(x)\,dt. 
%\]
%Comparing this with Definition \ref{layer-cake-symm}, we deduce that $\mathbbm{1}_{\lev_{\geq t}\widetilde{\tau}_u f}=\mathbbm{1}_{\tau_u(\lev_{\geq t}f)}$ for almost all $t$. It follows that the layer cake Minkowski symmetral coincides with the $(-\infty)$-Minkowski symmetral on quasiconcave functions.

The next corollary shows that the layer cake Minkowski symmetral coincides with the $(-\infty)$-Minkowski symmetral on quasiconcave functions.

\begin{corollary}
For every $f\in\mathcal{C}_{-\infty}(\R^n)$, we have $\widetilde{\tau}_u f=\tau_u^{-\infty}f$.
\end{corollary}

\begin{proof}
By Proposition \ref{mainProp}(viii) and the layer cake decomposition \eqref{layer-cake-fn-rep}, for every $x\in\R^n$ we have
\[
\tau_u^{-\infty}f(x)
=\int_0^\infty \mathbbm{1}_{\lev_{\ge t}\tau_u^{-\infty}f}(x)\,dt
=\int_0^\infty \mathbbm{1}_{\tau_u(\lev_{\ge t}f)}(x)\,dt
=\widetilde{\tau}_u f(x).
\]
\end{proof}

\begin{remark}
The linearity property in Remark \ref{linear-level-sets} is only guaranteed to hold when $\alpha=-\infty$. In general, for $\alpha \in(-\infty,0]$ and $\lambda\in(0,1)$ we merely have the  inclusion
\begin{equation}\label{contain}
\lev_{\geq t}(\lambda\cdot_\alpha f\star_\alpha (1-\lambda)\cdot_\alpha g)\supset \lambda\lev_{\geq t}f+(1-\lambda)\lev_{\geq t}g.
\end{equation}
See, for example, \cite[Remark 2.3]{BCF-2014}. In particular, by Lemma \ref{reflection-lemma}(iii) we find that
\[
\lev_{\geq t}(\tau_u^\alpha f)\supset \frac{1}{2}\lev_{\geq t}f+\frac{1}{2}\lev_{\geq t}R_u f=\frac{1}{2}\lev_{\geq t}f+\frac{1}{2}R_u\lev_{\geq t}f=\tau_u\lev_{\geq t}f.
\]
%where in the first equality we used
%\begin{align*}
%\lev_{\geq t}R_u f&=\{x\in\R^n:\, R_u f(x)\geq t\}=\{x\in\R^n:\, f(R_u(x))\geq t\}\\
%&=\{R_u(x):\, x\in\lev_{\geq t}f\}=R_u\lev_{\geq t}f.
%\end{align*}
\end{remark}

\begin{comment}
{\color{red}
\begin{remark}
A geometric proof of (vi) can be given in the case of log-concave functions ($\alpha=0$). The argument is the same as for the case of Steiner symmetrizations of functions in \cite{CGN-2018}. Let $\widetilde{H}=H\times\R$. From the definition of $\tau_{\widetilde{H}}(\cdot)$ we have
\[
\proj_{\widetilde{H}}\epi(\psi)=\proj_{\widetilde{H}}\tau_{\widetilde{H}}(\epi(\psi))=\tau_{\widetilde{H}}(\epi(\psi))\cap\widetilde{H},
\]
 This is equivalent to
\[
\proj_{\widetilde{H}}\hyp(f)=\proj_{\widetilde{H}}\hyp(\tau_H f)=\hyp(\tau_H f)\cap\widetilde{H}.
\]
Hence,
\[
\proj_H f=\proj_H\tau_H f=(\tau_H f)_{|H}
\]
where $(\tau_H f)_{|H}$ is the restriction of $\tau_H f$ to $H$.
\end{remark}
}
\end{comment}

\begin{remark}
For every $K\in\mathcal{K}^n$ and every $u\in\Sp$, it holds that  $S_u K\subset \tau_u K$ (see, for example, \cite[Corollary 7.3]{symm-in-geom}). Thus, for every  $f\in\mathcal{C}_{-\infty}(\R^n)$ and every $u\in\Sp$, we have $S_u f\leq \tau_u^{-\infty} f$ since
\begin{align*}
S_u f(x)=\int_0^\infty\mathbbm{1}_{S_u(\lev_{\geq t}f)}(x)\,dt
\leq \int_0^\infty\mathbbm{1}_{\tau_u(\lev_{\geq t}f)}(x)\,dt%=\widetilde{\tau}_u f(x)
=\tau_u^{-\infty}f(x)
\end{align*}
for all $x\in\R^n$.
\end{remark}
%%%%%%%%%%%%%%%%%%%%%%%%%%%%

\subsection{The Minkowski symmetral of a convex function}

The Minkowski symmetral $\tau_u\psi=\tau_H\psi$ of a convex function $\psi\in\mathrm{Conv}_{\rm c}(\R^n)$ may be defined as 
\begin{equation}
\tau_u\psi:=\frac{1}{2}\sq\psi\square\frac{1}{2}\sq R_u\psi. 
\end{equation}
Thus if $f=e^{-\psi}\in\LC_{\rm c}(\R^n)$, then $\tau_u f=e^{-\tau_u\psi}$.  Note that $\tau_u\psi\in\mathrm{Conv}_{\rm c}(\R^n)$ as the infimal convolution of lower semicontinuous, coercive, convex functions (see, e.g., \cite[Lemma 2.3]{Hofstatter-Schuster}). Since $\psi$ is measurable, its epigraph is also measurable and hence $\tau_{\widetilde{H}}\epi(\psi)$ is well-defined and measurable. Furthermore,  we have
\begin{equation}\label{epi-symm-def}
\epi(\tau_H\psi) = \frac{1}{2}\epi(\psi)+\frac{1}{2}\epi(R_H\psi)= \frac{1}{2}\epi(\psi)+\frac{1}{2}R_{\widetilde{H}}\epi(\psi)=\tau_{\widetilde{H}}\epi(\psi).
\end{equation}

%%%%%%%%%%%%%%%%%%%%%%%%%%%%%%%%%%

\subsubsection{The base function of an $\alpha$-Minkowski symmetral}

\begin{lemma}\label{base-symmetral-lemma}
Let $\alpha\in(-\infty,0]$ and $f\in\mathcal{C}_\alpha(\R^n)$. Then 
\[
\base_\alpha\tau_u^\alpha f=\tau_u\base_\alpha f.
\]
\end{lemma}

\begin{proof}
First, let $\alpha\in(-\infty,0)$ and $x\in\R^n$. From the definitions it follows that
\begin{align*}
(\tau_u^\alpha f)(x) &=\left(\tfrac{1}{2}\cdot_\alpha f\star_\alpha\tfrac{1}{2}\cdot_\alpha R_u f\right)(x)
=\sup_{y+z=x}\left[\tfrac{1}{2}f(2y)^\alpha+\tfrac{1}{2}R_u f(2z)^\alpha\right]^{1/\alpha}\\
&=\sup_{y+z=x}\left[1-\alpha\left(\tfrac{1}{2}\base_\alpha f(2y)+\tfrac{1}{2}\base_\alpha R_u f(2z)\right)\right]^{1/\alpha}\\
&=\left[1-\alpha\inf_{y+z=x}\left(\tfrac{1}{2}\base_\alpha f(2y)+\tfrac{1}{2}\base_\alpha R_u f(2z)\right)\right]^{1/\alpha}\\
&=\left[1-\alpha\left(\left(\tfrac{1}{2}\sq\base_\alpha f\right)\square\left(\tfrac{1}{2}\sq\base_\alpha R_u f\right)\right)(x)\right]^{1/\alpha}\\
&=\left[1-\alpha\left(\left(\tfrac{1}{2}\sq\base_\alpha f\right)\square\left(\tfrac{1}{2}\sq R_u\base_\alpha f\right)\right)(x)\right]^{1/\alpha}\\
&=(1-\alpha\tau_u\base_\alpha f(x))^{1/\alpha}
\end{align*}
where we also used the fact that  $\base_\alpha R_u f=R_u\base_\alpha f$. The conclusion now follows from \eqref{alpha-concave-rep}. The special case $\alpha=0$ is handled in the same way. %In the case $\alpha=0$, let $f=e^{-\psi}\in\LC(\R^n)$ where $\psi\in\mathrm{Conv}(\R^n)$. As before we have 
%\begin{align*}
 %   (\tau_u  f)(x) &= \left(\tfrac{1}{2}\cdot f\star\tfrac{1}{2}\cdot R_u f\right)(x)
  %  =\sup_{y+z=x}f(2y)^{1/2}R_u f(2z)^{1/2}\\
   % &=\exp\left(-\inf_{y+z=x}\left\{\tfrac{1}{2}\sq \psi(y)+\tfrac{1}{2}\sq R_u\psi(z)\right\}\right)
    %=\exp\left(-\left(\tfrac{1}{2}\sq\psi\square\tfrac{1}{2}\sq R_u\psi\right)(x)\right)
    %=e^{-\tau_u\psi(x)}
%\end{align*}
%where  we used $R_u f=R_u e^{-\psi}=e^{-R_u \psi}$.
\end{proof}
%%%%%%%%%%%%%%%%%%%%%%%%%%%%%%%%%%%%%%%

\subsection{The conjugate Minkowski symmetral of a convex function}

Next, we show that the Minkowski symmetrization of a convex function can be equivalently formulated in terms of a symmetrization of its Legendre--Fenchel (conjugate) transform. This construction has the advantage of working with usual addition rather than infimal convolution. 

\begin{definition}
Let $\psi\in\mathrm{Conv}_{\rm c}(\R^n)$ and $u\in\Sp$. The \emph{conjugate Minkowski symmetral} $\tau_u^*\psi:\R^n\to\R$ of $\psi$ is defined by
\begin{equation}
\tau_u^*\psi := \frac{1}{2}\mathcal{L}\psi+\frac{1}{2}R_u(\mathcal{L}\psi).
\end{equation}
\end{definition}

The conjugate Minkowski symmetral satisfies the following properties. 

\begin{lemma}\label{transforms-lemma}
Let $\psi\in\mathrm{Conv}_{\rm c}(\R^n)$, $m\in\mathbb{N}$ and $u\in\Sp$. Then the following properties hold true:
\begin{itemize}
\item[(i)] $\mathcal{L}(R_u\psi)=R_u(\mathcal{L}\psi)$.

\item[(ii)] $\mathcal{L}(\tau_u\psi)=\tau_u^*\psi$.

\item[(iii)] $\tau_u\psi=\mathcal{L}(\tau_u^*\psi)$.

\item[(iv)] For any ordered set    $u_1,\ldots,u_m\in\Sp$, we have  $\bigcirc_{i=1}^m\tau_{u_i}\psi=\mathcal{L}(\bigcirc_{i=1}^m\tau_{u_i}^*\psi)$.

\item[(v)] $\tau_u^*\psi$ is symmetric about $u^\perp$.
\end{itemize}
\end{lemma}

\begin{proof}
To show (i), note that for all $x,y\in\R^n$ and all  $u\in\Sp$, 
\begin{equation*}
\langle x,R_u(y)\rangle = \langle x,y-2\langle y,u\rangle u\rangle=\langle x,y\rangle-2\langle y,u\rangle\langle x,u\rangle
=\langle x-2\langle x,u\rangle u,y\rangle=\langle R_u(x),y\rangle.
\end{equation*}
Hence,
\begin{align*}
\mathcal{L}(R_u\psi)(x)
&=\sup_{y\in\R^n}\big(\langle x,y\rangle-R_u\psi(y)\big)
=\sup_{y\in\R^n}\big(\langle x,y\rangle-\psi(R_u(y))\big)\\
&=\sup_{y\in\R^n}\big(\langle x,R_u(y)\rangle-\psi(y)\big)
=\sup_{y\in\R^n}\big(\langle R_u(x),y\rangle-\psi(y)\big)\\
&=(\mathcal{L}\psi)(R_u(x))%=R_u((\mathcal{L}\psi)(x))
=R_u(\mathcal{L}\psi)(x).
\end{align*}

To prove (ii), note that by Lemma \ref{conjugate-properties}(iii) and the reflection property in part (i) we get
\begin{align*}
\mathcal{L}(\tau_u\psi) = \mathcal{L}(\tfrac{1}{2}\sq\psi\square\tfrac{1}{2}\sq R_u\psi)
=\frac{1}{2}\mathcal{L}\psi+\frac{1}{2}\mathcal{L}(R_u\psi)
=\frac{1}{2}\mathcal{L}\psi+\frac{1}{2}R_u(\mathcal{L}\psi)
=\tau_u^*\psi.
\end{align*}

Part (iii) follows from  (ii) and Lemma \ref{conjugate-properties}(i). Part (iv) follows from (iii) inductively. For part (v),   by the linearity of the reflection operator, for all $u\in\Sp$ and all $x\in\R^n$ we have
\[
R_u\tau_u^*\psi(x)=R_u\left(\frac{1}{2}(\mathcal{L}\psi)(x)+\frac{1}{2}R_u(\mathcal{L}\psi)(x)\right)
=\frac{1}{2}R_u(\mathcal{L}\psi)(x)+\frac{1}{2}(\mathcal{L}\psi)(x)=\tau_u^*\psi(x).
\]
This means that $\tau_u^*$ is symmetric about $u^\perp$.
\end{proof}

\begin{comment}
\begin{corollary}\label{successive-legendre-transform-cor}
For every $\psi\in\mathrm{Conv}_{\rm c}(\R^n)$ and all $u_1,\ldots,u_m\in\Sp$,
\begin{align*}
\mathcal{L}\left(\bigcirc_{i=1}^m\tau_{u_i}\psi\right)&=\textstyle\frac{1}{2^m}\sum_{\ell=0}^m\sum_{\{j_1,\ldots,j_\ell\}\subset\{1,\ldots,m\}}\mathcal{L}\left(\bigcirc_{i=1}^\ell R_{u_{j_i}}\psi\right).
\end{align*}
\end{corollary}
For example, if $m=2$ then
\begin{align*}
\mathcal{L}(\tau_{u_2}\tau_{u_1}\psi) &= \frac{1}{4}\left(\mathcal{L}\psi+\mathcal{L}(R_{u_1}\psi)+\mathcal{L}(R_{u_2}\psi)+\mathcal{L}(R_{u_2}R_{u_1}\psi)\right)\\
&=\frac{1}{4}\left(\mathcal{L}\psi+\mathcal{L}(\psi\circ R_{u_1}^{-1})+\mathcal{L}(\psi\circ R_{u_2}^{-1})+\mathcal{L}(\psi\circ R_{u_1}^{-1}R_{u_2}^{-1})\right).
\end{align*}
\end{comment}

%%%%%%%%%%%%%%%%%%%%%%%%%%%%%%%%%%%%%%%%%%%%%
\section{Convergence of successive Minkowski symmetrizations}\label{section-successive-symmetrizations}

In the realm of convex bodies in $\R^n$, successive Minkowski symmetrizations of a given body $K$ can be made to converge to a ball with the same mean width as $K$  in the Hausdorff metric (see, e.g., \cite{GruberBook,SchneiderBook}). Similarly, as we show in this section, Minkowski symmetrizations of a coercive convex function $\psi\in\mathrm{Conv}_{\rm c}(\R^n)$ can be made to epi-converge to a radial convex function, which we call the ``reflectional epi-symmetrization'' of $\psi$.   

\begin{theorem}
\label{thm:full-sequence-minkowski}
Let $\psi \in \mathrm{Conv}_{\rm c}(\R^n)$. Choose unit vectors $u_1,\dots,u_m \in \Sp$ such that
the closed subgroup of $\mathrm{O}(n)$ generated by the reflections $R_{u_1},\dots,R_{u_m}$ is all of $\mathrm{O}(n)$. Define a cyclic sequence $\{\psi_k\}_{k\geq 0}\subset \mathrm{Conv}_{\rm c}(\R^n)$ by
\[
\psi_0:=\psi,
\qquad
\psi_{k+1}:=\tau_{u_{((k \bmod m)+1)}}\psi_k
\quad\text{for }k\geq 0.
\]
Then there exists a unique radial function $\psi_{\rm sym}\in\mathrm{Conv}_{\rm c}(\R^n)$ such that
\[
\mathcal{L}\psi_{\rm sym}(x)=\int_{\mathrm{O}(n)}\mathcal{L}\psi(\vartheta^{-1}x)\,d\eta(\vartheta),
\quad x\in\R^n,
\]
where $\eta$ denotes Haar probability measure on $\mathrm{O}(n)$, and
\[
\psi_k \stackrel{\epi}{\longrightarrow} \psi_{\rm sym}\quad\text{as }k\to\infty.
\]
%In particular, if $f=e^{-\psi}\in \LC_{\rm c}(\R^n)$ and $f_k:=e^{-\psi_k}$, then
%\[
%f_k \stackrel{{\rm hypo}}{\longrightarrow}f_{\rm sym}^{(0)}:=e^{-\psi_{\rm sym}}.
%\]
\end{theorem}

We refer to the function $\psi_{\rm sym}$ as the \emph{reflectional epi-symmetrization} of $\psi$.

\begin{remark}
    The integral appearing in the definition of $\mathcal{L}\psi_{\rm sym}$ is the usual spherical average. Indeed, for each $x\neq o$, the push-forward of Haar probability measure under the map $\vartheta\mapsto \vartheta^{-1}x$ is the normalized surface area measure on the sphere $|x|\Sp$. Hence
\[
\mathcal{L}\psi_{\rm sym}(x)=\int_{\Sp}\mathcal{L}\psi(|x|\omega)\,d\sigma(\omega),\qquad x\neq o,
\]
and for $x=o$ the same formula is understood as the value $\mathcal{L}\psi(o)$.
\end{remark}

\begin{remark}
The existence of finite families of reflections generating a dense subgroup of $\mathrm{O}(n)$ is known; see, for example, \cite[Proposition 4.2]{Burchard-Chambers-Dranovski-2017} and the related discussion in \cite{Bianchi-Gardner-Gronchi-2022}. For completeness, one may also see this directly as follows. For $n=1$, we have $\mathrm{O}(1)=\{\pm 1\}$, and for $n\geq 2$, we can take
    \begin{align*}
        r_0&:=R_{e_1};\\
        r_i&=R_{\frac{e_i-e_{i+1}}{\sqrt{2}}} \quad\text{for }i=1,\ldots,n-1;\\
        r_\theta&:=R_w,\quad\text{where }w=(\cos(\theta/2),\sin(\theta/2),0,\ldots,0)\quad\text{and}\quad\theta/\pi\text{ is irrational.}
    \end{align*}
    Then:
    \begin{itemize}
        \item $R_{(e_i-e_{i+1})/\sqrt{2}}$ swaps $e_i$ and $e_{i+1}$, so these reflections generate all permutation matrices.

        \item The product $R_w R_{e_1}$ is a rotation by angle $\theta$ in the $e_1e_2$-plane.

        \item Since $\theta/\pi$ is irrational, the cyclic subgroup generated by that rotation is dense in $\mathrm{SO}(2)$ on the $e_1e_2$-plane.

        \item Conjugating by permutation matrices generates all coordinate-plane rotations.

        \item The coordinate-plane rotations generate $\mathrm{SO}(n)$. 

        \item Since $r_0=R_{e_1}$ is a reflection, adjoining it to $\mathrm{SO}(n)$ generates all of $\mathrm{O}(n)$.
    \end{itemize} 
\end{remark}

\vspace{1mm}

For $\alpha\in(-\infty,0)$ and $f=(1-\alpha\psi)^{1/\alpha}\in\mathcal{C}_\alpha(\R^n)$, we refer to $f_{\rm sym}^{(\alpha)}:=(1-\alpha\psi_{\rm sym})^{1/\alpha}$ as the \emph{reflectional hypo-symmetrization} of $f$ associated with the parameter $\alpha$. Note that  $f_{\rm sym}^{(\alpha)}$ is radial and $f_{\rm sym}^{(\alpha)}\in\mathcal{C}_\alpha(\R^n)$. In the special case $\alpha=0$,  for $f=e^{-\psi}\in\LC_{\rm c}(\R^n)$ we set $f_{\rm sym}^{(0)}:=e^{-\psi_{\rm sym}}$. This notation records the possible dependence on the chosen parameter $\alpha$ when $f$ belongs to more than one class $\mathcal{C}_\alpha(\R^n)$.

\begin{corollary}\label{convergence-hypo-symmetral}
Let $\alpha\in(-\infty,0]$. For every $f\in\mathcal{C}_\alpha(\R^n)$, there is a sequence of Minkowski symmetrizations  of $f$ which hypo-converges to its reflectional hypo-symmetrization $f_{\rm sym}^{(\alpha)}\in\mathcal{C}_{\alpha}(\R^n)$.
\end{corollary}

\begin{remark}
    The function $\psi_{\rm sym}$ is  related to the rotational epi-symmetrization studied by Colesanti, Ludwig and Mussnig \cite{CLM-Hadwiger1}, and $f_{\rm sym}^{(\alpha)}$ is related to the aforementioned mean width rearrangement of Salani \cite{Salani-2015}. Both of those symmetrizations are obtained by convergent sequences of functional versions of Hadwiger rotation means. It would be interesting to determine whether Salani's radial symmetrization and the present reflectional hypo-symmetrization coincide on their common domain; we do not address this question here. 
\end{remark}

\begin{remark}
    A different symmetrization procedure for log-concave functions was recently studied in \cite{Hoehner-Chasioti}, where Blaschke symmetrizations, rather than Minkowski symmetrizations, are shown to converge, up to translations, to a radial log-concave function called the ``mean Blaschke symmetral''.
\end{remark}

\begin{proof}[Proof of Theorem \ref{thm:full-sequence-minkowski}]
Set $\phi:=\mathcal{L}\psi \in \mathrm{Conv}_{\rm c}(\R^n)$. For each $i\in\{1,\dots,m\}$, define a probability measure $\nu_i$ on $\mathrm{O}(n)$ by $\nu_i:=\frac12(\delta_\mathrm{Id}+\delta_{R_{u_i}})$, where $\mathrm{Id}$ is the identity in $\mathrm{O}(n)$. For any function $g:\R^n\to[0,\infty)$ and any Borel probability measure $\nu$ on $\mathrm{O}(n)$, define the \emph{averaging operator}
\[
T_\nu g(x):=\int_{\mathrm{O}(n)} g(\vartheta^{-1}x)\,d\nu(\vartheta),
\quad x\in\R^n,
\]
whenever the integral is well-defined. First note that for every $i\in\{1,\ldots,m\}$, we have
\[
T_{\nu_i}\phi(x)=\frac{1}{2}\phi(x)+\frac{1}{2}\phi(R_{u_i}x)
=\frac{1}{2}\mathcal{L}\psi(x)+\frac{1}{2}R_{u_i}(\mathcal{L}\psi)(x).
\]
Hence, by Lemma~\ref{transforms-lemma}(ii),
\[
\mathcal{L}(\tau_{u_i}\psi)=T_{\nu_i}(\mathcal{L}\psi)=T_{\nu_i}\phi.
\]
Inductively, if we let $\omega_k$ denote the probability measure corresponding to the first
$k$ symmetrizations, then
\[
\mathcal{L}\psi_k=T_{\omega_k}\phi,
\quad k\geq 0.
\]

Now define the $m$-fold convolution
\[
\mu:=\nu_m*\nu_{m-1}*\cdots *\nu_1,
\]
and for $j\in\{0,1,\dots,m-1\}$ let
\[
\lambda_0:=\delta_\mathrm{Id},
\qquad
\lambda_j:=\nu_j*\nu_{j-1}*\cdots *\nu_1, \quad j\geq 1.
\]
Then for every $k\geq 0$ and every $j\in\{0,1,\dots,m-1\}$, we have
\[
\omega_{km+j}=\lambda_j * \mu^{*k}
\]
where $\mu^{*k}$ denotes the $k$-fold convolution of $\mu$. Indeed, by definition for any Borel probability measures $\sigma,\tau$ on $\mathrm{O}(n)$ we have
\begin{align*}
    T_{\sigma}(T_{\tau}\phi)&=\int_{\mathrm{O}(n)}\int_{\mathrm{O}(n)}\phi(\theta^{-1}\vartheta^{-1}x)\,d\sigma(\theta)\,d\tau(\vartheta)
    =\int_{\mathrm{O}(n)}\int_{\mathrm{O}(n)}\phi((\vartheta\theta)^{-1}x)\,d\sigma(\theta)\,d\tau(\vartheta)\\
    &=\int_{\mathrm{O}(n)}\phi(\rho^{-1}x)\,d(\sigma*\tau)(\rho)=T_{\sigma*\tau}\phi(x).
\end{align*}
Now one symmetrization applies $T_{\nu_i}$, so after the first $k$ steps the measure $\omega_k$ is the convolution of the corresponding step measures, with the latest one on the left:
\[
\omega_0=\delta_\mathrm{Id},\quad \omega_1=\nu_1,\quad\omega_2=\nu_2*\nu_1,\quad\ldots
\]
Hence one full cycle of $m$ steps yields $\omega_m=\nu_m*\cdots*\nu_1=\mu$. After $k$ full cycles, we thus get $\mu^{*k}$. Then applying the next $j$ steps contributes the prefix measure $\lambda_j=\nu_j*\cdots*\nu_1$. Therefore, $\omega_{km+j}=\lambda_j*\mu^{*k}$ and
\[
\mathcal{L}\psi_{km+j}=T_{\omega_{km+j}}\phi=T_{\lambda_j*\mu^{*k}}\phi.
\]
In other words,
\[
\mathcal{L}\psi_{km+j}(x)
=\int_{\mathrm{O}(n)}\phi(\vartheta^{-1}x)\,d(\lambda_j*\mu^{*k})(\vartheta).
\]

Next, observe that $\mathrm{Id}\in \supp(\mu)$ and $R_{u_i}\in \supp(\mu)$ for every $i$, since in the product defining $\mu$ one may choose the reflection $R_{u_i}$ in the $i$th factor and the identity in all remaining factors. Consequently, the closed subgroup generated by $\supp(\mu)$
is all of $\mathrm{O}(n)$. We will use the following result on adapted aperiodic random walks on compact groups  (see  \cite{Borda-2021,Stromberg}).

\begin{theorem}\cite[Theorem A]{Borda-2021}\label{thm:stromberg}
Let $\mu$ be a regular Borel probability measure on a compact Hausdorff group $G$. Then $\mu^{*k}$ converges weakly to the Haar measure $\eta$ as $k\to\infty$ if and only if $\mu$ is adapted (i.e., not supported on a proper closed subgroup) and strictly aperiodic (i.e., not concentrated on a coset of a proper closed normal subgroup).
\end{theorem}

We apply this to the compact Hausdorff group  $G=\mathrm{O}(n)$. The step measure $\mu$ has $\mathrm{Id}\in\supp(\mu)$, and the reflections generate a dense subgroup of $\mathrm{O}(n)$. Hence $\mu$ is adapted. Also, if $\supp(\mu)\subset gN$ for some proper closed normal subgroup $N$, then $\mathrm{Id}\in \supp(\mu)\subset gN$, so $gN=N$, or equivalently, $g\in N$. Hence $\supp(\mu)\subset gN=N$, which is a proper normal subgroup. This is a contradiction to the adaptedness of $\mu$. Thus, $\mu$ is strictly aperiodic. Therefore, by Theorem \ref{thm:stromberg} we have
\[
\mu^{*k} \longrightarrow \eta
\quad\text{weakly on }\mathrm{O}(n),
\]
where $\eta$ is the Haar probability measure on $\mathrm{O}(n)$. Hence, for each fixed
$j\in\{0,1,\dots,m-1\}$,
\[
\lambda_j * \mu^{*k} \longrightarrow \eta
\quad\text{as }k\to\infty,
\]
because convolution on the left by the fixed measure $\lambda_j$ is continuous in the weak topology and $\lambda_j * \eta=\eta$.

Next, we define
\[
\widetilde\phi(x):=\int_{\mathrm{O}(n)}\phi(\vartheta^{-1}x)\,d\eta(\vartheta),\quad x\in\R^n.
\]
Then $\widetilde\phi$ is a proper,
lower semicontinuous, convex function, and $o\in \operatorname{int}(\dom (\widetilde\phi))$ (see Lemma \ref{lem:widetilde-phi} below). Thus there
exists a unique function $\psi_{\rm sym}\in\mathrm{Conv}_{\rm c}(\R^n)$ such that
\[
\mathcal{L}\psi_{\rm sym}=\widetilde\phi.
\]
Since the Haar measure is bi-invariant, for every $\rho\in \mathrm{O}(n)$ we have
\[
\widetilde\phi(\rho x)=
\int_{\mathrm{O}(n)}\phi(\vartheta^{-1}\rho x)\,d\eta(\vartheta)
=\int_{\mathrm{O}(n)}\phi(\theta^{-1}x)\,d\eta(\theta)
=\widetilde\phi(x),
\]
so $\widetilde\phi$ is $\mathrm{O}(n)$-invariant. %We claim that this implies that $\widetilde\phi$ is radial. By the Cartan--Dieudonn\'e theorem, every orthogonal transformation $\rho\in\mathrm{O}(n)$ is a product of at most $n$ reflections. Thus, for any $\rho\in\mathrm{O}(n)$, there exist reflections $R_{u_1},\ldots,R_{u_m}$ with $m\leq n$ such that $\rho=R_{u_m}\circ\cdots\circ R_{u_1}$. Therefore, by the invariance of $\widetilde{\phi}$ under reflections,
%\[
%\widetilde{\phi}(\rho x)=\widetilde{\phi}\left(\bigcirc_{i=1}^m R_{u_i}(x)\right)=\widetilde{\phi}(x)
%\]
%for any $x\in\R^n$. Hence $\psi_{\rm sym}$ is also radial.

It remains to prove the epi-convergence of the full sequence. To do this, we use some tools from \cite{Hoehner-Mussnig}. Set
\[
r(\psi):=\sup\{s\geq 0:\, sB_n\subset \dom(\mathcal{L}\psi)\}\in(0,\infty].
\]

\medskip
\noindent
\emph{Step 1: Convergence on the interior region $|x|<r(\psi)$.} Fix $x\in\R^n$ with $|x|<r(\psi)$. Then $\vartheta^{-1}x\in \operatorname{int}(r(\psi)B_n)\subset \dom(\phi)$ for every
$\vartheta\in \mathrm{O}(n)$. Since $\phi$ is convex, it is continuous on $\operatorname{int}(\dom(\phi))$, so the map $h_x(\vartheta):=\phi(\vartheta^{-1}x)$ 
is continuous on the compact group $\mathrm{O}(n)$, and thus it is bounded. Therefore, for every fixed $j\in\{0,\dots,m-1\}$, we have
\[
\mathcal{L}\psi_{km+j}(x)=
\int_{\mathrm{O}(n)} h_x(\vartheta)\,d(\lambda_j*\mu^{*k})(\vartheta)
\longrightarrow
\int_{\mathrm{O}(n)} h_x(\vartheta)\,d\eta(\vartheta)
=\widetilde\phi(x)
=\mathcal{L}\psi_{\rm sym}(x).
\]

\medskip
\noindent
\emph{Step 2: Convergence on the exterior region $|x|>r(\psi)$ when $r(\psi)<\infty$.} Assume now that $r(\psi)<\infty$, and fix $x\in\R^n$ with $|x|>r(\psi)$. Since $\operatorname{int}(r(\psi)B_n)$ is the largest open centered ball contained in $\dom(\phi)$ and $\dom(\phi)$ is
convex, there exists $u_0\in\Sp$ such that
\[
\dom(\phi)\subset \{y\in\R^n:\,\langle y,u_0\rangle\leq r(\psi)\}.
\]
Hence the set
\[
V_x:=\{\vartheta\in \mathrm{O}(n): \langle \vartheta^{-1}x,u_0\rangle>r(\psi)\}
\]
is a nonempty open subset of $\mathrm{O}(n)$, and for every $\vartheta\in V_x$ we have $\vartheta^{-1}x\notin \dom(\phi)$, hence 
$\phi(\vartheta^{-1}x)=+\infty$. Since $V_x$ is nonempty and open, its Haar measure satisfies $\eta(V_x)>0$. By weak
convergence and the Portmanteau theorem (see, e.g., Theorem 2.1 in Chapter 1 of \cite{billingsley1999convergence}, or \cite[Theorem 1.1]{Zaharopol-2018}),
\[
\liminf_{k\to\infty} (\lambda_j*\mu^{*k})(V_x)\geq \eta(V_x)>0.
\]
Thus, for all sufficiently large $k$,
\[
(\lambda_j*\mu^{*k})(V_x)>0.
\]
Now, since the integrand $\phi(\vartheta^{-1}x)$ is identically $+\infty$ on $V_x$, it follows that
for all sufficiently large $k$,
\[
\mathcal{L}\psi_{km+j}(x)
=\int_{\mathrm{O}(n)}\phi(\vartheta^{-1}x)\,d(\lambda_j*\mu^{*k})(\vartheta)
=+\infty.
\]
On the other hand,
\[
\mathcal{L}\psi_{\rm sym}(x)=\widetilde\phi(x)=+\infty
\]
as well, since $\eta(V_x)>0$ and the same integrand is $+\infty$ on $V_x$. Thus
\[
\mathcal{L}\psi_{km+j}(x)\longrightarrow \mathcal{L}\psi_{\rm sym}(x)
\quad\text{for every }x\text{ with }|x|>r(\psi).
\]
If $r(\psi)=\infty$, then there is no exterior region and Step 2 is vacuous.

\medskip
Combining the two steps, we obtain the pointwise convergence
\[
\mathcal{L}\psi_k(x)\longrightarrow \mathcal{L}\psi_{\rm sym}(x)\quad\text{as }k\to\infty
\]
on the dense subset $\{x\in\R^n: |x|\neq r(\psi)\}$ 
(with the convention that this is all of $\R^n$ if $r(\psi)=\infty$).
Since $\mathcal{L}\psi_{\rm sym}$ is lower semicontinuous, convex, and has nonempty interior domain, Lemma~\ref{le:epi_conv_pointwise} implies that $\mathcal{L}\psi_{\rm sym}=\epilim_{k\in\mathbb{N}} \mathcal{L}\psi_k$.
Applying Lemma~\ref{conjugate-properties}(vi), we conclude that  $\psi_{\rm sym}=\epilim_{k\in\mathbb{N}}\psi_k$.
\end{proof}

\begin{comment}
\begin{lemma}\label{ref-epi-means-conv}
For every $\psi\in\mathrm{Conv}_{\rm c}(\R^n)$, there exists a sequence of Minkowski symmetrizations of $\psi$ which epi-converges to a radial function $\psi_{\rm sym}\in\mathrm{Conv}_{\rm c}(\R^n)$.
\end{lemma}

\begin{proof}
       In view of the identity \eqref{epi-symm-def}, the argument we follow is based on the classical geometric case (see, e.g., \cite[pp. 157--158]{G-Toth-book-2015}). Let $H_0,\ldots,H_{n-2}\in\Gr(n,n-1)$ be such that their mutual dihedral angles are irrational multiples of $\pi$. Apply the ordered sequence of Minkowski symmetrizations $\psi\mapsto\bigcirc_{i=0}^{n-2}\tau_{H_i}\psi$, and repeat this procedure cyclically to obtain the sequence $\{\psi_i\}_{i\geq 0}$ given by $\psi_0=\psi$, $\psi_{i+1}=\tau_{H_{i\pmod{(n-1)}}}\psi_i$, $i\geq 0$. Note that $\psi_{i+1}$ is symmetric with respect to $H_{i\pmod{(n-1)}}$, $i\geq 0$. By a functional version of Blaschke's selection theorem (see \cite[Theorem 2.15]{Mussnig-Li} and \cite[Theorems 4.18 and 7.6]{Rockafellar-Wets}), $\{\psi_i\}$ has a subsequence which epi-converges to a coercive convex function $\psi_{\rm sym}\in\mathrm{Conv}_{\rm c}(\R^n)$, and it is symmetric with respect to $H_1,\ldots,H_{n-1}$. Now by the irrationality condition on the mutual dihedral angles of the hyperplanes, it follows that $\psi_{\rm sym}$ is rotation invariant. Thus, $\psi_{\rm sym}$ is the convex function obtained from $\psi$ by Minkowski symmetrizations.
\end{proof}
\end{comment}

\begin{lemma}\label{lem:widetilde-phi}
    Given $\psi\in\mathrm{Conv}_{\rm c}(\R^n)$, let $\phi:=\mathcal{L}\psi$ and 
    \[
    \widetilde{\phi}(x):=\int_{\mathrm{O}(n)}\phi(\vartheta^{-1}x)\,d\eta(\vartheta)
    \]
    where $\eta$ is the Haar probability measure on $\mathrm{O}(n)$. Then $\widetilde{\phi}$ is convex, lower semicontinuous, proper, and $o\in\operatorname{int}(\dom(\widetilde{\phi}))$.
\end{lemma}

\begin{proof}
\begin{comment}We will use two facts. Since $\psi\in\mathrm{Conv}_{\rm c}(\R^n)$ is proper, there exists $x_0\in\R^n$ such that $\psi(x_0)<\infty$. Since $\psi$ is coercive, there exists $R_0>0$ such that $|y|\geq R_0$ implies $\psi(y)\geq 0$. Choose $R\geq\max\{R_0,|x_0|\}$. Then $x_0\in B_R=\{y\in\R^n:\,|y|\leq R\}$, so there exists $y_0\in B_R$ such that $\psi(y_0)<\infty$. Since $\psi$ is lower semicontinuous, it attains its minimum on the compact set $B_R$; call it
\[
m_R:=\min_{y\in B_R}\psi(y)\in\R.
\]
Hence
\[
\inf_{y\in\R^n}\psi(y)\geq\min\{m_R,0\}>-\infty,
\]
so if we set $m_\psi:=\inf_{y\in\R^n}\psi(y)$, then $m_\psi\in\R$ is finite. Consequently, for every $x\in\R^n$, 
\[
\phi(x)=\sup_{y\in\R^n}\left(\langle x,y\rangle-\psi(y)\right)\geq \langle x,x_0\rangle-\psi(x_0)
\]
\end{comment}
Let us first verify that $\widetilde{\phi}$ is well-defined. Since $\psi\in\mathrm{Conv}_{\rm c}(\R^n)$, it is convex, proper, lower semicontinuous and coercive. Thus $\phi$ is proper, convex and lower semicontinuous. Since $\psi$ is proper, there exists $y_0\in\R^n$ such that $\psi(y_0)<\infty$. Thus for every $z\in\R^n$,
\[
\phi(z)=\sup_{y\in\R^n}\left(\langle z,y\rangle-\psi(y)\right)\geq\langle z,y_0\rangle-\psi(y_0)>-\infty.
\]
In particular, for each fixed $x\in\R^n$ and each $\vartheta\in\mathrm{O}(n)$, by the Cauchy--Schwarz inequality
\begin{equation}\label{phi-lower-bd-proof}
\phi(\vartheta^{-1}x)\geq \langle\vartheta^{-1}x,y_0\rangle-\psi(y_0) \geq -|x|\cdot|y_0|-\psi(y_0)>-\infty.
\end{equation}
Thus the integrand $\vartheta\mapsto\phi(\vartheta^{-1}x)$ is bounded below by a finite constant depending on $x$. Since $\phi$ is lower semicontinuous, hence Borel measurable, and $\vartheta\mapsto \vartheta^{-1}x$ is continuous, the composite map $\vartheta\mapsto \phi(\vartheta^{-1}x)$ is Borel measurable on the compact group $\mathrm{O}(n)$. Therefore, $\widetilde{\phi}(x)\in(-\infty,+\infty]$ is well-defined for $x\in\R^n$.

Next, we prove that $\widetilde{\phi}$ is convex. Fix $x,y\in\R^n$ and $\lambda\in[0,1]$. For each $\vartheta\in\mathrm{O}(n)$,
    \[
\vartheta^{-1}(\lambda x+(1-\lambda)y)=\lambda\vartheta^{-1}x+(1-\lambda)\vartheta^{-1}y.
    \]
    Thus, since $\phi$ is convex,
    \begin{align*}
        \phi\left(\vartheta^{-1}(\lambda x+(1-\lambda)y)\right) &=\phi\left(\lambda\vartheta^{-1}x+(1-\lambda)\vartheta^{-1}y\right)
        \leq \lambda\phi(\vartheta^{-1}x)+(1-\lambda)\phi(\vartheta^{-1}y).
    \end{align*}
    Integrating both sides of this inequality, we obtain
    \begin{align*}
        \widetilde{\phi}\left(\lambda x+(1-\lambda)y\right) &=\int_{\mathrm{O}(n)}\phi\left(\vartheta^{-1}(\lambda x+(1-\lambda)y)\right)d\eta(\vartheta)\\
        &\leq \lambda\int_{\mathrm{O}(n)}\phi(\vartheta^{-1}x)\,d\eta(\vartheta)+(1-\lambda)\int_{\mathrm{O}(n)}\phi(\vartheta^{-1}y)\,d\eta(\vartheta)\\
        &=\lambda\widetilde{\phi}(x)+(1-\lambda)\widetilde{\phi}(y).
    \end{align*}
    Therefore, $\widetilde{\phi}$ is convex.

    Next, we show that $o\in\operatorname{int}(\dom(\widetilde{\phi}))$. Since $\psi$ is coercive, there exist constants $a>0$ and $b\in\R$ such that 
    \begin{equation}\label{eq:psi-coercive-lower-bd}
        \psi(y) \geq a|y|-b
    \end{equation}
    for all $y\in\R^n$ (see, e.g., \cite[Lemma 2.5]{Colesanti-Fragala-variational}). Applying this and the Cauchy--Schwarz inequality, we get that for any $x,y\in\R^n$,
    \[
\langle x,y\rangle-\psi(y)\leq |x|\cdot|y|-(a|y|-b)=(|x|-a)|y|+b.
    \]
    Therefore, if $|x|<a$, then 
    \[
\phi(x)=\sup_{y\in\R^n}\left(\langle x,y\rangle-\psi(y)\right) \leq b<\infty.
    \]
    Hence, the open ball $\operatorname{int}(B(o,a))=\{x\in\R^n:\,|x|<a\}$ is contained in $\dom(\phi)$. Moreover, for any $x\in \operatorname{int}(B(o,a))$, we have $|\vartheta^{-1}x|=|x|<a$ for all $\vartheta\in\mathrm{O}(n)$. Hence $\phi(\vartheta^{-1}x)\leq b<\infty$ for all $\vartheta\in\mathrm{O}(n)$, so
    \[
\widetilde{\phi}(x)=\int_{\mathrm{O}(n)}\phi(\vartheta^{-1}x)\,d\eta(\vartheta)\leq b<\infty.
    \]
    Thus $\operatorname{int}(B(o,a))\subset\dom(\widetilde{\phi})$; in particular, $o\in\operatorname{int}(\dom(\widetilde{\phi}))$.

    Since $\widetilde{\phi}$ is finite on the open ball $\operatorname{int}(B(o,a))$, it is not identically $+\infty$. Also, by \eqref{phi-lower-bd-proof} for every $x\in\R^n$ we have
    \[
\widetilde{\phi}(x) = \int_{\mathrm{O}(n)}\phi(\vartheta^{-1}x)\,d\eta(\vartheta)\geq -|x|\cdot|y_0|-\psi(y_0)>-\infty,
    \]
    so $\widetilde{\phi}$ never takes the value $-\infty$. This proves that $\widetilde{\phi}$ is proper.

    It remains to prove that $\widetilde{\phi}$ is lower semicontinuous. Let $x_k\to x$ in $\R^n$. Then the set $\{x_k\}_{k=1}^\infty\cup\{x\}$ is bounded, i.e., there exists $S>0$ such that $|x_k|\leq S$ for all $k\in\mathbb{N}$ and $|x|\leq S$. Thus by \eqref{phi-lower-bd-proof}, for every $\vartheta\in\mathrm{O}(n)$,
    \[
\phi(\vartheta^{-1}x_k) \geq -S|y_0|-\psi(y_0)\text{ for every }k\in\mathbb{N},\text{ and } \phi(\vartheta^{-1}x) \geq -S|y_0|-\psi(y_0).
    \]
    Set $C:=S|y_0|+|\psi(y_0)|+1$. Then for every $\vartheta\in{\rm O}(n)$,
    \[
    \phi(\vartheta^{-1}x_k)+C\geq 0\text{ for every }k\in\mathbb{N},\text{ and } \phi(\vartheta^{-1}x)+C\geq 0.
    \]
    By continuity, for each fixed $\vartheta\in\mathrm{O}(n)$ we have $\vartheta^{-1}x_k\to\vartheta^{-1}x$ as $k\to\infty$. Since $\phi$ is lower semicontinuous,
    \[
\phi(\vartheta^{-1}x)\leq\liminf_{k\to\infty}\phi(\vartheta^{-1}x_k).
    \]
    Hence,
      \[
0\leq \phi(\vartheta^{-1}x)+C\leq\liminf_{k\to\infty}\left(\phi(\vartheta^{-1}x_k)+C\right).
    \]
    Applying Fatou's lemma to the nonnegative, measurable functions $\vartheta\mapsto \phi(\vartheta^{-1}x_k)+C$, we obtain
     \[
\int_{\mathrm{O}(n)}\left(\phi(\vartheta^{-1}x)+C\right) d\eta(\vartheta)\leq\liminf_{k\to\infty}\int_{\mathrm{O}(n)}\left(\phi(\vartheta^{-1}x_k)+C\right)d\eta(\vartheta).
    \]
    Since $\eta$ is a probability measure on $\mathrm{O}(n)$, this gives us
    \[
\int_{\mathrm{O}(n)}\phi(\vartheta^{-1}x)\,d\eta(\vartheta)\leq\liminf_{k\to\infty}\int_{\mathrm{O}(n)}\phi(\vartheta^{-1}x_k)\,d\eta(\vartheta).
    \]
    In other words,
    \[
\widetilde{\phi}(x) \leq \liminf_{k\to\infty}\widetilde{\phi}(x_k).
    \]
    Thus, $\widetilde{\phi}$ is lower semicontinuous. 
\end{proof}

The next result shows that an $\alpha$-concave function and its reflectional hypo-symmetrization have the same $\alpha$-mean width when $\alpha\in(-\tfrac{2}{n-1},0]$. 

\begin{proposition}\label{prop:sym-preserve-mw}
    Fix $\alpha\in(-\tfrac{2}{n-1},0]$, and let $f\in\mathcal{C}_\alpha(\R^n)$. Then 
    \[
w_\alpha(f_{\rm sym}^{(\alpha)})=w_\alpha(f).
    \]
\end{proposition}

\begin{proof}
    We only prove the case $\alpha\in(-\tfrac{2}{n-1},0)$; the case $\alpha=0$ follows mutatis mutandis. Set $\psi:=\base_\alpha f$ and $f_{\rm sym}^{(\alpha)}=(1-\alpha \psi_{\rm sym})^{1/\alpha}$. By definition, $h_f^{(\alpha)}=\mathcal{L}(\base_\alpha f)=\mathcal{L}\psi$. Likewise, $h_{f_{\rm sym}^{(\alpha)}}^{(\alpha)}=\mathcal{L}\psi_{\rm sym}$. By Theorem \ref{thm:full-sequence-minkowski}, 
    \[
\mathcal{L}\psi_{\rm sym}=\int_{\mathrm{O}(n)}\mathcal{L}\psi(\vartheta^{-1}x)\,d\eta(\vartheta).
    \]
    Hence
    \[
h_{f_{\rm sym}^{(\alpha)}}^{(\alpha)}(x)=\int_{\mathrm{O}(n)}h_f^{(\alpha)}(\vartheta^{-1}x)\,d\eta(\vartheta).
    \]

    Since $\psi\in\mathrm{Conv}_{\rm c}(\R^n)$, there exists $y_0\in\R^n$ such that $\psi(y_0)<\infty$. Thus,
    \[
h_f^{(\alpha)}(x)=\mathcal{L}\psi(x)\geq\langle x,y_0\rangle-\psi(y_0)\geq-|x||y_0|-\psi(y_0).
    \]
    Set $g(x):=|y_0||x|+|\psi(y_0)|+1$. Then
    \[
h_f^{(\alpha)}(x)+g(x)\geq 1+|\psi(y_0)|-\psi(y_0)\geq 1.
    \]
    In particular, $h_f^{(\alpha)}+g\geq 0$. Since $\alpha\in(-\tfrac{2}{n-1},0)$, the function $(1+|x|)\rho_\alpha(x)$ is integrable on $\R^n$. This implies $g(x)\in L^1(\rho_\alpha(x)\,dx)$.

    Therefore, if $\alpha\in(-\tfrac{2}{n-1},0)$ then
    \begin{align*}
        w_\alpha(f_{\rm sym}^{(\alpha)})+\int_{\R^n}g(x)\rho_\alpha(x)\,dx &= \int_{\R^n}\left(h_{f_{\rm sym}^{(\alpha)}}^{(\alpha)}(x)+g(x)\right)\rho_\alpha(x)\,dx\\
        &=\int_{\R^n}\left(\int_{\mathrm{O}(n)}h_f^{(\alpha)}(\vartheta^{-1}x)\,d\eta(\vartheta)+g(x)\right)\rho_\alpha(x)\,dx\\
        &=\int_{\R^n}\int_{\mathrm{O}(n)}\left(h_f^{(\alpha)}(\vartheta^{-1}x)+g(x)\right)\rho_\alpha(x)\,d\eta(\vartheta)\,dx
    \end{align*}
   where we used the fact that $\eta$ is a probability measure on ${\rm O}(n)$. Since the integrand is nonnegative, by Tonelli's theorem we get
    \[
w_\alpha(f_{\rm sym}^{(\alpha)})+\int_{\R^n}g(x)\rho_\alpha(x)\,dx =\int_{\mathrm{O}(n)}\int_{\R^n}\left(h_f^{(\alpha)}(\vartheta^{-1}x)+g(x)\right)\rho_\alpha(x)\,dx\,d\eta(\vartheta).
    \]
     Now fix $\vartheta\in\mathrm{O}(n)$, and make the change of variables $z=\vartheta^{-1}x$. Since $\vartheta$ is an orthogonal transformation, we have $dx=dz$ and $|x|=|z|$. Hence $\rho_\alpha(x)=\rho_\alpha(z)$ and $g(x)=g(z)$. Therefore,
    \[
\int_{\R^n}\left(h_f^{(\alpha)}(\vartheta^{-1}x)+g(x)\right)\rho_\alpha(x)\,dx=\int_{\R^n}\left(h_f^{(\alpha)}(z)+g(z)\right)\rho_\alpha(z)\,dz.
    \]
    This no longer depends on $\vartheta$, so 
    \begin{align*}
        w_\alpha(f_{\rm sym}^{(\alpha)})+\int_{\R^n}g(x)\rho_\alpha(x)\,dx =\int_{\R^n}\left(h_f^{(\alpha)}(z)+g(z)\right)\rho_\alpha(z)\,dz=w_\alpha(f)+\int_{\R^n}g(z)\rho_\alpha(z)\,dz.
    \end{align*}
    Subtracting the finite quantity $\int_{\R^n}g\rho_\alpha$ from both sides, we finally obtain
    \[
w_\alpha(f_{\rm sym}^{(\alpha)})=w_\alpha(f).
    \]
\end{proof}

%%%%%%%%%%%%%%%%
\subsection{Convergence to unconditional functions}

Klartag and Milman \cite[Lemma 2.3]{Klartag-Milman-2003} proved that any set in $\R^n$ can be transformed into an unconditional set (i.e., a set which is symmetric with respect to the vectors $\{u_1,\ldots,u_n\}$ in an orthonormal basis of $\R^n$) using $n$ Steiner symmetrizations. 
A function $f:\R^n\to(-\infty,+\infty]$ is \emph{unconditional} if for every $(x_1,\ldots,x_n)\in\R^n$, we have 
\[
f(x_1,\ldots,x_n)=f(|x_1|,\ldots,|x_n|).
\]
Lin \cite[Theorem 3.4]{Lin-2017} proved that any coercive convex function $\psi:\R^n\to(-\infty,+\infty]$ can be transformed into an unconditional function using $n$ Steiner symmetrizations. We prove an analogue for $\alpha$-Minkowski symmetrizations in the following

\begin{theorem}\label{unconditional}
Let $\alpha\in[-\infty,0]$. Every $\alpha$-concave function  $f\in\mathcal{C}_\alpha(\R^n)$  can be transformed into an unconditional function $\overline{f}\in\mathcal{C}_\alpha(\R^n)$ using $n$ $\alpha$-Minkowski symmetrizations.
\end{theorem}

To prove this result, we will use the following lemma, which is the analogue of \cite[Lemma 3.4]{Lin-2017} for conjugate Minkowski symmetrizations of convex functions.

\begin{lemma}\label{unconditional-lemma}
Let $u_1,u_2\in\Sp$ and $\langle u_1,u_2\rangle=0$. If $\psi\in\mathrm{Conv}_{\rm c}(\R^n)$ is symmetric about $u_1^\perp$, then $\tau_{u_2}^*\psi$ is symmetric about both $u_1^\perp$ and $u_2^\perp$. 
\end{lemma}

\begin{proof}
First note that since $\langle u_1,u_2\rangle=0$, for every $x\in\R^n$ we have
\begin{align*}
R_{u_2}R_{u_1}(x) &= R_{u_1}(x)-2\langle R_{u_1}(x),u_2\rangle u_2\\
&=x-2\langle x,u_1\rangle u_1-2(\langle x,u_2\rangle u_2-2\langle x,u_1\rangle\langle u_1,u_2\rangle u_2)\\
&=x-2(\langle x,u_1\rangle u_1+\langle x,u_2\rangle u_2).
\end{align*}
The last line is symmetric in $u_1$ and $u_2$, so swapping the roles of $u_1$ and $u_2$ in the previous computation we get  $R_{u_2}R_{u_1}=R_{u_1}R_{u_2}$. 

By Lemma \ref{transforms-lemma}(v), we know that $\tau_{u_2}^*\psi$ is symmetric about $u_2^\perp$, so we only need to prove that $\tau_{u_2}^*\psi$ is symmetric about $u_1^\perp$. Indeed, for every $x\in\R^n$,
\begin{align*}
R_{u_1}\tau_{u_2}^*\psi(x) &=R_{u_1}\left(\frac{1}{2}(\mathcal{L}\psi)(x)+\frac{1}{2}R_{u_2}(\mathcal{L}\psi)(x)\right)\\
&=\frac{1}{2}R_{u_1}(\mathcal{L}\psi)(x)+\frac{1}{2}R_{u_1}R_{u_2}(\mathcal{L}\psi)(x)\\
&=\frac{1}{2}R_{u_1}(\mathcal{L}\psi)(x)+\frac{1}{2}R_{u_2}R_{u_1}(\mathcal{L}\psi)(x)\\
&=\frac{1}{2}\mathcal{L}(R_{u_1}\psi)(x)+\frac{1}{2}R_{u_2}(\mathcal{L}(R_{u_1}\psi))(x)\\
&=\frac{1}{2}(\mathcal{L}\psi)(x)+\frac{1}{2}R_{u_2}(\mathcal{L}\psi)(x)\\
&=\tau_{u_2}^*\psi(x),
\end{align*}
which shows that $\tau_{u_2}^*\psi$ is symmetric about $u_1^\perp$. In the second line, we used the linearity of the reflection operator, in the third line we used $R_{u_1}R_{u_2}=R_{u_2}R_{u_1}$, in the fourth line we used Lemma \ref{transforms-lemma}(i), and in the fifth line we used the fact that $\psi$ is symmetric about $u_1^\perp$.
\end{proof}

From Lemma \ref{unconditional-lemma} we shall deduce the following result, which is the analogue of \cite[Lemma 3.4]{Lin-2017} for $\alpha$-Minkowski symmetrizations.

\begin{corollary}\label{unconditional-cor}
Let $u_1,u_2\in\Sp$ and $\langle u_1,u_2\rangle=0$. 
\begin{itemize}
\item[(i)] If $\psi\in\mathrm{Conv}_{\rm c}(\R^n)$ is symmetric about $u_1^\perp$, then $\tau_{u_2}\psi$ is symmetric about both $u_1^\perp$ and $u_2^\perp$. 

\item[(ii)] Given $\alpha\in[-\infty,0]$, 
if $f\in\mathcal{C}_\alpha(\R^n)$ is symmetric about $u_1^\perp$,  then $\tau_{u_2}^\alpha f$ is symmetric about both $u_1^\perp$ and $u_2^\perp$.

%\item[(iii)] If $f\in\mathcal{C}_{-\infty}(\R^n)$ is symmetric about $u_1^\perp$,  then $\tau_{u_2}^{-\infty} f$ is symmetric about both $u_1^\perp$ and $u_2^\perp$.
\end{itemize}
\end{corollary}

\begin{proof}
(i) Since $\tau_{u_2}\psi$ is symmetric about $u_2^\perp$, we only need to show that it is symmetric about $u_1^\perp$. Indeed, for all $x\in\R^n$,
\begin{align*}
R_{u_1}\tau_{u_2}\psi(x) &= R_{u_1}\mathcal{L}(\tau_{u_2}^*\psi)(x)
=\mathcal{L}(R_{u_1}\tau_{u_2}^*\psi)(x)
=\mathcal{L}(\tau_{u_2}^*\psi)(x)
=\tau_{u_2}\psi(x).
\end{align*}
The first equality follows from Lemma \ref{transforms-lemma}(iii), the second from Lemma \ref{transforms-lemma}(i), and the third  follows from Lemma \ref{unconditional-lemma}. This proves that $\tau_{u_2}\psi$ is symmetric about $u_1^\perp$.

(ii) Next, let $\alpha\in(-\infty,0)$ and $f=(1-\alpha\psi)^{1/\alpha}\in\mathcal{C}_\alpha(\R^n)$, where $\psi=\base_\alpha f$. By Lemma \ref{lem:base-alpha-convc}, $\base_\alpha f\in\mathrm{Conv}_{\rm c}(\R^n)$. Note that $f$ is symmetric about $u^\perp$ if and only if $\base_\alpha f$ is symmetric about $u^\perp$. The result now follows from (i). The case $\alpha=0$ is handled similarly.

Finally, let $\alpha=-\infty$. For every $t>0$, we have
\begin{align*}
\lev_{\geq t}R_{u_1}\tau_{u_2}^{-\infty}f
&=R_{u_1}(\lev_{\geq t}\tau_{u_2}^{-\infty}f)
=R_{u_1}\left(\frac{1}{2}\lev_{\geq t}f+\frac{1}{2}\lev_{\geq t}R_{u_2}f\right)\\
&=\frac{1}{2}R_{u_1}(\lev_{\geq t}f)+\frac{1}{2}R_{u_1}(\lev_{\geq t}R_{u_2}f)
=\frac{1}{2}\lev_{\geq t}R_{u_1}f+\frac{1}{2}\lev_{\geq t}R_{u_1}R_{u_2}f\\
&=\frac{1}{2}\lev_{\geq t}R_{u_1}f+\frac{1}{2}\lev_{\geq t}R_{u_2}R_{u_1}f\\
&=\frac{1}{2}\lev_{\geq t}f+\frac{1}{2}\lev_{\geq t}R_{u_2}f
=\lev_{\geq t}\tau_{u_2}^{-\infty}f.
\end{align*}
Thus by the layer cake principle, 
\[
R_{u_1}\tau_{u_2}^{-\infty}f(x)=\int_0^\infty\mathbbm{1}_{\lev_{\geq t}R_{u_1}\tau_{u_2}^{-\infty}f}(x)\,dt
=\int_0^\infty\mathbbm{1}_{\lev_{\geq t}\tau_{u_2}^{-\infty}f}(x)\,dt=\tau_{u_2}^{-\infty}f(x)
\]
for all $x\in\R^n$.
\begin{comment}Note that if $K\in\mathcal{K}^n$ is symmetric about $u_1^\perp$, then $\tau_{u_2}K$ is symmetric about $u_1^\perp$ and $u_2^\perp$ since
\[
R_{u_1}(\tau_{u_2}K) = R_{u_1}\left(\frac{1}{2}K+\frac{1}{2}R_{u_2}K\right)=\frac{1}{2}R_{u_1} K+\frac{1}{2}R_{u_1}R_{u_2}K
=\frac{1}{2}R_{u_1} K+\frac{1}{2}R_{u_2}R_{u_1}K
=\tau_{u_2}K.
\]
Thus if $f\in\mathcal{C}_{-\infty}(\R^n)$ is symmetric about $u_1^\perp$
\end{comment}
\end{proof}

We can now prove Theorem \ref{unconditional}.

\begin{proof}[Proof of Theorem \ref{unconditional}]
Let $\alpha\in[-\infty,0]$ and let $\{u_1,\ldots,u_n\}$ be an orthonormal basis of $\R^n$. Applying Corollary \ref{unconditional-cor} inductively, we deduce that the iterated $\alpha$-Minkowski symmetral $\bigcirc_{i=1}^n\tau_{u_i}^\alpha f$ is symmetric about $u_i^\perp$ for $i=1,\ldots,n$. Therefore, $f$ can be transformed into the unconditional function $\overline{f}:=\bigcirc_{i=1}^n\tau_{u_i}^\alpha f$ using $n$ $\alpha$-Minkowski symmetrizations. Finally, note that $\overline{f}\in\mathcal{C}_\alpha(\R^n)$ by  Lemma \ref{closure-lemma}(ii).
\end{proof}

%%%%%%%%%%%%%%%%%%%%%%%%%%%%%%%%%%%%%%%%%%%%
\section{Applications}\label{polytopes-section}

\subsection{Motivation}

A classical result in convex geometry states that among all convex bodies of a given mean width, a  Euclidean ball is hardest to approximate by inscribed polytopes with a given number of vertices. In this section, we present a functional analogue of this fact. The proof will use Minkowski symmetrizations of functions.

For convex bodies $K,P\in\mathcal{K}^n$ with $P\subset K$,  the \emph{mean width difference}  $\Delta_w:\mathcal{K}^n\times\mathcal{K}^n\to[0,\infty)$ is defined by 
\begin{equation}\label{MWdeviation}
\Delta_w(P,K)=w(K)-w(P). 
\end{equation} 
The mean width difference has applications to the approximation of convex bodies by polytopes. For some examples, we refer the reader to  \cite{BH-2022, BHK, FlorianMetric, Glasauer-Gruber, Hoehner-survey, Ludwig1999,McClure-Vitale-1975,Vitale1985} and the references therein.

Given $K\in\mathcal K^n$ and an integer $N\geq n+1$,  let $\mathscr P_N(K)$  denote the set of all polytopes contained in $K$  with at most $N$ vertices.  A solution $\widehat{P}_N$ to the optimization problem
\[
\inf_{P\in\mathscr P_N(K)}\Delta_w(P,K)
\]
is called a {\it best-approximating polytope} of $K$ with respect to the mean width difference. Best-approximating polytopes are guaranteed to exist by the Blaschke selection theorem and the continuity of the mean width functional with respect to the Hausdorff metric.

Next, we present the classical result which says that the Euclidean ball is hardest to approximate by inscribed polytopes in the mean width difference. 

\begin{theorem}\label{mainThm-old}
Fix integers $n\geq 2$ and $N\geq n+1$, and consider  the functional  $F_{N}:\mathcal K^n\to[0,\infty)$ defined by
\begin{equation}\label{ourfunctional}
F_{N}(K):=\inf_{P\in\mathscr P_N(K)}\frac{\Delta_w(P,K)}{w(K)}.
\end{equation}
Then $F_{N}$ attains its maximum at Euclidean balls. In other words, for every $K\in\mathcal{K}^n$ we have
\begin{equation}\label{mainresult}
F_{N}(K)\leq F_{N}(B_n).
\end{equation}
\end{theorem}

Theorem \ref{mainThm-old} may be viewed as an extension of a result of Schneider \cite{Schneider1971}, where, in particular, the planar result ($n=2$) was proved; in this case, the ball is the unique maximizer \cite{FlorianPrachar, Schneider1971} (see also \cite[p. 189]{FlorianExtremum}). Florian  established a planar version of Theorem \ref{mainThm-old} for convex polygons in arbitrary position \cite{AFlorian1992}, and remarked  that in the inscribed case, the extension from $n=2$ to general dimensions $n\geq 2$ is possible (see also \cite{FlorianExtremum}). Theorem \ref{mainThm-old} is also a special case of a  result of Bucur, Fragal\`a and Lamboley \cite[Theorem 2.1]{BucurFragalaLamboley} (see Section \ref{extremal-section}). To the best of our knowledge, for arbitrary dimensions $n\geq 3$ it is not known if Euclidean balls are the only maximizers in Theorem \ref{mainThm-old}. 

An analogous version of Theorem \ref{mainThm-old} was also proved by Schneider \cite{Schneider-1967} for circumscribed polytopes with a given number of facets. The construction and results in this subsection should be compared with those in the dual  outer approximation theory developed in \cite{Hoehner-Mussnig}. The present section treats the inner approximation problem corresponding to inscribed polytopes and finite sets of break points, whereas \cite{Hoehner-Mussnig} treats a dual outer approximation problem modeled on  supporting affine data.

%%%%%%%%%%%%%%%%%%%%%%%%%%%%%

\subsection{Inner (log-)linearizations}\label{inner-linearizations-subsection}

In the setting of convex functions, the analogue of a polytope inscribed in a convex body with $N$ vertices is called an inner linearization with $N$ break points. Inner linearizations  have found applications to convex optimization  and stochastic programming; see, for example,  \cite{Bertsekas2011, BL-stochastic-book}.  Let us now describe their   construction. Given   $\psi\in\mathrm{Conv}(\R^n)$  and   a finite point set $\mathbf{X}_N:=\{(x_i,y_i)\}_{i=1}^N\subset\epi(\psi)$ with $x_i\in\dom(\psi)$ and $y_i\geq\psi(x_i)$, consider  the convex hull of the vertical rays with these endpoints:
\begin{equation}\label{union-rays}
\textstyle \conv^{\uparrow}\mathbf{X}_N := \conv\left(\bigcup_{i=1}^N\{(x_i,y): y\geq y_i\}\right).%\psi(x_i)\}).
\end{equation}
Here the  superscript $\uparrow$ indicates that the points $\mathbf{X}_N$ lie in the epigraph of $\psi$. The function $p_{\mathbf{X}_N}$ whose epigraph is the set \eqref{union-rays} is called an \emph{inner linearization} of $\psi$. 

\begin{definition}
Given $\psi\in\mathrm{Conv}(\R^n)$ and a finite set of points $\mathbf{X}_N:=\{(x_1,y_1),\ldots,(x_N,y_N)\}\subset\epi(\psi)$, the \emph{inner linearization} $p_{\mathbf{X}_N}:\R^n\to(-\infty,+\infty]$ is defined by
\[
p_{\mathbf{X}_N}(x):=\begin{cases}
\textstyle\min\left\{\sum_{i=1}^N \lambda_i y_i : \substack{\sum_{i=1}^N\lambda_i=1,\,\lambda_1,\ldots,\lambda_N\geq 0,\\\sum_{i=1}^N\lambda_i x_i=x}\right\}, &\text{if }x\in\conv\{x_1,\ldots,x_N\}\\
+\infty, &\text{otherwise}.
\end{cases}
\]
\end{definition}
 In particular, $p_{\mathbf{X}_N}\geq\psi$ and $p_{\mathbf{X}_N}$  is the unique piecewise affine function with $\dom(p_{\mathbf{X}_N})=\conv\{x_1,\ldots,x_N\}$ and $\epi(p_{\mathbf{X}_N})=\conv^{\uparrow}\mathbf{X}_N$. Moreover, $p_{\mathbf{X}_N}\in\mathrm{Conv}_{\rm c}(\R^n)$. A point $x\in\dom(\psi)$ is called a \emph{break point} of $p_{\mathbf{X}_N}$ if $(x,p_{\mathbf X_N}(x))\in\R^n\times \R$ is an extreme point of $\epi(p_{\mathbf{X}_N})$. Note that if $\mathbf{X}_N\subset\epi(\psi)$ is finite, then the associated inner linearization is the greatest lower semicontinuous convex minorant of the extended-real-valued function that takes the prescribed values at the points of $\mathbf{X}_N$ and $+\infty$ everywhere else. For more background, we refer the reader to \cite{RockafellarBook,Rockafellar-Wets}. 

Given $\psi\in\mathrm{Conv}_{\rm c}(\R^n)$ and an integer $N\geq n+2$, let 
$\mathscr{C}_N(\psi)$ be the set of all inner linearizations of $\psi$ with at most $N$ break points.  For $f=e^{-\psi}\in\LC_{\rm c}(\R^n)$, we also set $\mathscr{P}_N(f):=\{e^{-p}:\,p\in\mathscr{C}_N(\psi)\}$. We refer to a  function  $q\in\mathscr{P}_N(f)$ as an \emph{inner log-linearization} of $f$  (see also \cite{PB-2020,Rinott}). Note that if $q=e^{-p}\in\mathscr{P}_N(f)$, then   $q\leq f$ and $q\in\LC_{\rm c}(\R^n)$. A point $x\in\supp(f)$ is called a \emph{break point} of $q$ if $x$ is a break point of $p$.

 Just as any  convex body $K\in\mathcal{K}^n$ may be approximated by an inscribed polytope $P\subset K$ with at most $N$ vertices in the mean width difference $w(K)-w(P)$, a convex function $\psi\in\mathrm{Conv}_{\rm c}(\R^n)$ may be approximated by an inner linearization $p\geq\psi$ with at most $N$ break points under the mean width difference $w_0(e^{-\psi})-w_0(e^{-p})$.  

Since the function $x\mapsto e^{-x}$ is a bijection from $\R$ to $(0,\infty)$, one may use inner linearizations of $\psi$ to define a functional version of the convex hull operation for a given set of points in the hypograph of $f=e^{-\psi}$. 

\begin{definition}\label{def-f-cvx-hull}
Given $\psi\in\mathrm{Conv}(\R^n)$ and a  set $S\subset\epi(\psi)$, let $\exp(-S):=\{(x,e^{-y})\in\R^n\times\R:\, (x,y)\in S\}\subset\hyp(f)$.   For $f=e^{-\psi}\in\LC(\R^n)$ and a finite set $\mathbf{Y}_N\subset\hyp(f)\cap(\R^n\times(0,\infty))$, set
\[
\conv_{\downarrow} \mathbf{Y}_N:=\exp\left(-\conv^{\uparrow}\left(-\log \mathbf{Y}_N\right)\right).%e^{-\conv^{\uparrow}\left(-\log \mathbf{Y}_N\right)}
\]
where $-\log \mathbf{Y}_N:=\{(x,-\log y)\in\R^n\times\R:\,(x,y)\in \mathbf{Y}_N\}\subset\epi(\psi)$.     
\end{definition}

Geometrically, this construction means: (i) first map $\mathbf{Y}_N$ to its preimage in $\epi(\psi)$ via the function $(x,y)\mapsto (x,-\log y)$, then (ii) take the convex hull of the vertical rays of the resulting set, then (iii) apply the exponential map $(x,y)\mapsto (x,e^{-y})$ to this set. Its  image, contained in $\hyp(f)$, is $\conv_{\downarrow}\mathbf{Y}_N$. Inner log-linearizations have been  studied in the literature before under different names; see, for example,  \cite{Hoehner-Mussnig,Hoehner-Novaes,PB-2020,Rinott}. 

Note that for any finite set of points $\mathbf{Y}_N\subset\hyp(f)\cap(\R^n\times(0,\infty))$, the set $\conv_{\downarrow}\mathbf{Y}_N$ is well-defined.  In particular, if   $\mathbf{Y}_N:=\{(x_1,y_1),\ldots,(x_N,y_N)\}\subset \hyp(f)$, then we have 
\[
\conv_{\downarrow}\mathbf{Y}_N=\exp\left(-\conv\left(\textstyle{\bigcup_{i=1}^N}\left\{(x_i,-\log y):\, 0<y\leq y_i\right\}\right)\right).%=e^{-\conv\left(\bigcup_{i=1}^N\{(x_i,-\log y):\, y\leq y_i\}\right)}. 
\]

\noindent The construction of $\conv_{\downarrow}\mathbf{Y}_N$ is illustrated below for the function $f(x)=e^{-x^2}$.%, where $\psi(x)=x^2$.

 \begin{center}
     \includegraphics[scale=0.7]{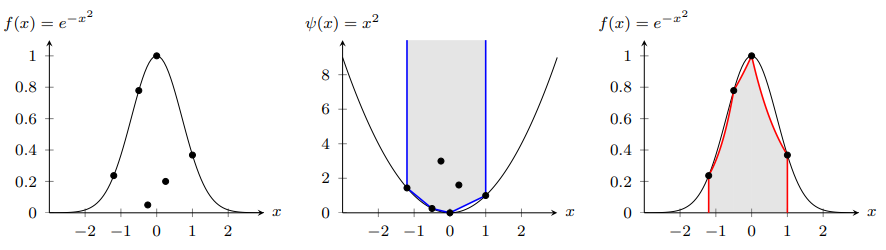}
 \end{center}

 \noindent{\bf\footnotesize Figure 1.} {\footnotesize In the left figure, a finite point set $\mathbf{Y}_N$ is contained in $\hyp(f)$. In the middle figure, the set $\mathbf{Y}_N$ has been mapped to its preimage $\mathbf{X}_N:=-\log\mathbf{Y}_N\subset\epi(\psi)$; the shaded  region there is $\conv^{\uparrow}\mathbf{X}_N$, and the blue curve in its  boundary is the graph of the corresponding inner linearization $p_{\mathbf{X}_N}$. In the right figure, $\conv^{\uparrow}\mathbf{X}_N$ has been mapped back to its image in $\hyp(f)$; the shaded region there is $\conv_{\downarrow}\mathbf{Y}_N$ and the red curve in its boundary is the graph of the corresponding inner log-linearization $\exp(-p_{\mathbf{X}_N})$, which has four break points.} 

 \subsubsection{Inner $\alpha$-linearizations}\label{rmk:alpha-remark-inner-alpha-linearization}
 
More generally, for any $\alpha\in(-\infty,0]$, one may define the set $\conv_\downarrow\mathbf{Y}_N$ with respect to an $\alpha$-concave function $f$ in an analogous way. This leads to the concept of ``inner $\alpha$-linearizations'' of an $\alpha$-concave function $f\in\mathcal{C}_\alpha(\R^n)$. The following definition is from \cite{Hoehner-Novaes}. 

\begin{definition}
    Let $\alpha\in(-\infty,0)$. For a convex function $\psi\in\mathrm{Conv}(\R^n)$ and a set $S\subset\epi(\psi)$, let
    \begin{align*}
(1-\alpha S)^{1/\alpha}&:=\left\{(x,(1-\alpha y)^{1/\alpha})\in\R^n\times\R:\,(x,y)\in S\right\}.
    \end{align*}
    For $f=(1-\alpha\psi)^{1/\alpha}\in\mathcal{C}_\alpha(\R^n)$ with $\psi=\base_\alpha f$ and a finite set $\mathbf{Y}_N\subset\hyp(f)\cap(\R^n\times(0,\infty))$, set
    \[
    \conv_\downarrow\mathbf{Y}_N:=\left[1-\alpha\conv^{\uparrow}\left(\frac{1-\mathbf{Y}_N^\alpha}{\alpha}\right)\right]^{1/\alpha}
    \]
    where $\frac{1-\mathbf{Y}_N^\alpha}{\alpha}:=\left\{(x,\frac{1-y^\alpha}{\alpha}):\,(x,y)\in \mathbf{Y}_N\right\}$.
\end{definition}
This construction has also been considered before in the literature; for more background, we refer the reader to \cite{Hoehner-Novaes,PB-2020,Rinott}. It is  also complementary to the construction of generalized outer linearizations studied in  \cite{Hoehner-Mussnig}, where the finite data are taken in the dual variable rather than as break points in the primal epigraph.

For $\alpha\in(-\infty,0)$ and $f=(1-\alpha\psi)^{1/\alpha}\in\mathcal{C}_\alpha(\R^n)$, let  $\mathscr{P}_{\alpha, N}(f):=\{(1-\alpha p)^{1/\alpha}:\,p\in\mathscr{C}_N(\psi)\}$. We call a function $q\in\mathscr{P}_{\alpha, N}(f)$ an \emph{inner $\alpha$-linearization} of $f$. Note that if $q\in\mathscr{P}_{\alpha,N}(f)$ then $q\leq f$ and $q\in\mathcal{C}_\alpha(\R^n)$. As before, a point $x\in\supp(f)$ is called a \emph{break point} of $q=(1-\alpha p)^{1/\alpha}\in\mathscr{P}_{\alpha,N}(f)$ if $x$ is a break point of $p\in\mathscr{C}_N(\psi)$. Taking the case $\alpha=0$ to be understood in the limiting sense, we set  $\mathscr{P}_{0,N}(f):=\mathscr{P}_N(f)$.

%%%%%%%%%%%%%%%%%%%%%%%
\subsection{An extremal property of the  reflectional hypo-symmetrization}\label{extremal-hypo-sec}

%Given $f\in\LC_{\rm c}(\R^n)$ and an integer $N\geq n+2$,  there exists a \emph{best-approximating inner log-linearization} $\widehat{p}_N$ for which
%\[
%d_w(f,\widehat{p}_N)=\inf_{q\in\mathscr{P}_N(f)}\{w_0(f)-w_0(q)\}.
%\]
%This follows from a  compactness argument (see, for example, \cite[Section 2.3]{Mussnig-Li}). 
The main result of this section is a functional analogue of Theorem \ref{mainThm-old}:

\begin{theorem}\label{mainThm-new}
Let $\alpha\in(-\frac{2}{n-1},0]$. Fix integers $n\geq 1$ and $N\geq n+2$, and consider  the functional $G_N: \mathcal{C}_\alpha(\R^n)\to[0,\infty)$ defined by
\begin{equation}\label{ourfunctional-new}
G_N(f):=\sup_{q\in\mathscr P_{\alpha,N}(f)}w_\alpha(q).
\end{equation}
Then for every 
 $f\in\mathcal{C}_\alpha(\R^n)$, we have
\begin{equation}\label{mainresult-new}
G_{N}(f)\geq G_{N}(f_{\rm sym}^{(\alpha)}).
\end{equation}
\end{theorem}
The proof will use the lower semicontinuity of the $\alpha$-mean width, which is why we restrict $\alpha$ to lie in the interval $(-\tfrac{2}{n-1},0]$. We will present the proof of  the log-concave case $\alpha=0$ in full detail. The proof for the case $\alpha\in(-2/(n-1),0)$ follows mutatis mutandis after replacing the usual Asplund sum operations $\star$ and $\cdot$ by $\star_\alpha$ and $\cdot_\alpha$, using the corresponding properties of these operations with convex base functions, and mapping inner linearizations of $\psi$ to inner $\alpha$-linearizations of $f=(1-\alpha\psi)^{1/\alpha}$ instead of inner log-linearizations of $f=e^{-\psi}$.

%\begin{remark}
%Observe that $d_w(f,\widehat{p}_N) = w_0(f)-G_N(f)$.
%\end{remark}
\vspace{1mm}

The main step in the proof of Theorem \ref{mainThm-new} is the following 

\begin{lemma}\label{lem1}
Let $n,N\in\mathbb{N}$ be such that $N\geq n+2$. Given $f,g\in\LC_{\rm c}(\R^n)$, let $p_{f\star g}\in\mathscr{P}_N(f\star g)$. Then there exists $p_f\in\mathscr{P}_N(f)$ and $p_g\in\mathscr{P}_N(g)$ such that $p_{f\star g}\leq p_f\star p_g$.
\end{lemma}

\begin{remark}
Let $\psi\in\mathrm{Conv}_{\rm c}(\R^n)$ and let $\mathbf{X}_N\subset\epi(\psi)$ be finite. If $\mathbf{X}_N'$ denotes the set of those points of $\mathbf{X}_N$ whose vertical rays contain the extreme points of $\epi(p_{\mathbf{X}_N})$, then $p_{\mathbf{X}_N}=p_{\mathbf{X}_N'}$. 
Indeed, deleting from a finite vertical ray representation the generators that are not extreme points of the resulting epigraph does not change the closed convex hull of the vertical rays. Thus, an inner linearization is completely determined by its break points. 
\end{remark}

\begin{remark}
In the proof below, we will also use the following consequence: given a convex function $\psi\in\mathrm{Conv}_{\rm c}(\R^n)$ and an inner linearization $\ell\in\mathscr{C}_N(\psi)$ with at most $N$ break points, it is possible to choose a set $\mathbf{X}_N$ consisting of exactly $N$ points in $\epi(\psi)$ such that $p_{\mathbf{X}_N}=\ell$. Indeed, suppose $\ell$ has $m\leq N$ break points. Then there exists a finite set $\mathbf{X}_m=\{(x_i,y_i)\}_{i=1}^m\subset\epi(\psi)$ such that $\epi(\ell)=\conv^{\uparrow}\mathbf{X}_m$. If $m=N$, then there is nothing to prove, so assume that $m<N$. Choose an index $i\in\{1,\ldots,m\}$, say $i=1$, and add $N-m$ further points on the same vertical ray, e.g., $(x_1,y_1+1),\ldots,(x_1,y_1+N-m)$. These extra points are redundant, because their vertical rays are contained in $\{(x_1,t):\,t\geq y_1\}\subset \conv^{\uparrow}\mathbf{X}_m$. Thus,  enlarging $\mathbf{X}_m$ in this way does not change $\conv^{\uparrow}\mathbf{X}_m$ and hence does not change $\ell$.
\end{remark}

\begin{proof}[Proof of Lemma \ref{lem1}]\label{lem1-pf}
 Let $f=e^{-\varphi}$ and $g=e^{-\psi}$, where $\varphi,\psi\in\mathrm{Conv}_{\rm c}(\R^n)$ are given. Let  $(v_1,t_1),\ldots,(v_N,t_N)\in\epi(\varphi\square\psi)$, where $v_1,\ldots,v_N\in\dom(\varphi\square\psi)$ and $t_i\geq  (\varphi\square\psi)(v_i)$ for every $i$. Set  $P_{\varphi\square\psi}:=\conv^{\uparrow}\{(v_i,t_i)\}_{i=1}^N$. Define the inner linearization  $\phi\in\mathscr{C}_N(\varphi\square\psi)$ by $\epi(\phi)=P_{\varphi\square\psi}$ and the inner log-linearization $p_{f\star g}\in\mathscr{P}_N(f\star g)$ by $p_{f\star g}=e^{-\phi}$. Note that $f\star g=e^{-(\varphi\square\psi)}$. Since $\epi(\varphi\square\psi)=\epi(\varphi)+\epi(\psi)$, by definition of the Minkowski sum of sets, for each $i\in\{1,\ldots,N\}$ there exists $(a_i,r_i)\in\epi(\varphi)$ and $(b_i,s_i)\in\epi(\psi)$ such that $(v_i,t_i)=(a_i,r_i)+(b_i,s_i)$, where $a_i\in\dom(\varphi),b_i\in\dom(\psi),r_i\geq\varphi(a_i)$ and $s_i\geq\psi(b_i)$. Set
 \begin{align*}
     Q_{\varphi}&:=\conv^{\uparrow}\{(a_1,r_1),\ldots,(a_N,r_N)\}\\
     R_{\psi}&:=\conv^{\uparrow}\{(b_1,s_1),\ldots,(b_N,s_N)\},
 \end{align*}
 and define the functions $q_{\varphi}$ and $r_{\psi}$ by $\epi(q_{\varphi})=Q_{\varphi}$ and $\epi(r_{\psi})=R_{\psi}$, respectively. 
 By construction, $q_{\varphi}\in\mathscr{C}_N(\varphi)$ and $r_{\psi}\in\mathscr{C}_N(\psi)$.  
 
 Let $(x,t)\in P_{\varphi\square\psi}$. By \eqref{union-rays}, there exist  $\lambda_1,\ldots,\lambda_k\in[0,1]$ with $\sum_{i=1}^k\lambda_i=1$ such that 
 \begin{equation}\label{inclusion1}
(x,t)=\sum_{j=1}^k\lambda_j(v_j,t_j)%=\sum_{j=1}^k\lambda_j[(a_j,r_j)+(b_j,s_j)]
=\left(\sum_{j=1}^k\lambda_j a_j,\sum_{j=1}^k\lambda_j r_j\right)+\left(\sum_{j=1}^k\lambda_j b_j,\sum_{j=1}^k\lambda_j s_j\right).
 \end{equation}
 %$u\in\mathbb{S}^{n}$ and $t\geq 0$ such that %$x=\sum_{i=1}^N\lambda_i v_i+tu$. Hence,
 %\[
%x=\sum_{i=1}^N\lambda_i(a_i+b_i)+tu=\left(\sum_{i=1}^N\lambda_i a_i+t_1 u_1\right)+\left(\sum_{i=1}^N\lambda_i b_i+t_2 u_2\right)\in Q_\varphi+R_\psi%\epi(\varphi)+\epi(\psi)
 %\]
%for some $u_1,u_2\in\mathbb{S}^n$ and $t_1,t_2\in[0,t]$ satisfying $u_1+u_2=u$ and $t_1+t_2=t$, respectively. 
Since $\varphi$ is convex and $r_j\geq\varphi(a_j)$, we have
\[
\varphi\left(\sum_{j=1}^k\lambda_j a_j\right)\leq \sum_{j=1}^k\lambda_j\varphi(a_j)\leq \sum_{j=1}^k\lambda_j r_j
\]
which implies $(\sum_j \lambda_j a_j,\sum_j\lambda_j r_j)\in\epi(q_\varphi)$. Similarly, $(\sum_j \lambda_j b_j,\sum_j\lambda_j s_j)\in\epi(r_\psi)$. Thus, by \eqref{inclusion1} we deduce that  $(x,t)\in\epi(q_\varphi)+\epi(r_\psi)=\epi(q_\varphi\square r_\psi)$. Since $(x,t)\in P_{\varphi\square\psi}$ was arbitrary, we have 
\begin{equation}\label{epi-inclusion}
\epi(\phi)= P_{\varphi\square\psi}\subset Q_\varphi+R_\psi=\epi(q_\varphi\square r_\psi).
\end{equation}
Now consider  $p_f:=e^{-q_\varphi}\in\mathscr{P}_N(f)$ and $p_g:=e^{-r_\psi}\in\mathscr{P}_N(g)$. The inclusion \eqref{epi-inclusion} yields $\hyp(p_{f\star g})\subset\hyp(p_f\star p_g)$, which is equivalent to $p_{f\star g}\leq p_f\star p_g$. 
\end{proof}

The next result shows that a Minkowski symmetrization decreases the functional $G_{N}$.
 
\begin{lemma}\label{mainLemma}
Let $n,N\in\mathbb{N}$ be such that $N\geq n+2$. Let $f\in\LC_{\rm c}(\R^n)$ and $u\in\mathbb{S}^{n-1}$. For every  $q\in\mathscr{P}_N(\tau_u  f)$, there exists  $\widetilde{q}\in\mathscr{P}_N(f)$ such that $w_0(\widetilde{q})\geq w_0(q)$. Therefore, $G_{N}(f)\geq G_{N}(\tau_u  f)$. 
\end{lemma}

\begin{proof}
Let $q\in\mathscr{P}_N(\tau_u  f)$. By Lemma \ref{lem1}, there exist $q_1\in\mathscr{P}_N(\frac{1}{2}\cdot f)$ and $q_2\in\mathscr{P}_N(\frac{1}{2}\cdot R_u f)$ such that $q\leq q_1\star q_2$. Thus, for every $x\in\R^n$ we have $(2\cdot q_1)(x)=q_1^2(\frac{x}{2})\leq f(x)$ and $(2\cdot q_2)(x)=q_2^2(\frac{x}{2})\leq R_u f(x)$. It follows that $2\cdot q_1\in\mathscr{P}_N(f)$ and $2\cdot q_2\in\mathscr{P}_N(R_u f)$. Hence, by the monotonicity of the mean width and its linearity with respect to the Asplund sum, as well as its invariance under orthogonal transformations, we obtain
\begin{align*}
    w_0(q) &\leq w_0(q_1\star q_2)=\frac{1}{2}w_0(2\cdot q_1)+\frac{1}{2}w_0(2\cdot q_2)
    =\frac{1}{2}w_0(2\cdot q_1)+\frac{1}{2}w_0(R_u(2\cdot q_2)).
\end{align*}
Therefore, at least one of $w_0(2\cdot q_1)$ and $w_0(R_u(2\cdot q_2))$ must be greater than or equal to $w_0(q)$, so we define
\[
  \widetilde{q} :=
  \begin{cases}
     2\cdot q_1 & \text{if}\quad w_0(q_1)\geq w_0(q_2); \\
    R_u(2\cdot q_2) & \text{if} \quad w_0(q_1)< w_0(q_2).
  \end{cases}
\]
Since $q\in\mathscr{P}_N(\tau_u  f)$ was arbitrary, we deduce that
\[
G_N(\tau_u f)=\sup_{q\in\mathscr{P}_N(\tau_u  f)}w_0(q)\leq \sup_{\widetilde{q}\in\mathscr{P}_N(f)}w_0(\widetilde{q})=G_N(f).
\]
\end{proof}

\begin{lemma}\label{cont-lemma}
Let $n,N\in\mathbb{N}$ be such that $N\geq n+2$. Then $G_{N}:\LC_{\rm c}(\R^n)\to[0,\infty)$ is a lower semicontinuous functional of $f\in\LC_{\rm c}(\R^n)$ with respect to the topology of hypo-convergence.
\end{lemma}

\begin{proof}
Let $\{f_k=e^{-\psi_k}\}_{k\in\mathbb N}\subset \LC_{\rm c}(\mathbb R^n)$ be such that $f_k \stackrel{\hyp}{\longrightarrow}f$ as  $k\to\infty$, where $f=e^{-\psi}\in \LC_{\rm c}(\mathbb R^n)$. Since hypo-convergence of log-concave functions is equivalent to epi-convergence of their convex base functions, we have $\psi=\epilim_{k\in\mathbb{N}}\psi_k$. We want to prove that
\begin{equation}\label{lower-semi-cont-GN}
\liminf_{k\to\infty} G_N(f_k)\geq G_N(f).
\end{equation}

Fix $\varepsilon>0$. By the definition of $G_N(f)$, there exists $q=e^{-p}\in\mathscr{P}_N(f)$
such that
\[
w_0(q)>G_N(f)-\varepsilon.
\]
Since $q\in\mathscr{P}_N(f)$, the function $p$ is an inner linearization of $\psi$ with at most $N$ break points. Hence there exist an integer $m\leq N$ and points
\[
\mathbf{X}_m:=\{(x_1,y_1),\dots,(x_m,y_m)\}\subset \epi(\psi)
\]
such that $\epi(p)=\conv^{\uparrow}\mathbf{X}_m$. For each $i\in\{1,\dots,m\}$, since $\psi=\epilim_{k\in\mathbb{N}}\psi_k$, there exists a subsequence
$\{x_{i,k}\}_{k\in\mathbb N}\subset\mathbb R^n$ such that $x_{i,k}\to x_i$ and $\psi_k(x_{i,k})\to \psi(x_i)$. Set $c_i:=y_i-\psi(x_i)\geq 0$, and define $y_{i,k}:=\psi_k(x_{i,k})+c_i$. Then $(x_{i,k},y_{i,k})\in \epi(\psi_k)$ for every $k$, and
\[
(x_{i,k},y_{i,k})\longrightarrow (x_i,y_i)
\quad\text{as }k\to\infty.
\]

Now let
\[
\widetilde{\mathbf{X}}_k:=\{(x_{1,k},y_{1,k}),\dots,(x_{m,k},y_{m,k})\}\subset \epi(\psi_k),\quad k\in\mathbb{N},
\]
and let $p_k$ be the inner linearization of $\psi_k$ determined by $\widetilde{\mathbf{X}}_k$, i.e.,
\[
\epi(p_k)=\conv^{\uparrow} \widetilde{\mathbf{X}}_k,\quad k\in\mathbb{N}.
\]
Then $p_k\in\mathscr{C}_N(\psi_k)$, so $q_k:=e^{-p_k}\in\mathscr{P}_N(f_k)$.

We claim that 
\begin{equation}\label{eq:epi-lim-p-tbs}
p=\epilim_{k\in\mathbb{N}}p_k.
\end{equation}
To prove this, it is enough to show that the epigraphs $\epi(p_k)$ converge to $\epi(p)$ in the Painlev\'e--Kuratowski sense (see, e.g., \cite[Definition 7.1]{Rockafellar-Wets}). First, let $(x,t)\in \epi(p)=\conv^{\uparrow}\mathbf{X}_m$. Then there exist $\lambda_1,\dots,\lambda_m\geq 0$
with $\sum_{i=1}^m\lambda_i=1$ and some $s\geq 0$ such that
\[
x=\sum_{i=1}^m \lambda_i x_i,
\qquad
t=\sum_{i=1}^m \lambda_i y_i+s.
\]
For $k\in\mathbb{N}$, define
\[
x_k:=\sum_{i=1}^m \lambda_i x_{i,k},
\qquad
t_k:=\sum_{i=1}^m \lambda_i y_{i,k}+s.
\]
Then $(x_k,t_k)\in \epi(p_k)=\conv^{\uparrow} \widetilde{\mathbf{X}}_k$ for every $k$, and $(x_k,t_k)\to (x,t)$ as $k\to\infty$. 
Thus every point of $\epi(p)$ is a limit of points from $\epi(p_k)$.

Conversely, suppose that $(x_k,t_k)\in \epi(p_k)$ for every $k\in\mathbb{N}$, and that $(x_k,t_k)\to (x,t)$ as $k\to\infty$. 
Since $(x_k,t_k)\in \conv^{\uparrow}\widetilde{\mathbf{X}}_k$, for each $k$ there exist
$\lambda_{1,k},\dots,\lambda_{m,k}\geq 0$ with $\sum_{i=1}^m\lambda_{i,k}=1$ and some $s_k\geq 0$
such that
\[
x_k=\sum_{i=1}^m \lambda_{i,k}x_{i,k},
\qquad
t_k=\sum_{i=1}^m \lambda_{i,k}y_{i,k}+s_k.
\]
Passing to a subsequence, if necessary, by compactness of the simplex we may assume that $\lambda_{i,k}\to \lambda_i \geq 0$ for $i\in\{1,\ldots,m\}$ and $\sum_{i=1}^m\lambda_i=1$. Using $(x_{i,k},y_{i,k})\to (x_i,y_i)$ as $k\to\infty$, we obtain
\[
x=\sum_{i=1}^m \lambda_i x_i,
\qquad
t\geq \sum_{i=1}^m \lambda_i y_i.
\]
Hence $(x,t)\in \conv^{\uparrow}\mathbf{X}_m=\epi(p)$. Therefore, $\epi(p_k)\to \epi(p)$ as $k\to\infty$, which proves \eqref{eq:epi-lim-p-tbs}.

Since $q_k=e^{-p_k}$ and $q=e^{-p}$, this is equivalent to $q_k \stackrel{\hyp}{\longrightarrow}q$ as $k\to\infty$.
By Lemma~\ref{lem:walpha-lsc}, the mean width $w_0$ is lower semicontinuous with respect to hypo-convergence.
Therefore,
\[
\liminf_{k\to\infty} w_0(q_k)\geq w_0(q).
\]
Since $q_k\in\mathscr{P}_N(f_k)$, we also have
\[
G_N(f_k)\geq w_0(q_k)
\quad\text{for every }k\in\mathbb{N}.
\]
It follows that
\[
\liminf_{k\to\infty} G_N(f_k)
\geq
\liminf_{k\to\infty} w_0(q_k)
\geq
w_0(q)
>
G_N(f)-\varepsilon.
\]
Since $\varepsilon>0$ was arbitrary, we conclude that
\[
\liminf_{k\to\infty} G_N(f_k)\geq G_N(f).
\]
Thus $G_N$ is lower semicontinuous on $\LC_{\rm c}(\mathbb R^n)$ with respect to hypo-convergence.
\end{proof}

%\begin{proof}
 %   By Lemma \ref{MW-lower-semi-lem}, the function $f\mapsto w_0(f)$ is lower semicontinuous with respect to the topology of hypo-convergence in $\LC_{\rm c}(\R^n)$. The result now follows since the supremum of a family of lower semicontinuous functions is lower semicontinuous (see, e.g., \cite[Proposition 11.1]{GM-book}).
%\end{proof}

We are now ready to prove Theorem \ref{mainThm-new}.

\subsection{Proof of Theorem \ref{mainThm-new}}

We give the proof for $\alpha=0$; the case $\alpha\in(-2/(n-1),0)$ follows by the modifications described above. Let $f\in\LC_{\rm c}(\R^n)$. By Theorem \ref{convergence-hypo-symmetral}, there exists a sequence $\{u_i\}\subset\Sp$ such that $\bigcirc_{i=1}^m\tau_{u_i}  f\stackrel{\hyp}{\longrightarrow}f_{\rm sym}^{(0)}$  as $m\to\infty$. Applying Lemma \ref{mainLemma} recursively $m$ times, we obtain
\[
G_{N}(f)\geq G_{N}(\bigcirc_{i=1}^m\tau_{u_i}  f),\qquad m=1,2,\ldots.
\]
Taking lower limits as $m\to\infty$ and applying Lemma \ref{cont-lemma}, we obtain
\[
G_{N}(f) \geq \liminf_{m\to\infty}G_{N}(\bigcirc_{i=1}^m\tau_{u_i}  f)%=G_{N}(\lim_{m\to\infty}\bigcirc_{i=1}^m\tau_{u_i}  f)
\geq G_{N}(f_{\rm sym}^{(0)}).
\]
Since $f\in\LC_{\rm c}(\R^n)$ was arbitrary, the result follows. \qed

%\begin{remark}
 %  In view of Remark \ref{rmk:alpha-remark-inner-alpha-linearization}, we note that the proof of Theorem \ref{mainThm-new} can be modified to extend the result from $\alpha=0$ to all $\alpha\in(-\tfrac{2}{n-1},0]$. 
%\end{remark}

%%%%%%%%%%%%%%%%%%%%%%%%%%%%%%%%%%%%
\subsection{A general inequality}\label{extremal-section}

As the proof of Theorem \ref{mainThm-new}  reveals, a more general extremal property of the reflectional hypo-symmetrization holds true.

\begin{theorem}\label{new-thm}
Let $\alpha\in(-\infty,0]$. Suppose that $F: \mathcal{C}_\alpha(\R^n) \to[0,\infty)$  
is upper semicontinuous with respect to hypo-convergence,  invariant under orthogonal transformations, and is increasing with respect to Minkowski symmetrizations, that is,
\begin{equation}\label{BFL-fn}
\forall f\in\mathcal{C}_\alpha(\R^n),\quad F(f)\leq F(\tau_u^\alpha f).
%\forall f,g \in \mathcal{C}_\alpha(\R^n), \,\forall\lambda\in(0,1), \, F(\lambda\cdot_\alpha f\star_\alpha(1-\lambda)\cdot_\alpha g) \geq M_{\gamma}^{(\lambda,1-\lambda)}(F(f),F(g)).%\lambda F(f)^\frac{\alpha}{n+\alpha} +(1-\lambda)F(g)^{\frac{\alpha}{n+\alpha}}.
\end{equation}
Then for every $f\in\mathcal{C}_\alpha(\R^n)$, we have 
\begin{equation}\label{new1}
F(f) \leq F(f_{\rm sym}^{(\alpha)}).
\end{equation}
\end{theorem}

\begin{proof}
Let $f\in\mathcal{C}_\alpha(\R^n)$.  By hypothesis,  for any Minkowski symmetrization we have $F(\tau_u^\alpha f) \geq F(f)$.  By Theorem \ref{convergence-hypo-symmetral}, there exists a sequence of Minkowski symmetrizations such that  $\bigcirc_{i=1}^m\tau_{u_i}^\alpha  f\stackrel{\hyp}{\longrightarrow} f_{\rm sym}^{(\alpha)}$  as $m\to\infty$. Therefore, by the previous inequality and the upper semicontinuity of $F$, we obtain
\[
F(f) \leq \limsup_{m\to\infty}F\left(\bigcirc_{i=1}^m \tau_{u_i}^\alpha  f\right)\leq  F(f_{\rm sym}^{(\alpha)}).
\]
\end{proof}

\begin{remark}\label{log-concave-remark}
    The analogous statement of Theorem \ref{new-thm} holds true for lower semicontinuous functions $F$ which are decreasing with respect to Minkowski symmetrizations (meaning the inequality in \eqref{BFL-fn} reverses); in this case, the inequality \eqref{new1} reverses.
\end{remark}

\begin{remark}
    One may use the results in Subsection \ref{extremal-hypo-sec} to show that Theorem \ref{mainThm-new} is a special case of Theorem \ref{new-thm}.
\end{remark}

\subsection{A functional Urysohn-type inequality}\label{sec:urysohn}

As a corollary of Theorem \ref{new-thm}, we obtain a functional  Urysohn-type inequality. 

\begin{corollary}\label{urysohn-cor}
    For every $f\in\LC_{\rm c}(\R^n)$, we have $J(f)\leq J(f_{\rm sym}^{(0)})$. If $J(f)>0$, then equality holds if and only if $f(x)=g(x-a)$ for some radial function $g\in\LC_{\rm c}(\R^n)$ and some $a\in\R^n$.
\end{corollary}

Let us first establish the inequality. After that, we will prove the equality conditions.
\begin{proof}[Proof of the inequality]
    The function $f\mapsto J(f)$ is continuous with respect to the topology of hypo-convergence on $\LC_{\rm c}(\R^n)$ (see \cite[Lemma 16]{CLM-IMRN}). The total mass $J(\cdot)$ is  invariant under orthogonal transformations, and by Proposition \ref{mainProp}(vii), we have $J(\tau_u^0 f)\geq J(f)$.
    Thus,  the desired inequality  follows from Theorem \ref{new-thm} with $\alpha=0$ and $F=J$.
\end{proof}

We will establish the equality conditions through a few lemmas. The next lemma implies the sufficiency of the condition $f(x)=g(x-a)$.

\begin{lemma}\label{lem:sufficiency-lemma}
    Let $f\in\LC_{\rm c}(\R^n)$ and suppose that $f(x)=g(x-a)$ for some radial function $g\in\LC_{\rm c}(\R^n)$ and $a\in\R^n$. Then $g=f_{\rm sym}^{(0)}$. In particular, $J(f)=J(f_{\rm sym}^{(0)})$.
\end{lemma}

\begin{proof}
    Since $g\in\LC_{\rm c}(\R^n)$, there exists $\varphi\in\mathrm{Conv}_{\rm c}(\R^n)$ such that $g=e^{-\varphi}$. Since $g$ is radial, $\varphi$ is also radial, i.e., $\varphi(\rho x)=\varphi(x)$ for all $\rho\in\mathrm{O}(n)$ and all $x\in\R^n$. Set $\psi(x):=\varphi(x-a)$, so that $f(x)=e^{-\psi(x)}$. Recall  $f_{\rm sym}^{(0)}=e^{-\psi_{\rm sym}}$, where, by Theorem \ref{thm:full-sequence-minkowski}, $\psi_{\rm sym}$ is the unique radial function satisfying
    \[
\mathcal{L}\psi_{\rm sym}(x)=\int_{\mathrm{O}(n)}\mathcal{L}\psi(\vartheta^{-1}x)\,d\eta(\vartheta),\quad x\in\R^n.
    \]
    Thus it is enough to show that $\mathcal{L}\varphi=\mathcal{L}\psi_{\rm sym}$. For every $x\in\R^n$, we have
    \[
\mathcal{L}\psi(x)=\sup_{y\in\R^n}(\langle x,y\rangle-\psi(y))=\sup_{y\in\R^n}(\langle x,y\rangle-\varphi(y-a)).
    \]
    Set $z=y-a$, so that $y=z+a$. Then
    \[
\mathcal{L}\psi(x)=\sup_{z\in\R^n}(\langle x,z+a\rangle-\varphi(z))=\langle x,a\rangle+\sup_{z\in\R^n}(\langle x,z\rangle-\varphi(z))=\langle x,a\rangle+\mathcal{L}\varphi(x).
    \]

    Fix $\rho\in\mathrm{O}(n)$. Then
    \[
\mathcal{L}\varphi(\rho x)=\sup_{y\in\R^n}(\langle\rho x,y\rangle-\varphi(y))=\sup_{y\in\R^n}(\langle x,\rho^{-1}y\rangle-\varphi(y)).
    \]
    Set $z=\rho^{-1}y$, so that $y=\rho z$. Since $\varphi$ is radial, $\varphi(y)=\varphi(\rho z)=\varphi(z)$. Hence
    \[
\mathcal{L}\varphi(\rho x)=\sup_{z\in\R^n}(\langle x,z\rangle-\varphi(z))=\mathcal{L}\varphi(x).
    \]
    Thus, $\mathcal{L}\varphi$ is radial. Therefore,
    \begin{align*}
        \mathcal{L}\psi_{\rm sym}(x)&=\int_{\mathrm{O}(n)}\mathcal{L}\psi(\vartheta^{-1}x)\,d\eta(\vartheta)=\int_{\mathrm{O}(n)}\left(\langle\vartheta^{-1}x,a\rangle+\mathcal{L}\varphi(\vartheta^{-1}x)\right)d\eta(\vartheta)\\
        &=\int_{\mathrm{O}(n)}\langle\vartheta^{-1}x,a\rangle\,d\eta(\vartheta)+\mathcal{L}\varphi(x).
    \end{align*}

    It remains to show that 
    \begin{equation}\label{eq:equals-zero-tbs}
    \int_{\mathrm{O}(n)}\langle\vartheta^{-1}x,a\rangle\,d\eta(\vartheta)=0.
    \end{equation}
    Define $m:=\int_{\mathrm{O}(n)}\vartheta a\,d\eta(\vartheta)\in\R^n$. Then
    \[
\int_{\mathrm{O}(n)}\langle\vartheta^{-1}x,a\rangle\,d\eta(\vartheta)=\int_{\mathrm{O}(n)}\langle x,\vartheta a\rangle\,d\eta(\vartheta)=\langle x,m\rangle.
    \]
    Moreover, for every $\rho\in\mathrm{O}(n)$, the vector $m$ is invariant under $\rho$:
    \[
\rho m=\int_{\mathrm{O}(n)}\rho(\vartheta a)\,d\eta(\vartheta)=\int_{\mathrm{O}(n)}(\rho\vartheta)a\,d\eta(\vartheta)=\int_{\mathrm{O}(n)}\sigma a\,d\eta(\sigma)=m.
    \]
    The only vector fixed by every orthogonal transformation is $o$, so $m=o$. This proves \eqref{eq:equals-zero-tbs}. Consequently, $\mathcal{L}\psi_{\rm sym}(x)=\mathcal{L}\varphi(x)$ for all $x\in\R^n$, i.e., $\mathcal{L}\psi_{\rm sym}=\mathcal{L}\varphi$. Since both $\psi_{\rm sym}$ and $\varphi$ are proper and lower semicontinuous, applying the Legendre--Fenchel transform again we get $\varphi=\psi_{\rm sym}$. Therefore, $g=e^{-\varphi}=e^{-\psi_{\rm sym}}=f_{\rm sym}^{(0)}$. Finally, by the translation-invariance of the Lebesgue measure,
    \[
J(f)=\int_{\R^n}g(x-a)\,dx=\int_{\R^n}g(x)\,dx=\int_{\R^n}f_{\rm sym}^{(0)}(x)\,dx=J(f_{\rm sym}^{(0)}).
    \]
\end{proof}

\subsubsection{Technical lemmas}

\begin{lemma}\label{period-lemma}
Let $f:\R^n\to[0,\infty)$ be measurable and suppose $0<J(f)<\infty$. If there exists $p\in\R^n$ such that $f(x+p)=f(x)$ for all $x\in\R^n$, then $p=o$.
\end{lemma}

\begin{proof}
Suppose by way of contradiction that $p\neq o$. Since $J(f)>0$, there exists $\varepsilon>0$ such that the superlevel set $E:=\lev_{\geq \varepsilon}f$ has positive Lebesgue measure. Indeed, if $\vol_n(E)=0$ for every $\varepsilon>0$, then
$f=0$ almost everywhere, contradicting $J(f)>0$. Choose $R>0$ such that $\vol_n(E\cap B(o,R))>0$, and set $A:=E\cap B(o,R)$. Then $A$ is measurable, bounded, and $0<\vol_n(A)<\infty$.  Moreover, $f(x)\geq \varepsilon$ for every $x\in A$. 
Since $f$ is $p$-periodic, for every $k\in\mathbb{Z}$ we have $f(x+kp)=f(x)$ for all $x\in\R^n$. Hence $f(x)\geq \varepsilon$ for every $x\in A+kp$.

Since $A$ is bounded and $p\neq o$, we can choose an integer $m\geq 1$ so large that the sets $A, A+mp, A+2mp,\ldots$ are pairwise disjoint. Indeed, since $A\subset B(o,R)$, each translate $A+jmp$ is contained in $B(jmp,R)$, and the balls $B(jmp,R)$ are pairwise disjoint if $m|p|>2R$. Therefore, for every $N\in\mathbb{N}$,
\[
\int_{\R^n} f(x)\,dx
\geq
\sum_{j=0}^N \int_{A+jmp} f(x)\,dx
\geq 
\sum_{j=0}^N \varepsilon\vol_n(A+jmp)
=(N+1)\varepsilon\vol_n(A).
\]
Letting $N\to\infty$, we obtain $J(f)=\infty$, contradicting the assumption that $J(f)<\infty$. Thus $p=o$.
\end{proof}

The final lemma asserts that for a metrizable topological group, invariance under a dense subgroup implies invariance under its closure. 

\begin{lemma}\label{dense-lemma}
    Let $G$ be a metrizable topological group acting continuously on a topological space $X$, and let $H$ be a dense subgroup of $G$. Let $f:X\to\R$ be upper semicontinuous. If $f(hx)=f(x)$ for all $h\in H$ and all $x\in X$, then $f(gx)=f(x)$ for all $g\in G$ and all $x\in X$. 
\end{lemma}

\begin{proof}
    Fix $g\in G$ and $x\in X$. Since $H$ is dense in $G$, we can choose a sequence $\{h_i\}\subset H$ such that $h_i\to g$. By continuity of the action, we have $h_i x\to gx$ as $i\to\infty$. Since $f$ is upper semicontinuous,
    \[
f(gx) \geq \limsup_{i\to\infty}f(h_i x).
    \]
    But $f$ is $H$-invariant, so $f(h_i x)=f(x)$ for every $i$. Hence $f(gx)\geq f(x)$. 
    
    Now apply the same argument to $g^{-1}$. Choose a sequence $\{k_i\}\subset H$ with $k_i\to g^{-1}$ as $i\to\infty$. Then $k_i(gx)\to g^{-1}(gx)=x$ as $i\to\infty$, so by the upper semicontinuity of $f$ we obtain
    \[
f(x)\geq\limsup_{i\to\infty}f(k_i(gx)).
    \]
    Again, since $f$ is $H$-invariant we have $f(k_i(gx))=f(gx)$ for every $i$. Thus $f(x)\geq f(gx)$. Combining the previous two inequalities, we get $f(gx)=f(x)$. This proves the lemma.
\end{proof}

It remains to prove the necessity of $f(x)=g(x-a)$. In this direction, we first characterize the equality conditions of a single symmetrization.

\begin{lemma}\label{lem:one-step-equality}
    Let $f\in\LC_{\rm c}(\R^n)$ with $J(f)>0$, and let $u\in\Sp$. If $R_u f$ is a translate of $f$, then there exist $t\in\R$ and $g\in\LC_{\rm c}(\R^n)$ with $g=g\circ R_u$ such that
\[
f(x)=g(x-tu),\qquad x\in\R^n.
\]
Conversely, if such a representation holds, then $R_u f$ is a translate of $f$ by a vector parallel to $u$.
\end{lemma}

\begin{proof}
   For $z\in\R^n$, denote the translation of $f$ by $z$ by $T_z f(x):=f(x-z)$. Assume first that $R_u f=T_z f$ for some $z\in\R^n$. Applying $R_u$ again, we get
   \[
   R_u(T_z f)=T_{R_u(z)}R_u f.
   \]
   Indeed, both sides evaluated at $x$ are equal to $f(R_u x-z)$. Hence
   \begin{align*}
       f &=R_u(R_u f)=R_u(T_z f)= T_{R_u(z)}R_u f=T_{R_u(z)}T_z f=T_{z+R_u(z)}f.
   \end{align*}
   Thus $f$ has period $p:=z+R_u(z)$. Since a nonzero integrable function cannot have a nonzero period by Lemma \ref{period-lemma}, we have $p=z+R_u(z)=o$. Therefore $z$ is parallel to $u$, so $z=2tu$ for some $t\in\R$. Hence $R_u f=T_{2tu}f$. Define $g:=T_{tu}f$. Then
   \begin{align*}
R_u f(x+tu)&=f(R_u(x+tu))=f(R_u x-tu),\\
T_{2tu}f(x+tu)&=f(x-tu).
   \end{align*}
   Using $R_u f=T_{2tu}f$ evaluated at $x+tu$, we obtain
   \[
g(R_u x)=f(R_u x-tu)=f(x-tu)=g(x),
   \]
   so $g$ is symmetric with respect to $u^\perp$, and $f=T_{-tu}g$. Renaming $-t$ as $t$ gives $f(x)=g(x-tu)$.

   Conversely, suppose that $f(x)=g(x-tu)$ for some $t\in\R$ and some $g\in\LC_{\rm c}(\R^n)$ satisfying $g=g\circ R_u$. Then, for every $x\in\R^n$,
   \[
R_u f(x)=f(R_u x)=g(R_u x-tu)=g(R_u(x+tu))=g(x+tu)=f(x+2tu).
   \]
   Thus $R_u f=T_{-2tu}f$, which is a translate of $f$ by a vector parallel to $u$.
\end{proof}

To complete the proof, we use the following lemma.
\begin{lemma}\label{lem:necessity-2}
    Let $f\in\LC_{\rm c}(\R^n)$ with $J(f)>0$. If $J(f)=J(f_{\rm sym}^{(0)})$, then $f$ is a translation of a radial function $g\in\LC_{\rm c}(\R^n)$.
\end{lemma}

\begin{proof}
    By Theorem \ref{convergence-hypo-symmetral}, there exists a sequence of Minkowski symmetrizations such that  $\bigcirc_{i=1}^m\tau_{u_i}^0  f\stackrel{\hyp}{\longrightarrow} f_{\rm sym}^{(0)}$  as $m\to\infty$. By the monotonicity of the total mass under the Minkowski symmetrization (see Proposition \ref{mainProp}(vii)), this implies
    \[
J(f) \leq J(\tau_{u_1}^0 f)\leq J(\tau_{u_2}^0\tau_{u_1}^0 f)\leq\ldots\leq J(f_{\rm sym}^{(0)}).
    \]
    Since $J(f)=J(f_{\rm sym}^{(0)})$, this implies that 
    \[
J(f)=J(\tau_{u_1}^0 f)=J(\tau_{u_2}^0\tau_{u_1}^0 f)=\ldots=J(f_{\rm sym}^{(0)}).
    \]
    Thus, for every direction $u_i$ in the convergent sequence, equality holds for the corresponding one-step symmetrization of the current iterate. By Proposition \ref{mainProp}(vii), the reflection of that current iterate across $u_i^\perp$ is a translate of the current iterate. Lemma \ref{lem:one-step-equality} then shows that the current iterate is a translate of a function symmetric with respect to $u_i^\perp$. By Theorem \ref{thm:full-sequence-minkowski}, the chosen reflections generate a dense subgroup of $\mathrm{O}(n)$. By Lemma \ref{dense-lemma}, this implies that $f$ is a translation of a radial log-concave function. This completes the proof.
\end{proof}

%%%%%%%%%%%%%%%%%%%%%%%%

\subsubsection{Recovering the classical Urysohn inequality}

To recover the classical Urysohn inequality, we  will need one more lemma. Let $K_w^*:=\left(\frac{w(K)}{w(B_n)}\right)B_n$ denote the $n$-dimensional Euclidean ball centered at the origin that has the same mean width as $K$.

\begin{lemma}\label{indicators-lemma}
For every $K\in\mathcal{K}^n$, we have $(\mathbbm{1}_K)_{\rm sym}=\mathbbm{1}_{K_w^*}$, or equivalently, $(I_K^\infty)_{\rm sym}=I_{K_w^*}^\infty$.
\end{lemma}

\begin{proof}
By the classical convergence theorem for Minkowski symmetrizations of convex bodies, there exists a sequence of hyperplanes $H_i\in\Gr(n,n-1)$ such that 
\[
K_m:=\bigcirc_{i=1}^m \tau_{H_i}K\longrightarrow K_w^*
\]
in the Hausdorff metric as $m\to\infty$. By Remark \ref{rmk:indicator-symm}, for every $i$ we have $\tau_{H_i}\mathbbm{1}_K=\mathbbm{1}_{\tau_{H_i}K}$. Hence, inductively, for each $m\in\mathbb{N}$ we obtain $\bigcirc_{i=1}^m\tau_{H_i}\mathbbm{1}_K=\mathbbm{1}_{K_m}$. Therefore,
\[
\hyp\left(\bigcirc_{i=1}^m\tau_{H_i}\mathbbm{1}_K\right)=K_m\times[0,1].
\]
Since $K_m\to K_w^*$ in the Hausdorff metric as $m\to\infty$, we also have $K_m\times[0,1]\to K_w^*\times[0,1]$ in the Hausdorff metric as $m\to\infty$. Equivalently, the sequence $\mathbbm{1}_{K_m}$ hypo-converges to $\mathbbm{1}_{K_w^*}$ as $m\to\infty$. Thus, by Theorem \ref{thm:full-sequence-minkowski} we have $(\mathbbm{1}_K)_{\rm sym}^{(0)}=\mathbbm{1}_{K_w^*}$. This is equivalent to $(I_K^\infty)_{\rm sym}=I_{K_w^*}^\infty$ since
\[
e^{-I_{K_w^*}^\infty}=\mathbbm{1}_{K_w^*}=(\mathbbm{1}_K)_{\rm sym}^{(0)}=e^{-(I_K^\infty)_{\rm sym}}.
\]
\end{proof}

\begin{remark}
 Finally, let us explain how to recover the classical Urysohn inequality from the preceding  results. Choose $f=\mathbbm{1}_K$, where $K$ is a convex body in $\R^n$. Applying Corollary \ref{urysohn-cor} and Lemma \ref{indicators-lemma}, we get
\[
\vol_n(K) = J(\mathbbm{1}_K) \leq J((\mathbbm{1}_K)_{\rm sym}^{(0)})=J(\mathbbm{1}_{K_w^*}) = \vol_n(K_w^*)=\left(\frac{w(K)}{w(B_n)}\right)^n\vol_n(B_n).
\]
Equality holds if and only if $\mathbbm{1}_K(x)=g(x-a)$ for all $x\in\R^n$, where $g\in\LC_{\rm c}(\R^n)$ is radial and $a\in\R^n$. Thus $g(x)=\mathbbm{1}_{K-a}(x)$ is radial, i.e., $\mathbbm{1}_{K-a}(\rho x)=\mathbbm{1}_{K-a}(x)$ for all $\rho\in\mathrm{O}(n)$ and all $x\in\R^n$. This means that $\mathbbm{1}_{\rho^{-1}(K-a)}(x)=\mathbbm{1}_{K-a}(x)$ for all $\rho\in\mathrm{O}(n)$ and all $x\in\R^n$; this is equivalent to $\rho^{-1}(K-a)=K-a$ for all $\rho\in\mathrm{O}(n)$, i.e., $K=RB_n+a$ for some $R>0$.
\end{remark}

%%%%%%%%%%%%%%%%%%%%%%%%%%%%%%%%%%%%%

\section*{Acknowledgments}

I would like to thank the anonymous referees for their careful reading and remarks which helped improve this paper. I would also like to thank Andrea Colesanti, Fabian Mussnig, and Liran Rotem for the discussions on the topic of this paper.

%%%%%%%%%%%%%%%%%%%%%%%%%%%%%%%%%
\bibliographystyle{plain}
\bibliography{main}

%%%%%%%%%%%%%%

\vspace{3mm}

\noindent {\sc Department of Mathematics \& Computer Science, Longwood University \\201 High St., Farmville, Virginia 23901}

\vspace{1mm}

\noindent {\it E-mail address:} {\tt hoehnersd@longwood.edu}

\end{document}